\documentclass{amsproc}

\usepackage{amsmath,amsthm,amsfonts,amscd,amsxtra,amsopn,amssymb,verbatim,pdfsync}
\usepackage{enumerate} 
\usepackage[mathscr]{eucal}
\usepackage{graphicx}
\usepackage{epstopdf}
\usepackage{psfrag}

\input xy
\xyoption{all}

\newcommand\A{\mathbb A}
\newcommand\C{\mathbb C}

\newcommand\E{\mathbb E}

\newcommand\M{\mathbb M}
\newcommand\N{\mathbb N}

\newcommand\R{\mathbb R}

\newcommand\Z{\mathbb Z}

\newcommand\cH{\mathcal H}

\newcommand\cN{\mathcal N}

\newcommand\cS{\mathcal S}

\newcommand\fh{\mathfrak h}

\newcommand\fM{\mathfrak M}

\newcommand{\ra}{\longrightarrow}

\newcommand{\into}{\hookrightarrow}

\def\doublearrow #1 #2
{\quad\raisebox{.1cm}{$\overset{#1}\ra$}
\hskip -.67cm
\raisebox{-.1cm}{$\underset {#2}\ra$}\quad}

\newcommand\wh{\widehat}

\newcommand\str{\operatorname{str}}
\newcommand\tr{\operatorname{tr}}

\newcommand\sdim{\operatorname{sdim}}

\newcommand\Aut{\operatorname{Aut}}

\newcommand\End{\operatorname{End}}

\newcommand\Diff{\operatorname{Diff}}

\newcommand{\id}{\operatorname{id}}

\newcommand\p{\partial}

\newcommand\im{\operatorname{im}}

\newcommand\Cl{C\ell} 

\newcommand\pt{\operatorname{{pt}}}
\newcommand\map{\operatorname{map}}

\newcommand\Iso{\operatorname{Iso}}
\newcommand\ev{\opn{ev}}

\newcommand{\one}{{1\!\!1}}
\newcommand{\ie}{{i.e.\;}}
\newcommand{\eg}{{e.g.\;}}

\newcommand\TMF{\opn{TMF}}

\newcommand\bra[1]{\langle #1\rangle}

\newcommand\Hom{\operatorname{Hom}}

\newcommand\Iff{if and only if }

\newcommand{\nb}{\nobreakdash}

\newcommand\scr{\mathscr}
\renewcommand\={\overset{\text{def}}=}
\newcommand{\opn}[1]{\operatorname{#1}}

\newcommand{\sA}{\text{\sf{A}}}
\newcommand{\sB}{\text{\sf{B}}}
\newcommand{\sC}{\text{\sf{C}}}
\newcommand{\sD}{\text{\sf{D}}}

\newcommand{\sS}{\text{\sf{S}}}

\newcommand\Cat{\text{\sf Cat}}

\newcommand\Man{\text{\sf Man}}

\newcommand{\Sym}{\text{\sf{Sym}}}
\newcommand\SymCat{\text{\sf SymCat}}

\newcommand\TA{\text{\sf TA}}
\newcommand\TV{\text{\sf TV}}
\newcommand{\Vect}{\text{\sf{Vect}}}
\newcommand{\SVect}{\text{\sf{SVect}}}

\newcommand{\Bord} [1]{#1\text{-\hskip -.01in{\sf Bord}}}
\newcommand{\EB} [1]{#1\text{-\hskip -.01in{\sf EB}}}
\newcommand{\RB} [1]{#1\text{-\hskip -.01in{\sf RB}}}

\newcommand{\EFT} [1]{#1\text{-\hskip -.01in}\opn{EFT}}
\newcommand{\TFT} [1]{#1 \text{-\hskip -.04in} \opn{TFT}}

\newtheorem{prop}{Proposition}[section]
\newtheorem{theorem}[prop]{Theorem}
\newtheorem{thm}[prop]{Theorem}
\newtheorem{cor}[prop]{Corollary}
\newtheorem{lem}[prop]{Lemma}

\newtheorem{ddefn}[prop]{Definition}
\newtheorem{eex}[prop]{Example}
\newtheorem{rrem}[prop]{Remark}
\newtheorem{eexercise}[prop]{Exercise}
\newtheorem{con}[prop]{Conjecture}
\newtheorem{hhome}[prop]{Homework}
\newtheorem{nnumber}[prop]{}

\newenvironment{defn}{\begin{ddefn}\rm}{\end{ddefn}}
\newenvironment{ex}{\begin{eex}\rm}{\end{eex}}
\newenvironment{rem}{\begin{rrem}\rm}{\end{rrem}}

\newenvironment{conj}{\begin{con}\rm}{\end{con}}

{\catcode`@=11\global\let\c@equation=\c@prop}

\long\def\comments #1\endcomments{
{\bf #1}\newline}
\newcommand\csM{\text{\sf csM}}
\newcommand\SymGrp{\text{\sf SymGrp}}

\newcommand{\CEFT} [1]{#1\text{-\hskip -.01in}\opn{CEFT}}
\renewcommand{\RB} [1]{#1\text{-\hskip -.01in{\sf RBord}}}
\renewcommand{\EB} [1]{#1\text{-\hskip -.01in{\sf EBord}}}
\newcommand{\bEB} [1]{#1\text{-\hskip -.01in}\overline{\text{\sf EBord}}}
\newcommand{\bTV}{\overline{\TV}}
\newcommand{\CB} [1]{#1\text{-\hskip -.01in{\sf CBord}}}
\newcommand{\CEB} [1]{#1\text{-\hskip -.01in{\sf CEBord}}}
\newcommand{\GMMan}{(G,\M)\text{-\hskip -.01in{\sf Man}}}
\newcommand{\GMBord}{(G,\M)\text{-\hskip -.01in{\sf Bord}}}
\newcommand\fl{{\opn{f\ell}}}
\newcommand\doublebackslash{\backslash\hskip -.05in\backslash}
\newcommand{\sM}{\text{\sf{M}}}
\newcommand\Bun{\text{\sf{Bun}}}
\renewcommand{\one}{I}
\newcommand\red{\opn{red}}
\renewcommand{\ie}{{i.e., }}
\renewcommand{\eg}{{e.g., }}

\title{Supersymmetric field theories and generalized cohomology}
\author{Stephan Stolz and Peter Teichner}
\thanks{Supported by grants from the National Science Foundation, the Max Planck Society and the Deutsche Forschungsgemeinschaft.}
\thanks{[2010] {Primary  55N20 Secondary 11F23 18D10 55N34 57R56 81T60}}

\address{Department of Mathemetics, University of Notre Dame, South Bend, Indiana, 46556}
\address{Department of Mathematics, University of California, Berkeley, CA, 94720 and Max Planck Institute for Mathematics, Bonn, Germany}
\begin{document}
\maketitle
\tableofcontents


%


\date{}


\maketitle

\section{Results and conjectures}\label{sec:results}

This paper is a survey of our mathematical notions of Euclidean field theories as models for (the cocycles in) a cohomology theory. This subject was pioneered by Graeme Segal \cite{Se1} who suggested more than two decades ago that a cohomology theory known as {\em elliptic cohomology} can be described in terms of $2$\nb-dimensional (conformal) field theories. Generally what we are looking for are isomorphisms of the form
\begin{equation}\label{eq:leitmotiv}
\left\{\text{supersymmetric field theories of degree~$n$ over $X$} \right\}/_{\text{concordance}} \cong h^n(X)
\end{equation}
where a field theory over a manifold $X$ can be thought of as a family of field theories parametrized by $X$, and  the  abelian groups $h^n(X), n\in\Z$, form some (generalized) cohomology theory.
Such an isomorphism would give geometric cocycles for these cohomology groups in terms of objects from physics, and it would allow us to use the computational power of algebraic topology to determine families of field theories up to concordance. 

To motivate our interest in isomorphisms of type \eqref{eq:leitmotiv}, we recall the well-known isomorphism
\begin{equation}\label{eq:fredholm}
\left\{\text{Fredholm bundles over $X$} \right\}/_{\text{concordance}} \cong K^0(X),
\end{equation}
where $K^0(X)$ is complex $K$\nb-theory. We showed in \cite{HoST} that the space of Euclidean field theories of dimension $1|1$ has the homotopy type of the space of Fredholm operators, making the connection to \eqref{eq:leitmotiv}. 

The isomorphism \eqref{eq:fredholm} is one of the pillars of index theory in the following sense. Let $\pi\colon M\to X$ be a fiber bundle whose fibers are $2k$\nb-manifolds. Let us assume that the tangent bundle along the fibers admits a spin$^c$\nb-structure. This assumption guarantees that we can construct the Dirac operator on each fiber and that these fit together to give a bundle of Fredholm operators over the base space $X$. Up to concordance, this family is determined by the element in $K^0(X)$ it corresponds to via the isomorphism \eqref{eq:fredholm}. The {\em Family Index Theorem} describes this element as the image of the unit $1\in K^0(M)$ under the (topological) {\em push-forward map} 
\begin{equation}\label{eq:pushforward} 
\pi_*\colon K^0(M)\ra K^{-2k}(X)\cong K^0(X).
\end{equation}
The construction of the map $\pi_*$ does not involve any analysis -- it is described in homotopy theoretic terms. 

There is also a physics interpretation of the push-forward map \eqref{eq:pushforward}.The Dirac operator on a Riemannian spin$^c$ manifold  $N$ determines a Euclidean field theory of dimension $1|1$ (physicists would call it ``supersymmetric quantum mechanics'' on $N$), which should be thought of as the  quantization of the classical system consisting of a superparticle moving in $N$. We can think of the bundle of Dirac operators associated to a fiber bundle of spin$^c$ manifolds $N\to M\to X$ as a $1|1$\nb-dimensional field theory over $X$. This allows the construction of a (physical) push-forward map $\pi^q_*$ as the fiberwise {\em quantization} of $\pi$. The Feynman-Kac formula implies that $\pi^q_*$ equals the analytic push-forward and hence the equality $\pi_* = \pi^q_*$ is equivalent to the family index theorem.  

Our proposed model for `elliptic cohomology' will be given by $2|1$\nb-dimensional Euclidean field theories. More precisely, we conjecture that these describe the (periodic version of the) universal theory of topological modular forms $\TMF^*$ introduced by Hopkins and Miller \cite{Ho}. There is a topological push-forward map $\pi_*:\TMF^0(M) \to \TMF^{-k}(X)$ if $\pi: M\to X$ is a fiber bundle of k-dimensional string manifolds. The above discussion then has a conjectural analogue which would lead to a family index theorem on loop spaces. The fiberwise quantization $\pi^q$ requires the existence of certain integrals over mapping spaces of surfaces to $M$. 

In the next section, we provide a very rough definition of our notion of field theory and describe our main results and conjectures. A detailed description of the rest of the paper can be found at the end of that section. This paper only contains ideas and outlines of proofs, details will appear elsewhere.

{\em Acknowledgements:} It is a pleasure to thank Dan Freed, Dmitri Pavlov, Ingo Runkel, Chris Schommer-Pries, Urs Schreiber and Ed Witten for valuable discussions. We also thank the National Science Foundation, the Deutsche Forschungsgemeinschaft and the Max Planck Institute for Mathematics in Bonn for their generous support. Most of the material of this paper was presented at the Max Planck Institute as a IMPRS-GRK summer school in 2009.

\subsection{Field theories}
More than two decades ago, Atiyah, Kontsevich and Segal proposed a definition of a field theory as a functor from a suitable bordism category to the category of topological vector spaces. Our notion of field theory is a refinement of theirs for which the definition is necessarily quite intricate because we have to add the precise notion of ``supersymmetry" and ``degree" (another refinement, namely  ``locality" still needs to be implemented). While about half of this paper is devoted to explaining the definition, this is not a complete account. Fortunately, for the description of our results and conjectures, only a cartoon picture of what we mean by a field theory is needed. 

Roughly speaking, a $d$\nb-dimensional (topological) field theory $E$ assigns to any closed smooth $(d-1)$\nb-manifold $Y$ a topological vector space $E(Y)$, and to a $d$\nb-dimensional bordism $\Sigma$ from $Y_0$ to $Y_1$ a continuous linear map $E(\Sigma)\colon E(Y_0)\to E(Y_1)$. There are four requirements:
\begin{enumerate}
\item If $\Sigma$, $\Sigma'$ are bordisms from $Y_0$ to $Y_1$ which are diffeomorphic relative boundary, then $E(\Sigma)=E(\Sigma')$. 
\item $E$ is compatible with composition; \ie if $\Sigma_1$ is a bordism from $Y_0$ to $Y_1$ and $\Sigma_2$ is a bordism from $Y_1$ to $Y_2$, and $\Sigma_2\cup_{Y_1} \Sigma_1$ is the bordism from $Y_0$ to $Y_2$ obtained by gluing $\Sigma_1$ and $\Sigma_2$ along $Y_1$, then
\[
E(\Sigma_2\cup_{Y_1} \Sigma_1)=E(\Sigma_2)\circ E(\Sigma_1).
\]
\item $E$ sends disjoint unions to tensor products, \ie $E(Y\amalg Y')= E(Y)\otimes E(Y')$ for the disjoint union $Y\amalg Y'$ of closed $(d-1)$\nb-manifolds $Y$, $Y'$, and 
\[
E(\Sigma\amalg\Sigma')=E(\Sigma)\otimes E(\Sigma')\colon E(Y_0)\otimes E(Y_0')\to E(Y_1)\otimes E(Y_1')
\]
if $\Sigma$ is a bordism from $Y_0$ to $Y_1$ and $\Sigma'$ is a bordism from $Y_0'$ to $Y_1'$. 
\item The vector space $E(Y)$ should depend smoothly on $Y$, and the linear map $E(\Sigma)$ should depend smoothly on $\Sigma$.
\end{enumerate}
The first two requirements can be rephrased by saying that $E$ is a functor from a suitable bordism category $\Bord{d}$ to the category $\TV$ of topological vector spaces. The objects of $\Bord{d}$ are closed $(d-1)$\nb-manifolds; morphisms are bordisms up to diffeomorphism relative boundary. The third condition amounts to saying that the functor $E\colon \Bord{d}\to\TV$ is a {\em symmetric monoidal functor} (with monoidal structure  given by disjoint union on $\Bord{d}$, and (projective) tensor product on $\TV$). Making the last requirement precise is more involved; roughly speaking it means that we need to replace the domain and range categories by their {\em family versions} whose objects are {\em families} of closed $(d-1)$\nb-manifolds (respectively topological vector spaces) parametrized by smooth manifolds (and similarly for morphisms). In technical terms, $\Bord{d}$ and $\TV$ are refined to become fibered categories over the Grothendieck site of smooth manifolds.

\begin{rem} The empty set is the monoidal unit with respect to disjoint union, and hence requirement (3) implies that $E(\emptyset)$ is a $1$\nb-dimensional vector space (here we think of $\emptyset$ as a closed manifold of dimension $d-1$). If $\Sigma$ is a closed $d$\nb-manifold, we can consider $\Sigma$ as a bordism from $\emptyset$ to itself, and hence 
\[
E(\Sigma)\in \Hom(E(\emptyset),E(\emptyset))=\C.
\]
More generally, if $\Sigma$ is a family of closed $d$\nb-manifolds parametrized by some manifold $S$ (\ie $\Sigma$ is a fiber bundle over $S$), then the requirement (4) implies that $E(\Sigma)$ is a smooth function on $S$.
\end{rem}

There are many possible variations of the above definition of field theory by equipping the closed manifolds $Y$ and the bordisms $\Sigma$ with additional structure. For example if the additional structure is a conformal structure, $E$ is called a {\em conformal} field theory; if the additional structure is a Riemannian metric, we will refer to $E$ as a {\em Riemannian} field theory. If no additional structure is involved, it is customary to call $E$ a  {\em topological} field theory, although it would seem better to call it a smooth field theory (and reserve the term `topological' for functors on the bordism category of topological manifolds). 

Our main interest in this paper is in {\em Euclidean field theories}, where the additional structure is a Euclidean structure, \ie a flat Riemannian metric. In physics, the word `Euclidean' is typically used to indicate a Riemannian metric as opposed to a Lorentzian metric without our flatness assumption on that metric, and so our terminology might be misleading (we will stick with it since `Euclidean structure' is a mathematical notion commonly used in rigid geometry, and the  alternative terminology `flat Riemannian field theory' has little appeal). 

An important invariant of a field theory $E$ is its {\em partition function} $Z_E$, obtained by evaluating $E$ on all closed $d$\nb-manifolds $\Sigma$ with the appropriate geometric structure, and thinking of $\Sigma\mapsto E(\Sigma)$ as a function on the moduli stack of closed $d$\nb-manifolds equipped with this structure, see Definition~\ref{def:partition2}. We will only be interested in the partition function of conformal or Euclidean $2$\nb-dimensional field theories restricted to surfaces of genus $1$. Hence the following low-brow definition:

\begin{defn}\label{def:partition} Let $E$ be a conformal or Euclidean field theory of dimension $2$. Then the {\em partition function of $E$} is the function
\[
Z_E\colon \fh\ra \C
\qquad
\tau\mapsto E(T_\tau)
\]
where $\fh$ is the upper half-plane $\{\tau\in\C\mid \im(\tau)>0\}$, and $T_\tau$  is the torus $T_\tau=\C/\Z\tau+\Z$ with the flat metric induced by the standard metric on the complex plane. 
\end{defn}

We note that for $A=
\left(\begin{smallmatrix}
a&b\\ c&d
\end{smallmatrix} \right)
\in SL_2(\Z)$ the torus $T_\tau$ is conformally equivalent (but generally not isometric) to $T_{\tau'}$ for $\tau'=\frac{a\tau+b}{c\tau+d}$. In particular, the partition function of a $2$\nb-dimensional conformal field theory is invariant under the $SL_2(\Z)$\nb-action (but generally not of a Euclidean field theory).  

\subsection{Field theories over a manifold}\label{subsec:FTover}

Another possible additional structure on bordisms $\Sigma$ is to equip them with smooth maps to a fixed manifold $X$. The resulting field theories are called {\em field theories over $X$}. We think of a field theory over $X$ as a family of field theories parametrized by $X$: If $E$ is a field theory over $X$, and $x$ is a point of $X$, we obtain a field theory $E_x$ by defining $E_x(Y):=E(Y\overset{c_x}\ra X)$, where $c_x$ is the constant map that sends every point of $Y$ to $x$, and similarly for bordisms $\Sigma$.

\begin{rem} \label{rem:partition} Let $\Sigma$ be a closed $d$\nb-manifold, $\map(\Sigma,X)$ the space of smooth maps and 
\[
\ev\colon \map(\Sigma,X)\times\Sigma\to X
\]
 the evaluation map. We think of the trivial fiber bundle $ \map(\Sigma,X)\times\Sigma\to  \map(\Sigma,X)$ and the evaluation map as a smooth family of $d$\nb-manifolds with maps to $X$, parametrized by the mapping space. In particular, if $E$ is a $d$\nb-dimensional field theory over $X$, we can evaluate $E$ on $(\map(\Sigma,X)\times\Sigma,\ev)$ to obtain a smooth function $Z_{E,\Sigma}\in C^\infty(\map(\Sigma,X))$ (this mapping space is not a finite dimensional manifold, but that is not a problem, since one can work with presheaves of manifolds). We note that $Z_{E,\Sigma}$ can be interpreted as part of the partition function of $E$.  It follows from requirement (1) that 
\begin{equation}\label{eq:partition}
Z_{E,\Sigma}\quad\text{belongs to}\quad C^\infty(\map(\Sigma,X))^{\Aut(\Sigma)},
\end{equation}
the functions invariant under the action of the automorphisms of $\Sigma$. Here $\Aut(\Sigma)$ is the diffeomorphism group of $\Sigma$ if $E$ is a topological field theory, and the group of structure preserving diffeomorphisms if $E$ is a field theory corresponding to some additional geometric structure.
\end{rem}

We observe that $d$\nb-dimensional topological field theories over $X$ are extremely familiar objects for $d=0,1$: If $E$ is a $0$\nb-dimensional field theory over $X$, we obtain the function $Z_{E,\pt}\in C^\infty(\map(\pt,X))=C^\infty(X)$ associated to the $0$\nb-manifold $\pt$ consisting of a single point. This construction gives an isomorphism between $0$\nb-dimensional topological field theories over $X$ and $C^\infty(X)$. A slightly stronger statement holds: the groupoid of $0$\nb-TFT's over $X$ is equivalent to the discrete groupoid with object set $C^\infty(X)$.

To a vector bundle with a connection $V\to X$ we can associate a $1$\nb-dimensional field theory $E_V$ over $X$ as follows. We think of a point $x\in X$ as a map $x\colon \pt\to X$ and decree that $E_V$ should associate to $x$ the fiber $V_x$; any object in the bordism category is isomorphic to a disjoint union of these, and hence the functor $E_V$ is determined on objects. If $\gamma\colon [a,b]\to X$ is a path in $X$ from $x=\gamma(a)$ to $y=\gamma(b)$, we think of $\gamma$ as a morphism from $x$ to $y$ in the bordism category and define $E_V(\gamma)\colon E_V(x)=V_x\to E_V(y)=V_y$ to be the parallel translation along the path $\gamma$. The proof of the following result  \cite{DST} is surprisingly subtle. Note that smooth paths give morphisms in our bordism category, but piecewise smooth paths do not. Therefore, one even needs to decide how to compose morphisms.

\begin{thm}[\cite{DST}]\label{thm:1TFT}
The groupoid $\TFT{1}(X)$ of $1$\nb-dimensional topological field theories over a manifold $X$ is equivalent to groupoid of finite-dimensional vector bundles over $X$ with connections.
\end{thm}

\subsection{Supersymmetric field theories}\label{subsec:susyFT1}

Another variant of field theories are {\em supersymmetric field theories of dimension $d|\delta$}, where $\delta$ is a non-negative integer. These are defined as above, but replacing $d$\nb-dimensional manifolds (respectively $(d-1)$\nb-manifolds) by supermanifolds of dimension $d|\delta$ (respectively $(d-1)|\delta$). The previous discussion is included since a supermanifold of dimension $d|0$ is just a manifold of dimension $d$. In order to formulate the right smoothness condition in the supersymmetric case, we need to work with families whose parameter spaces are allowed to be supermanifolds rather than just manifolds.

If $E$ is a $0|1$\nb-dimensional TFT over a manifold $X$ we can consider the function $Z_{E,\Sigma}$ of Remark \ref{rem:partition} for the $0|1$\nb-supermanifold $\Sigma=\R^{0|1}$. It turns out that the algebra of smooth functions on the supermanifold $\map(\R^{0|1},X)$ can be identified with $\Omega^*(X)$, the differential forms on $X$ \cite{HKST}, and that the subspace invariant under the $\Diff(\R^{0|1})$\nb-action is the space of closed $0$\nb-forms. This leads to the following result.

\begin{thm}[\cite{HKST}]\label{thm:0TFT}
The groupoid $\TFT{0|1}(X)$ of $0|1$\nb-dimensional topological field theories over a manifold $X$ is equivalent to the discrete groupoid with objects $\Omega_{cl}^0(X)$.
\end{thm}

The notion of Euclidean structures on manifolds can be generalized to supermanifolds, see \ref{subsec:superEuclidean}. In particular, we can talk about (supersymmetric) Euclidean field theories of dimension $d|\delta$ ($\delta>0$ means that we are talking about a supersymmetric theory, and hence the adjective `supersymmetric' is redundant). If $E$ is a $0|1$\nb-EFT over $X$ (Euclidean field theory of dimension $0|1$ over $X$), then as above we can consider the function $Z_{E,\R^{0|1}}\in C^\infty(\map(\R^{0|1},X))$. Now this function is invariant only under the subgroup $\Iso(\R^{0|1})\subset\Diff(\R^{0|1})$ consisting of the diffeomorphisms preserving the Euclidean structure on $\R^{0|1}$. We show in \cite{HKST} that the invariant subspace 
\[
C^\infty(\map(\R^{0|1},X)^{\Iso(\R^{0|1})}=\Omega^*(X)^{\Iso(\R^{0|1})}
\]
 is the space $\Omega_{cl}^{ev}(X)$ of closed, even-dimensional forms on $X$. This leads to 

\begin{thm}[\cite{HKST}] \label{thm:0EFT} The groupoid $\EFT{0|1}(X)$ of $0|1$\nb-dimensional Euclidean field theories over a manifold $X$ is equivalent to the discrete groupoid with objects $\Omega_{cl}^{ev}(X)$.
\end{thm}

We are intrigued by the fact that field theories over $X$ are quite versatile objects; depending on the dimension and the type of field theory  we get such diverse objects as smooth functions ($0$\nb-TFT's),  vector bundles with connections ($1$\nb-TFT's) or closed even-dimensional differential forms ($0|1$\nb-EFT's). Higher dimensional field theories over $X$ typically don't have interpretation as classical objects. For example, Dumitrescu has shown that a $\Z/2$\nb-graded vector bundle $V\to X$ equipped with a Quillen superconnection $\A$ leads to a $\EFT{1|1}$ $E_{V,\A}$ over $X$ \cite{Du}. However, not every $\EFT{1|1}$ over $X$ is isomorphic to one of these. 

There are {\em dimensional reduction} constructions that relate field theories of different dimensions. For example, there is a {\em dimensional reduction functor}
\begin{equation}\label{eq:dim_red}
\opn{red}\colon \EFT{1|1}(X)\ra \EFT{0|1}(X)
\end{equation}
In his thesis \cite{Ha}, Fei Han has shown that if $V\to X$ is a vector bundle with a connection $\nabla$, then the image of the Dumitrescu field theory $E_{V,\nabla}\in \EFT{1|1}(X)$ under this functor, interpreted as a closed differential form via Theorem \ref{thm:0EFT}, is the Chern character form of $(V,\nabla)$.
\subsection{Concordance classes of field theories}
Next we want to look at field theories over $X$ from a topological perspective. A smooth map $f\colon X\to Y$ induces a functor $\EB{d|\delta}(X)\to \EB{d|\delta}(Y)$ and by pre-composition this gives a functor
\[
f^*\colon \EFT{d|\delta}(Y)\ra \EFT{d|\delta}(X).
\]
In other words, field theories over manifolds are {\em contravariant} objects, like smooth functions, differential forms or vector bundles with connections. 

\begin{defn} Let $E_\pm$ be two field theories over $X$ of the same type and dimension. Then $E_+$ is {\em concordant} to $E_-$ if there is a field theory $E$ over $X\times \R$ such that the two bundles $i_\pm^*E$ are isomorphic to $\pi_\pm^*E_\pm$. Here $i_\pm\colon X \times (\pm 1, \pm \infty) \into X\times \R$ are the two inclusion maps and $\pi_\pm\colon X \times (\pm 1, \pm \infty) \to X$ are the projections. We will write $\EFT {d|\delta}[X]$ for the set of {\em concordance classes} of ${d|\delta}$\nb-dimensional EFT's over $X$.
\end{defn}

We remark that passing to concordance classes forgets `geometric information' while retaining `homotopical information'. More precisely, the functors from manifolds to sets, $X\mapsto \EFT {d|\delta}[X]$, are homotopy functors. For example, two vector bundles with connections are concordant \Iff the vector bundles are isomorphic; two closed differential forms are concordant \Iff they represent the same de Rham cohomology class. In particular, Theorems \ref{thm:0TFT}  and \ref{thm:0EFT} have the following consequence.

\begin{cor}[\cite{HKST}]\label{cor:dR_coh} $\TFT{0|1}[X]$ and $H^0_{dR}(X)$ are naturally isomorphic as rings. Similarly, there is an isomorphism between $\EFT{0|1}[X]$ and $H^{ev}_{dR}(X)$, the even-dimensional de Rham cohomology of $X$.
\end{cor}

\subsection{Field theories of non-zero degree}
Corollary \ref{cor:dR_coh} suggest the question whether there is a field theoretic description of $H^n_{dR}(X)$. The next theorem gives a positive answer to this question using the notion of `degree $n$ field theories' which will be discussed in Section~\ref{sec:twisted}. 

\begin{thm}[\cite{HKST}] Let $X$ be a smooth manifold. Then there are equivalences of groupoids computing $0|1$-field theories of degree~$n$ as follows:
\[
\TFT{0|1}^n(X)\cong \Omega_{cl}^{n}(X)
\qquad
\EFT{0|1}^n(X)\cong 
\begin{cases}\Omega^{ev}_{cl}(X)&n\text{ even}\\
\Omega^{odd}_{cl}(X)&n\text{ odd}
\end{cases}\\
\]
Moreover, there are isomorphisms of abelian groups
\begin{align*}
\TFT{0|1}^n[X]&\cong H_{dR}^n(X)\\
\EFT{0|1}^n[X]& \cong
\begin{cases}H^{ev}_{dR}(X)&n\text{ even}\\
H^{odd}_{dR}(X)&n\text{ odd}
\end{cases}
\end{align*}
These isomorphisms are compatible with the multiplicative structure on both sides (given by the tensor product of field theories on the left, and the cup product of cohomology classes on the right).
\end{thm}

We note that Fei Han's result mentioned at the end of section \ref{subsec:susyFT1} implies the commutativity of the diagram, interpreting the Chern-character as dimensional reduction:
\[
\xymatrix{
K^0(X)\ar[r]^<>(.5){ch}\ar[d]&H^{ev}_{dR}(X)\ar[d]^\cong\\
\EFT{1|1}[X]\ar[r]^<>(.5){\red}&\EFT{0|1}[X]
}
\]
Here $ch$ is the Chern character, $\red$ is the dimensional reduction functor \eqref{eq:dim_red}, and the left arrow is induced by mapping a vector bundle with connection to its Dumitrescu field theory. We believe that this arrow is an isomorphism because a very much related description of $K^0(X)$ via $1|1$-EFT's was given in \cite{HoST}.

\subsection{$\EFT{2|1}$'s and topological modular forms}

The definition of a partition function for a $2$\nb-dimensional Euclidean field theory (see Definition \ref{def:partition}) can be extended to Euclidean field theories of dimension $2|1$ (see Definition \ref{def:partition2}). For $E\in \EFT{2|1}$, the restriction $Z^{++}_E$ of its partition function to the non-bounding spin structure $++$ on the torus can be considered as a complex valued function on the upper half plane $\fh$, see Remark~\ref{rem:partition2}. As mentioned after Definition \ref{def:partition}, the partition function of a {\em conformal} $2$\nb-dimensional field theory is smooth (but not necessarily holomorphic) and invariant under the usual $SL_2(\Z)$\nb-action, while no modularity properties would be expected for {\em Euclidean} field theories. It turns out that for a $\EFT{2|1}$ $E$ the supersymmetry forces $Z^{++}_E$ to be holomorphic and $SL_2(\Z)$-invariant:

\begin{thm}[\cite{ST3},\cite{ST6}]\label{thm:modular} Let $E$ be a Euclidean field theory of dimension $2|1$. Then the function $Z^{++}_E\colon \fh\to \C$ is a holomorphic modular function with integral Fourier coefficients. Moreover, every such function arises as $Z^{++}_E$ from a $2|1$-Euclidean field theory $E$.
\end{thm}

We recall that a holomorphic modular function is a holomorphic function $f\colon \fh\to \C$ which is meromorphic at $\tau=i\infty$ and is $SL_2(\Z)$-invariant. A modular function is {\em integral} if the coefficients $a_k$ in its $q$\nb-expansion $f(\tau)=\sum_{k= -N}^\infty a_kq^k$, $q=e^{2\pi i\tau}$ are integers. The restriction on partition functions coming from the above theorem is quite strong since any (integral) holomorphic modular function is an (integral) polynomial in the $j$-function. 

 In the following discussion, we use $2|0$-dimensional Euclidean field theories of degree~$n$ which will be explained in the following section. These use the bordism category of Euclidean {\em spin} manifolds. 

\begin{thm}[\cite{ST6}]\label{thm:periodicity} There is a field theory $P\in \EFT{2|0}^{-48}$ with partition function $Z_P$ equal to the discriminant squared, which is a periodicity element in the sense that multiplication by $P$ gives an equivalence of groupoids
\[
\EFT{2|0}^n(X)\overset{\cong}\ra \EFT{2|0}^{n-48}(X)
\]
\end{thm}

\begin{conj} 
There is an isomorphism $\EFT{2|1}^n_{loc}[X]\cong \TMF^n(X)$ compatible with the multiplicative structure. Here $\TMF^*$ is the $24^2$ periodic cohomology theory of topological modular forms mentioned above. The periodicity class has modular form the 24th power of the discriminant. We expect that our 48 periodicity will turn into a $24^2$ periodicity after building in the locality. 
\end{conj}

$\EFT{2|1}^n_{loc}[X]$ are concordance classes of {\em local} (sometimes called {\em extended}) $2|1$\nb-dimensional Euclidean field theories over $X$. These are more elaborate objects than the field theories discussed so far. For $n=0$, they are $2$\nb-functors out of a bordism $2$\nb-category whose objects are $0|1$\nb-manifolds, whose morphisms are $1|1$\nb-dimensional bordisms, and whose $2$\nb-morphisms are $2|1$\nb-dimensional manifolds with corners (all of them furnished with Euclidean structures and maps to $X$). The need for working with {\em local} field theories in order to obtain cohomology theories is explained in our earlier paper \cite[Section 1.1]{ST1}: for non-local theories we can't expect exactness of the Mayer-Vietoris sequence. Regarding terminology from that paper:  an `(enriched) elliptic object over $X$' is now a `(local) $2$\nb-dimensional conformal  field theory over $X$'. 

While very general and beautiful results have been obtained in particular by Lurie on local versions of {\em topological} field theories \cite{Lu}, \cite{SP}, unfortunately even the {\em definition} of local Euclidean field theories has not been tied down.

\subsection{Gauged field theories and equivariant cohomology}
The definition of field theories over a manifold $X$ in \S \ref{subsec:FTover} can  be generalized if $X$ is equipped with a smooth action of a Lie group $G$. Let us equip bordisms $\Sigma$ with additional structure that consists of a triple $(P,f,\nabla)$, where $P\to \Sigma$ is a principal $G$\nb-bundle, $f\colon P\to X$ is a $G$\nb-equivariant map, and $\nabla$ is a connection on $P$, and similarly for $Y$'s. We will call the corresponding field theories {\em $G$\nb-gauged field theories over $X$}. If $G$ is the trivial group, this is just a field theory over $X$ in the previous sense. We think of a $G$-gauged field theory over $X$ as a $G$\nb-equivariant family of field theories parametrized by $X$. We write $\TFT{d|1}^n_G(X)$ for the groupoid of $G$\nb-gauged $d|1$\nb-dimensional topological field theories of degree $n$ over $X$, and $\TFT{d|1}^n_G[X]$ for the abelian group of their concordance classes. 

\begin{thm}[\cite{HSST}]
The group $\TFT{0|1}^n_G[X]$ is isomorphic to the equivariant de Rham cohomology group $H_{dR,G}^n(X)$.
\end{thm}

Like in the non-equivariant case, the proof of this result is based on identifying the groupoid $\TFT{0|1}^n_G(X)$, which we show is equivalent to the discrete groupoid whose objects consist of the $n$\nb-cocycles of the Weil model of equivariant de Rham cohomology \cite[(0.5),(0.7)]{GS}. 
We note that in the absence of a connection, the resulting TFT's give $G$-invariant closed differential forms and hence concordance classes lead to  less interesting groups. For example, if $G$ is compact and connected then one simply gets back ordinary de Rham cohomology $H_{dR}^n(X)$.

We believe that a different way of getting the Weil model would be to replace the target $X$ by the stack $\widehat X_G$ of $G$-bundles with connection and $G$-map to $X$ and then define $\TFT{0|1}^n(\widehat X_G)$ as for manifold targets (that definition extends from manifolds, aka.\ representable stacks, to all stacks). There is a geometric map  
\[
\TFT{0|1}^n_G(X) \ra \TFT{0|1}^n(\widehat X_G)
\]
which should be an isomorphism.

For $G=S^1$ we have also constructed an equivariant version of Euclidean field theories of dimension $0|1$ over $X$. 

\begin{thm}[\cite{HaST}] For an $S^1$\nb-manifold $X$,
the group $\EFT{0|1}^n_{S^1}[X]$ is isomorphic to the localized equivariant de Rham cohomology group $H_{dR,G}^{ev}(X)[u^{-1}]$ for even $n$, respectively  $H_{dR,G}^{odd}(X)[u^{-1}]$ for odd $n$. Here $u\in H^2_{dR}(BS^1)$ is a generator.
\end{thm}

For an $S^1$-gauged {\em Euclidean} field theory, we require that  the curvature of the principal $S^1$\nb-bundle $P\to \Sigma$  is equal to a $2$\nb-form canonically associated to the Euclidean structure on $\Sigma$. The point of this condition is that the Euclidean structure on $\Sigma$ and the connection on $P$ give a canonical Euclidean structure on the $1|1$\nb-manifold $P$.  
The result is a functor
\[
\EFT{1|1}(X)\ra \EFT{0|1}_{S^1}(LX)
\]
which is a generalization of the  `dimensional reduction functor' \eqref{eq:dim_red}.
 Passing to concordance classes we obtain a homomorphism
 \[
 K^0(X)\ra \EFT{1|1}[X]\ra \EFT{0|1}_{S^1}[LX]\cong H^{ev}_{dR,S^1}(LX])[u^{-1}]
 \]
 \begin{thm}[\cite{HaST}] If $V$ is a complex vector bundle over $X$, the image of $[V]\in K^0(X)$ under the above map is the Bismut-Chern character.
 \end{thm}
 
 This statement generalizes the field theoretic interpretation of the Chern character in terms of the reduction functor $\red$  \eqref{eq:dim_red} since the homomorphism 
 \[
 H^{ev}_{dR,S^1}(LX)[u^{-1}]\ra H^{ev}_{dR}(X)
 \]
 induced by the inclusion map $X=LX^{S^1}\into LX$ maps the Bismut-Chern character of $V$ to the Chern character of $V$.

\subsection{Comparison with our 2004 survey}
For the  readers who are familiar with our 2004 survey \cite{ST1}, we briefly summarize some advances that are achieved by the current paper. The main new ingredient is a precise definition of {\em supersymmetric} field theories. This requires to define the right notion of geometric supermanifolds which make up the relevant bordism categories. We decided to use  {\em rigid geometries} in the spirit of Felix Klein since these have a simple extension to odd direction. In supergeometry it is essential to work with families of objects that are parametrized by supermanifolds (the additional odd parameters). This forced us to work with the family versions (also known as fibered versions) of the bordism and vector space categories. In Segal's (and our) original notion of field theory, there was an unspecified requirement of `continuity' of the symmetric monoidal functor. It turns out that our family versions implement this requirement in the following spirit: A map is smooth if and only if it takes smooth functions to smooth functions. 

In the 2-dimensional case, the rigid geometry we use corresponds to flat Riemannian structures on surfaces. Following Thurston and others, we call this a {\em Euclidean} structure. The flatness has the effect that only closed surfaces of genus one arise in the bordism category and hence our Euclidean field theories contain much less information then, say, a conformal field theory. Again we think of this as an advantage for several reasons. One is simply the fact that it becomes much easier to construct examples of field theories by a generator and relation method as discussed in Section~\ref{subsec:generators}. Another reason is the conjectured relation to topological modular forms, where also only genus one information is used. The last reason is our desire to express the Witten genus \cite{Wi} of a closed Riemannian string manifold as the partition function of a field theory, the nonlinear $2|1$-dimensional $\Sigma$-model. It is well established in the physics community that such a field theory should exist and have modular partition function (a fact proven mathematically for the Witten genus by Don Zagier \cite{Za}).
However, it is also well known that this field theory can only be conformal if the Ricci curvature of the manifold vanishes. The question arises why a non-conformal field theory should have a modular (and in particular, holomorphic) partition function? One of the results in this paper is exactly this fact, proven precisely for our notion of $2|1$-dimensional Euclidean field theory. The holomorphicity is a consequence of the more intricate structure of the moduli stack of supertori. 

In the conformal world, surfaces can be glued together along their boundaries by diffeomorphism, as we did in \cite{ST1}. However, for most other geometries this is not possible any more and hence a precise notion of geometric bordism categories requires the introduction of collars. A precise way of doing this is one of the important, yet technical, contributions of this paper. We also decided to work with categories internal to a strict 2-category $\sA$ as our model for `weak 2-categories'. They are very flexible, allowing the introduction of fibered categories (needed for our family versions),  symmetric monoidal structures (modeling Pauli's exclusion principle) and flips (related to the spin-statistics theorem) by just changing the ambient 2-category $\sA$. We also need the isometries (2-morphisms of an internal category) of bordisms to define the right notion of twisted field theories in Section~\ref{sec:twisted}. In order to fully model {\em local} 2-dimensional twisted field theories, we'll have to choose certain target `weak 3-categories' in future papers. One possible model is introduced in the contribution \cite{DH} in the current volume.

The `adjunction transformations' in \cite[Def. 2.1]{ST1} are now completely replaced by allowing certain `thin bordisms', for example $L_0$ in Section~\ref{subsec:generators}. Because of the existence of these geometric 1-morphisms, any functor must preserve the adjunctions automatically. In addition, we don't consider the (anti)-involutions on the categories any more, partially because we want to allow non-oriented field theories and partially because our $2|1$-dimensional bordism category does not have a real structure any more.

Finally, we enlarged the target category of a field theory to allow general topological vector spaces (locally convex, complete Hausdorff) because the smoothness requirement for a field theory sometimes does not hold for Hilbert spaces, see Remark~\ref{rem:completion}. This has the additional advantage of being able to use the {\em projective} tensor product (leading to non-Hilbert spaces) for which inner products and evaluations are continuous operations (unlike for the Hilbert tensor product).

\subsection{Summary of Contents} \label{subsec:summary}

The next section leads up to the definition of field theories associated to rigid geometries in Definition \ref{def:GMFT}. This includes Euclidean field theories of dimension $d$, which are obtained by specializing the geometry to be the Euclidean geometry of dimension $d$. Along the way we present the necessary categorical background (on internal categories in \S  \ref{subsec:internal_cat}, categories with flip in \S \ref{subsec:cat_flip} and fibered categories in \S \ref{subsec:fibered}) as well as geometric background (the construction of the Riemannian bordism category in \S  \ref{subsec:Riem_bord} and the definition of families of rigid geometries in \S  \ref{subsec:rigid}). Section \ref{sec:2EFT} discusses $2$\nb-dimensional EFT's and their partition functions. The arguments presented in this section provide the first half of the outline of the proof of our modularity theorem \ref{thm:modular} for the partition functions of Euclidean field theories of dimension $2|1$. The outline of the proof is continued in \S  \ref{subsec:modular2} after some preliminaries on supermanifolds in \S  \ref{subsec:SMan}, on super Euclidean geometry in \S  \ref{subsec:superEuclidean}, and the definition of supersymmetric field theories associated to a supergeometry in \S  \ref{subsec:susyFT2}; specializing to the super Euclidean geometry, these are supersymmetric Euclidean field theories. In section \ref{sec:twisted} we define {\em twisted field theories}. This notion is quite general and includes Segal's weakly conformal field theories (see \S \ref{subsec:weakly_conformal}) as well as Euclidean field theories of degree $n$ (see \S \ref{subsec:degree}). The last section contains an outline of the proof of the periodicity theorem \ref{thm:periodicity}.

\section{Geometric field theories}
\subsection{Segal's definition of a conformal field theory}
In this section we start with Graeme Segal's definition of a $2$\nb-dimensional conformal field theory and elaborate suitably to obtain the definition of a $d$\nb-dimensional Euclidean field theory, and more generally, a field theory associated to every `rigid geometry' (see Definitions \ref{def:GMFT} and \ref{def:susyGMFT}). Segal has proposed an axiomatic description of $2$\nb-dimensional conformal field theories in a preprint that widely circulated for a decade and a half  (despite the ``do not copy" advice on the front) before it was published as \cite{Se2}. In the published version, Segal added a foreword/postscript commenting on developments since the original manuscript was written in which he proposes the following definition of conformal field theories.

\begin{defn}{\bf (Segal \cite[Postscript to section 4]{Se2})}\label{def:Segal_def}
A {\em 2-dimensional conformal field theory} $(H,U)$ consists of the following two pieces of data:
\begin{enumerate}
\item A functor $Y\mapsto H(Y)$ from the category of closed oriented smooth $1$\nb-mani\-folds to locally convex complete topological vector spaces, which takes disjoint unions to (projective) tensor products, and
\item For each oriented cobordism $\Sigma$, with conformal structure, from $Y_0$ to $Y_1$ a linear trace-class map $U(\Sigma)\colon H({Y_0})\to H(Y_1)$, subject to 
\begin{enumerate}
\item $U(\Sigma\circ \Sigma')=U(\Sigma)\circ U(\Sigma')$ when cobordisms are composed, and
\item $U(\Sigma\amalg\Sigma')=U(\Sigma)\otimes U(\Sigma')$.
\item If $f\colon \Sigma\to \Sigma'$ is a conformal equivalence  between conformal bordisms, the diagram
\begin{equation}\label{eq:conf_inv}
\xymatrix{
H(Y_0)\ar[r]^{U(\Sigma)}\ar[d]_{H(f_{|Y_0})}&H(Y_1)\ar[d]^{H(f_{|Y_0})}\\
H(Y_0')\ar[r]^{U(\Sigma')}&H(Y_1')
}
\end{equation}
is commutative.
\end{enumerate}
\end{enumerate}
Furthermore, $U(\Sigma)$ must depend smoothly on the conformal structure of $\Sigma$. 

Condition (c) is not explicitly mentioned in Segal's postscript to section 4, but it corresponds to identifying conformal surfaces with parametrized boundary in his  bordisms category if they are conformally equivalent relative boundary, which Segal does in the first paragraph of \S 4.
\end{defn}

\subsection{Internal categories}\label{subsec:internal_cat}

We note that the data $(H,U)$ in Segal's definition of a conformal field theory (Definition \ref{def:Segal_def}) can be interpreted as a pair of symmetric monoidal functors. Here $H$ is a functor from the groupoid of closed oriented smooth $1$\nb-manifolds to the groupoid of locally convex topological vector spaces. The domain of the functor $U$ is the groupoid whose objects are conformal bordisms and whose morphisms are conformal equivalences between conformal bordisms. The range of $U$ is the groupoid whose objects are trace-class operators ($=$ nuclear operators) between complete locally convex topological vector spaces and whose morphisms are commutative squares like diagram \eqref{eq:conf_inv}. The monoidal structure on the domain groupoids of $H$ and $U$ is given by the disjoint union, on the range groupoids it is given by the tensor product.

Better yet, the two domain groupoids involved fit together to form an {\em internal category} in the category of symmetric monoidal groupoids. The same holds for the two range groupoids, and the pair $(H,U)$ is a functor between these internal categories. It turns out that internal categories provide a convenient language not only for field theories a la Segal; rather, {\em all refinements that we'll incorporate in the following sections fit into this framework}. What changes is the  {\em ambient} strict $2$\nb-category, which now is the $2$\nb-category $\SymCat$ of symmetric monoidal categories, or equivalently, the $2$\nb-category $\Sym(\Cat)$ of symmetric monoidal objects in $\Cat$, the $2$\nb-category of categories. Later we will replace $\Cat$ by $\Cat/\Man$ (respectively $\Cat/\csM$) whose objects are categories fibered over the category $\Man$ of smooth manifolds (respectively $csM$ of supermanifolds).

Internal categories are described \eg in section XII.1 of the second edition of 
Mac Lane's book \cite{McL}, but his version of internal categories is too strict to define the internal bordism category we need as domain. A weakened version of internal categories and functors is defined for example by Martins-Ferreira in \cite{M} who calls them {\em pseudo categories}. This is a good reference which describes completely explicitly internal categories (also known as pseudo categories, \S 1), functors between them (also known as pseudo functors, \S2), natural transformations (also known as pseudo-natural transformations, \S 3), and modifications between natural transformations (also known as pseudo-modifications, \S 4). 

\begin{rem} There is a slight difference between the definition of natural transformation in \cite{M} and ours (see Definition \ref{def:lax_nattrans}) in that we won't insist on invertibility of a $2$\nb-cell which is part of the data of a natural transformation. It should be emphasized that our field theories of degree $n$ will be defined as natural transformations of functors between internal categories (see Definitions \ref{def:twistedEFT} and \ref{def:prelim_deg}), and here it is crucial that we don't insist on invertibility (see Remark \ref{rem:noninvertible_2cell}). The main result of \cite[Theorem 3]{M} implies  that for fixed internal categories $\sC$, $\sD$ the functors from $\sC$ to $\sD$ are the objects of a bicategory whose morphisms are natural transformations (its $2$\nb-morphisms are  called `modifications'). That result continues to hold since the construction of various compositions does not involve taking inverses. 
\end{rem}

For much of the material on internal categories covered in this subsection, we could simply refer to \cite{M} for definitions. For the convenience of the reader and since internal categories are the categorical backbone of our description of field theories, we will describe them  in some detail. We start with the definition of an internal category in an ambient category $\sA$. Then we explain why this is too strict to define our internal bordism category and go on to show how this notion can be suitably weakened if the ambient category $\sA$ is a strict $2$\nb-category. 

\begin{defn}{\bf (Internal Category)}\label{def:intcat} Let $\sA$ be a category with pull-backs (here $\sA$ stands for `ambient'). An {\em internal category} or {\em category object} in $\sA$ consists of two objects $\sC_0, \sC_1\in \sA$ and four morphisms
\[
s,t\colon \sC_1\ra \sC_0
\qquad
u\colon \sC_0\ra \sC_1
\qquad
c\colon \sC_1\times_{\sC_0}\sC_1\ra \sC_1
\]
(source, target, unit morphism and composition), subject to the following four conditions expressing the usual axioms for a category:
\begin{equation}\label{eq:dom_identity}
s\, u=1=t\, u\colon \sC_0\ra \sC_0,
\end{equation}
(this specifies source  and target of the identity map); the commutativity of the diagram
\begin{equation}\label{eq:dom_composition}
\xymatrix{
\sC_1\ar[d]_{t}
&\sC_1\times_{\sC_0}\sC_1\ar[d]^c\ar[l]_<>(.5){\pi_1}\ar[r]^<>(.5){\pi_2}
&\sC_1\ar[d]^{s}\\
\sC_0&\sC_1\ar[l]_{t}\ar[r]^{s}&\sC_0
}
\end{equation}
(this specifies source  and target of a composition);
the commutativity of the diagram
\begin{equation}\label{eq:identity}
\xymatrix{
\sC_1\ar[rr]^<>(.5){u\,t\times 1}\ar[drr]_{1}
&&\sC_1\times_{\sC_0}\sC_1\ar[d]^c
&&\sC_1\times_{\sC_0}\sC_0\ar[ll]_<>(.5){1\times u\, s}\ar[dll]^{1}\\
&&\sC_1&&
}
\end{equation}
expressing the fact the $u$ acts as the identity for composition, and the commutativity of the diagram
\begin{equation}\label{eq:associativity}
\xymatrix{
\sC_1\times_{\sC_0}\sC_1\times_{\sC_0}\sC_1\ar[r]^<>(.5){c\times 1}\ar[d]_{1\times c}
&\sC_1\times_{\sC_0}\sC_1\ar[d]^c\\
\sC_1\times_{\sC_0}\sC_1\ar[r]^c&\sC_1
}
\end{equation}
expressing associativity of composition. 

\medskip

\noindent{\bf (Functors between internal categories)} Following MacLane (\S XII.1), a functor $f\colon \sC\to \sD$ between internal categories $\sC,\sD$ in the same ambient category $\sA$ is a pair of morphisms in $\sA$
\[
f_0\colon \sC_0\ra \sD_0
\qquad
f_1\colon \sC_1\ra \sD_1.
\]
Thought of as describing the functor on ``objects" respectively ``morphisms", they are required to make the obvious diagrams commutative:
\begin{equation}\label{eq:D_0diagrams}
\xymatrix{
\sC_1\ar[r]^{f_1}\ar[d]_{s}
&\sD_1\ar[d]^{s}\\
\sC_0\ar[r]^{f_0}&\sD_0
}
\qquad
\xymatrix{
\sC_1\ar[r]^{f_1}\ar[d]_{t}
&\sD_1\ar[d]^{t}\\
\sC_0\ar[r]^{f_0}&\sD_0
}
\end{equation}
\begin{equation}\label{eq:D_1diagrams}
\xymatrix{
\sC_1\times_{\sC_0}\sC_1\ar[d]_{c_\sC}\ar[r]^<>(.5){f_1\times f_1}
&\sC_1\ar[d]^{c_\sD}\\
\sD_1\times_{\sD_0}\sD_1\ar[r]^<>(.5){f_1}&\sD_1
}
\qquad
\xymatrix{
\sC_0\ar[r]^{f_0}\ar[d]_{u}
&\sD_0\ar[d]^{u}\\
\sC_1\ar[r]^{f_1}&\sD_1
}
\end{equation}

\medskip

\noindent{\bf (Natural transformations)} If $f$, $g$ are two internal functors $\sC\to \sD$, a {\em natural transformation} $n$ from $f$ to $g$ is a morphism 
\[
n\colon \sC_0\ra \sD_1
\]
making the following diagrams commutative:
\begin{equation}\label{eq:nattrans}
\xymatrix{
&\sC_0\ar[dl]_{g_0}\ar[d]_n\ar[dr]^{f_0}\\
\sD_0&\sD_1\ar[l]_{t}\ar[r]^{s}&\sD_0
}
\qquad
\xymatrix{
\sC_1\ar[d]_{nt\times f_1}\ar[r]^<>(.5){g_1\times ns}
&\sD_1\times_{\sD_0}\sD_1\ar[d]^{c_\sD}\\
\sD_1\times_{\sD_0}\sD_1\ar[r]^{c_\sD}&\sD_1
}
\end{equation}
We note that the commutativity of the first diagram is needed in order to obtain the arrows $gt\times f_1$, $g_1\times ns$ in the second diagram. If the ambient category $\sA$ is the category of sets, then $n$ is a natural transformation from the functor $f$ to the functor $g$; the first diagram expresses the fact that for an object $a\in \sC_0$ the associated morphism $n_a\in \sD_1$ has domain $f_0(a)$ and range $g_0(a)$. The second diagram expresses the fact that for every morphism $h\colon a\to b$ the diagram 
\[
\xymatrix{
f_0(a)\ar[r]^{n_a}\ar[d]_{f_1(h)}&g_0(a)\ar[d]^{g_1(h)}\\
f_0(b)\ar[r]^{n_b}&g_0(b)
}
\]
is commutative.
\end{defn}

As mentioned before, we would like to regard Segal's  pair $(H,U)$ as a functor between internal categories where the ambient category $\sA$ is the category of symmetric monoidal groupoids. However, this is not quite correct due to the lack of associativity of the internal bordism category. In geometric terms, the problem is that if
$\Sigma_{i}$ is a bordism from $Y_i$ to $Y_{i+1}$ for $i=1,2,3$, then $(\Sigma_3\cup_{Y_3}\Sigma_2)\cup_{Y_2}\Sigma_1$ and $\Sigma_3\cup_{Y_3}(\Sigma_2\cup_{Y_2}\Sigma_1)$ are not strictly speaking {\em equal}, but only canonically conformally equivalent.
In categorical terms, this means that the diagram \eqref{eq:associativity} is not commutative; rather, the conformal equivalence between these bordisms is a morphism in the groupoid $\sC_1$ whose objects are conformal bordisms. This depends functorially on $(\Sigma_3,\Sigma_2,\Sigma_1)\in \sC_1\times_{\sC_0}\sC_1\times_{\sC_0}\sC_1$ and hence it provides an invertible natural transformation $\alpha$ between the two functors of diagram \eqref{eq:associativity}
\begin{equation}\label{eq:lax_associativity}
\xymatrix{
\sC_1\times_{\sC_0}\sC_1\times_{\sC_0}\sC_1\ar[r]^<>(.5){c\times 1}\ar[d]_{1\times c}
&\sC_1\times_{\sC_0}\sC_1\ar[d]^c\\
\sC_1\times_{\sC_0}\sC_1\ar[r]_c\ar@{=>}[ur]^\alpha_\cong&\sC_1
}
\end{equation}
The moral is that we should relax the associativity axiom of an internal category by replacing the assumption that the diagram above is commutative by the weaker assumption that the there is an invertible  $2$\nb-morphism $\alpha$ between the two compositions. This of course requires that the ambient category $\sA$ can be refined to be a strict $2$\nb-category (which happens in our case, with objects, morphisms, respectively $2$\nb-morphisms being symmetric monoidal groupoids, symmetric monoidal functors, respectively symmetric monoidal natural transformations).

This motivates the following definition:
\begin{defn}\label{def:lax_intcat}
An {\em internal category in a strict $2$\nb-category $\sA$} consists of the following data:
\begin{itemize}
\item objects $\sC_0$, $\sC_1$ of $\sA$;
\item morphisms $s,t,c$ of $\sA$ as in definition \ref{def:intcat};
\item invertible $2$\nb-morphisms $\alpha$ from \eqref{eq:lax_associativity}, and invertible $2$\nb-morphisms $\lambda$, $\rho$ in the following diagram
\end{itemize}
\begin{equation}\label{eq:lax_identity}
\xymatrix{
\sC_1\ar[rr]^<>(.5){ut\times 1}\ar[drr]_{1}^{}="2"
&&\sC_1\times_{\sC_0}\sC_1\ar[d]^c
&&\sC_1\ar[ll]_<>(.5){1\times us}\ar[dll]^{1}_{}="1"\\
&&\sC_1&&
\ar@{=>}^-\lambda "1,3";"2"
\ar@{=>}_-\rho "1,3";"1"
}
\end{equation}
The morphisms are required to satisfy conditions \eqref{eq:identity} and \eqref{eq:dom_composition}.  The $2$\nb-morphisms are subject to two coherence diagrams (see \cite[Diagrams (1.6) and (1.7)]{M}); in particular, there is a pentagon shaped diagram involving the `associator' $\alpha$.
In addition, it is required that composing the $2$\nb-morphisms $\alpha$, $\lambda$, or $\rho$ with the identity $2$\nb-morphism on $s$ or $t$ gives an identity $2$\nb-morphism.

For comparison with \cite{M} it might be useful to note that in that paper the letter $d$ (domain) is used instead of our $s$ (source), $c$ (codomain) instead of $t$ (target), $m$ (multiplication) instead of $c$ (composition) and $e$ instead of $u$ (unit).
\end{defn}

\begin{rem} In the special case of an internal category where $\sC_0$ is a terminal object of the strict $2$\nb-category $\sA$, this structure is called a {\em monoid in $\sA$}. The composition 
\[
c\colon \sC_1\times\sC_1\ra \sC_1,
\]
is thought of as a multiplication on $\sC_1$ with unit $u\colon \sC_0\to \sC_1$.
For example, a monoid in $\sA=\Cat$ is a {\em monoidal category}. 

Similarly, a {\em symmetric monoidal category} can be viewed as a {\em symmetric monoid} in $\Cat$ in the sense of the definition below. The point of this is that we will need to talk about symmetric monoids in other strict $2$\nb-categories, \eg categories fibered over some fixed category $\sS$. 
\end{rem}

\begin{defn}\label{def:symm_monoid}
A {\em symmetric monoid} in a strict $2$\nb-category $\sA$ is a monoid $(\sC_1,c,u,\alpha,\lambda,\rho)$ in $\sA$ together with an invertible $2$\nb-morphism $\sigma$ called {\em braiding isomorphism}:
\begin{equation}
\xymatrix{
\sC_1\times \sC_1\ar[rr]^m\ar[d]_\tau&&\sC_1\\
\sC_1\times\sC_1\ar[rru]_m^{}="1"&&
\ar@{=>}_<>(.5)\sigma^<>(.5)\cong"1,1";"1"
}
\end{equation}
Here $\tau$ is the morphism in $\sA$ that switches the two copies of $\sC_1$. The $2$\nb-morphisms $\sigma$, $\alpha$, $\lambda$ and $\rho$ are subject to  coherence conditions well-known in the case $\sA=\Cat$  \cite[Ch.\ XI, \S 1]{McL}. 
\end{defn}

Next we define functors between categories internal to a $2$\nb-category by weakening Definition \ref{def:intcat}.

\begin{defn}\label{def:lax_intfun}
Let $\sC$, $\sD$ be internal categories in a strict $2$\nb-category $\sA$. Then a {\em functor} $f\colon \sC\to \sD$ is a quadruple $f=(f_0,f_1,\mu,\epsilon)$, where $f_0\colon \sC_0\to \sD_0$, $f_1\colon \sC_1\to \sD_1$  are morphisms, and $\mu$, $\epsilon$  are invertible $2$\nb-morphisms
\[
\xymatrix{
\sC_1\times_{\sC_0}\sC_1\ar[d]_{f_1\times f_1}\ar[r]^<>(.5){c_\sC}
&\sC_1\ar[d]^{f_1}\\
\sD_1\times_{\sD_0}\sD_1\ar@{=>}[ur]^{\mu}_\cong\ar[r]_<>(.5){c_\sD}&\sD_1
}
\qquad
\xymatrix{
\sC_0\ar[r]^{f_0}\ar[d]_{u}
&\sD_0\ar[d]^{u}\\
\sC_1\ar[r]^{f_1}\ar@{=>}[ur]^{\epsilon}_\cong&\sD_1
}
\]
It is required that the diagrams \eqref{eq:D_0diagrams} commute. The $2$\nb-morphisms $\mu$, $\epsilon$ are subject to three coherence conditions (see \cite[Diagrams (2.5) and (2.6)]{M}) as well as the usual condition that horizontal composition with the identity $2$\nb-morphisms $1_s$ or $1_t$ results in an identity $2$\nb-morphism. 
\end{defn}

\begin{defn}\label{def:lax_nattrans} Let $f,g\colon \sC\to \sD$ be internal functors between internal categories in a strict $2$\nb-category $\sA$. A {\em natural transformation} from $f=(f_0,f_1,\mu^f,\epsilon^f)$ to $g=(g_0,g_1,\mu^g,\epsilon^g)$ is a pair $n=(n,\nu)$, where $n\colon \sC_0\to D_1$ is a morphism, and $\nu$ is a $2$\nb-morphism:
\begin{equation}\label{eq:lax_nattrans}
\xymatrix{
\sC_1\ar[d]_{g_1\times n\,s}\ar[r]^<>(.5){n\,t\times f_1}
&\sD_1\times_{\sD_0}\sD_1\ar[d]^{c_\sD}\\
\sD_1\times_{\sD_0}\sD_1\ar[r]^{c_\sD}\ar@{=>}[ur]^{\nu}&\sD_1
}
\end{equation}
It is required that the first diagram of  \eqref{eq:nattrans} is commutative. There are two coherence conditions for the $2$\nb-morphism $\nu$ (see \cite[Diagrams (3.5) and (3.6)]{M}), and the usual requirement concerning horizontal composition with $1_s$ and $1_t$ (these are the identity $2$\nb-morphisms of the morphisms $s$ respectively $t$).
\end{defn}
Note that the $2$\nb-morphism $\nu$ is  not required to be invertible. Some authors add the word `lax' in this more general case. See Remark~\ref{rem:noninvertible_2cell} for the reason why this more general notion is important for twisted field theories.

\subsection{The internal Riemannian bordism category}\label{subsec:Riem_bord}
Now we are ready to define $\RB d$, the category of $d$\nb-dimensional Riemannian bordisms. This is a category internal to $\sA=\SymGrp$, the strict $2$\nb-category of symmetric monoidal groupoids. 

We should mention that in our previous paper \cite{HoST} we defined the  Riemannian bordism category $d\text{-\hskip -.01in{\sf RB}}$. This is just a category, rather than an internal category like $\RB d$. In this paper we are forced to deal with this more intricate categorical structure, since we wish to consider {\em twisted} field theories, in particular field theories of {\em non-trivial degree}.

Before giving the formal definition of $\RB d$, let us make some remarks that hopefully will motivate the definition below. 
Roughly speaking, $\RB{d}_0$ is the symmetric monoidal groupoid of closed Riemannian manifolds of dimension $d-1$, and $\RB{d}_1$ is the symmetric monoidal groupoid of Riemannian bordisms of dimension $d$. The problem is to define the composition functor
\[
c\colon \RB d_1\times_{\RB d_0}\RB d_1\ra \RB d_1.
\]
If $\Sigma_1$ is a Riemannian bordism from $Y_0$ to $Y_1$ and $\Sigma_2$ is a Riemannian bordism from $Y_1$ to $Y_2$, $c$ should map $(\Sigma_2,\Sigma_1)\in \RB d_1\times_{\RB d_0}\RB d_1$ to the Riemannian manifold $\Sigma:=\Sigma_2\cup_{Y_1}\Sigma_1$ obtained by gluing $\Sigma_2$ and $\Sigma_1$ along their common boundary component $Y_1$. The problem is that the Riemannian metrics on $\Sigma_1$ and $\Sigma_2$ might not fit together to give a Riemannian metric on $\Sigma$. A necessary, but not sufficient condition for this is that the second fundamental form of $Y_1$ as a boundary of $\Sigma_1$ matches with the second fundamental form of $Y_1$ as a boundary of $\Sigma_2$.

In the usual gluing process of $d$\nb-dimensional bordisms, the two glued bordisms intersect in a closed $(d-1)$\nb-dimensional manifold $Y$, the object of the bordism category which is the source (respectively target) of the bordisms to be glued. For producing a Riemannian  structure on the glued bordism (actually, even for producing a smooth structure on it), it is better if the intersection is an {\em open} $d$\nb-manifold on which the Riemannian structures are required to match. This suggests to refine an {\em object} of the bordism category $\RB d_0$ to be  a pair $(Y,Y^c)$, where  $Y$  is an open Riemannian $d$\nb-manifold, and $Y^c\subset Y$ (the {\em core} of $Y$) is a closed $(d-1)$\nb-dimensional submanifold of $Y$. We think of $Y$ as {\em Riemannian collar} of the $(d-1)$\nb-dimensional core manifold $Y^c$ (this core manifold is the only datum usually considered). In order to distinguish the domain and range of a bordism, we will in addition require a decomposition $Y\smallsetminus Y^c=Y^+\amalg Y^-$ of the complement into  disjoint open subsets, both of which contain $Y^c$ in its closure.  
Domain and range of a bordism is customarily controlled by comparing the given orientation of the closed manifold $Y^c$ with the orientation induced by thinking of it as a part of the boundary of an oriented bordism $\Sigma$. Our notion makes it unnecessary to furnish our manifolds with orientations.

Our main goal here is to define the $d$\nb-dimensional Euclidean bordism category $\EB d$. We first define the Riemannian bordism category $\RB d$ and then $\EB d$ as the variation where we insist that all Riemannian metrics are {\em flat} and that the cores $Y^c\subset Y$ are totally geodesic. We want to provide pictures and it's harder to draw interesting pictures of flat surfaces (\eg the flat torus doesn't embed in $\R^3$). 

\begin{defn}\label{def:RB} The {\em $d$\nb-dimensional Riemannian bordism category $\RB d$} is the category internal to the strict $2$\nb-category $\sA=\SymGrp$ of symmetric monoidal group\-oids  defined as follows. Note that all Riemannian manifolds that arise are without boundary.

\medskip

\noindent{\bf The object groupoid ${\RB d}_0$.} The {\em objects} of the groupoid ${\RB d}_0$ are quadruples  $(Y,Y^c,Y^\pm)$, where $Y$ is a Riemannian $d$-manifold (without boundary and usually non-compact) and $Y^c\subset Y$ is a {\em compact} codimension~$1$ submanifold which we call the {\em core} of $Y$. Moreover, we require that $Y\smallsetminus Y^c=Y^+\amalg Y^-$, where $Y^\pm\subset Y$ are disjoint open subsets whose closures contain $Y^c$. Often we will suppress the data  $Y^c,Y^\pm$ in the notation and write $Y$ for an object of ${\RB d}$.

 \begin{figure}[h]
\begin{center}
\scalebox{.60}{\includegraphics{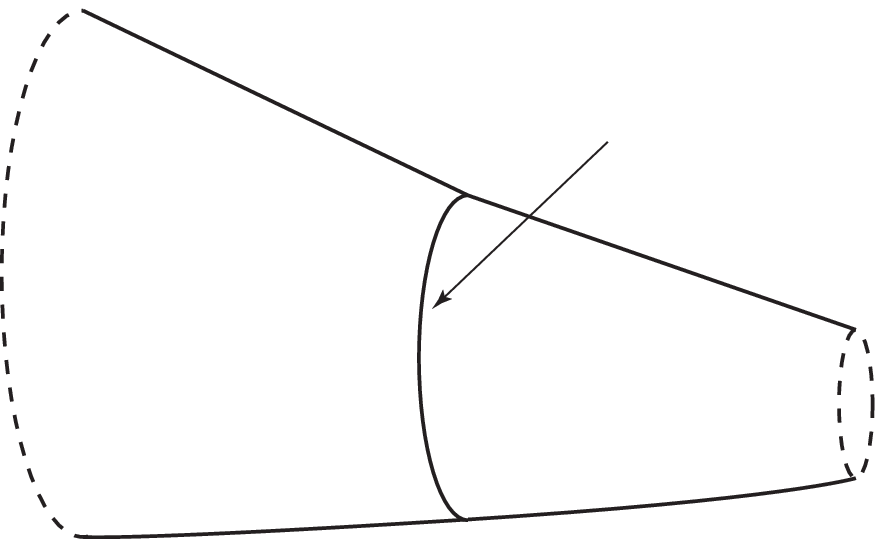}}
  \put(-130,30)	{$Y^-$}
  \put(-40,20)	{$Y^+$}
 \put(-45,75)	{$Y^c$}
\caption{ An object $(Y,Y^c,Y^\pm)$ of ${\RB 2}_0$}
\label{fig:object}
\end{center}
\end{figure}

An {\em isomorphism} in $\RB{d}$ from $Y_0$ to $Y_1$ is the germ of an invertible isometry $f\colon W_0\to W_1$. Here $W_j\subset Y_j$ are open neighborhoods of $Y_j^c$ and $f$ is required to send $Y_0^c$ to $Y_1^c$ and $W_0^\pm$ to $W_1^\pm$ where $W_j^\pm:=W_j\cap Y_j^\pm$. As usual for germs, two such isometries represent the {\em same} isomorphism if they agree on some smaller open neighborhood of $Y_0^c$ in $Y_0$.  

We remark that if $(Y,Y^c,Y^\pm)$ is an object of $\RB{d}$, and $W\subset Y$ is an open neighborhood of $Y^c$, then $(Y,Y^c,Y^\pm)$ is isomorphic to $(W,Y^c,Y^\pm\cap W)$. In particular, we can always assume that $Y$ is diffeomorphic to $Y^c\times (-1,+1)$, since by the tubular neighborhood theorem, there is always a neighborhood $W$ of $Y^c$ such that the pair $(W,Y^c)$ is diffeomorphic to $(Y^c\times (-1,+1),Y^c\times \{0\})$. 
 
 \medskip
 
\noindent{\bf The morphism groupoid ${\RB d}_1$} is defined as follows. 
An object of ${\RB d}_1$ consists of a pair $Y_0=(Y_0,Y^c_0)$, $Y_1=(Y_1,Y^c_1)$ of objects of ${\RB d}_0$ (the source respectively target) and a {\em Riemannian bordism} from $Y_0$ to $Y_1$, which is a triple  $(\Sigma,i_0,i_1)$ consisting of a Riemannian $d$-manifold $\Sigma$ and smooth maps $i_j\colon W_j\to \Sigma$. Here $W_j\subset Y_j$ are open neighborhoods of the cores $Y^c_j$.  Letting $i_j^\pm: W_j^\pm \to \Sigma$ be the restrictions of $i_j$ to $W_j^\pm:=W_j\cap Y_j^\pm$, we require that 
\begin{itemize}
\item[(+)] $i_j^+$ are isometric embeddings into $\Sigma \smallsetminus i_1(W_1^-\cup Y_1^c)$ and
\item[$(c)$] the {\em core} $\Sigma^c:=
\Sigma\smallsetminus \left(i_0(W_0^+)\cup i_1(W_1^-)\right)$ is compact.
\end{itemize}
Particular  bordisms are given by isometries $f\colon W_0\to W_1$ as above, namely by using $\Sigma= W_1$, $i_1=\id_{W_1}$ and $i_0=f$. Note that in this case the images of $i_0^+$ and $i_1^+$ are {\em not disjoint} but we didn't require this condition. 

Below is a picture of a Riemannian bordism; we usually draw the domain of the bordism to the right of its range, since we want to read compositions of bordisms, like compositions of maps, from right to left.
\begin{figure}[h]
\begin{center}
\scalebox{.80}{\includegraphics{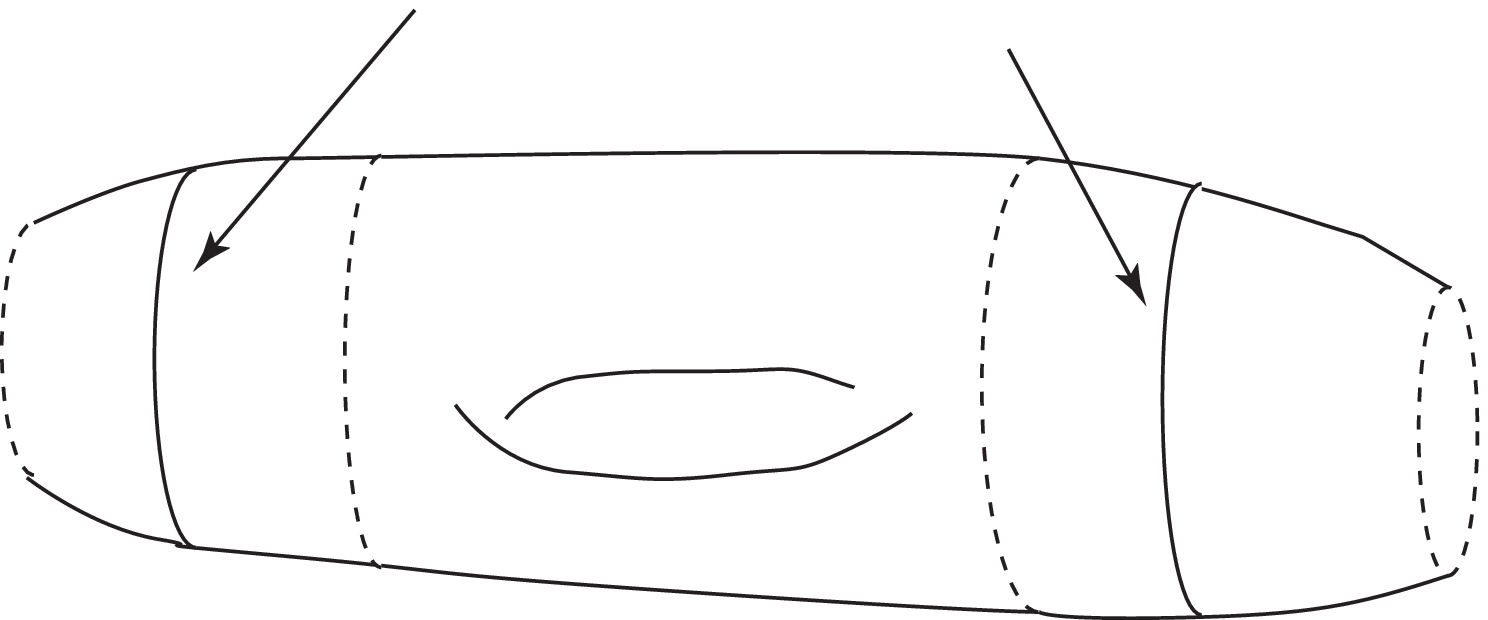}}
 \put(-55,50)	{$\scriptstyle i_0(W^+_0)$}
\put(-110,50)	{$\scriptstyle i_0(W^-_0)$}
\put(-125,140)	{$\scriptstyle i_0(Y^c_0)$}
\put(-180,70)	{$\Sigma$}
  \put(-295,55)	{$\scriptstyle i_1(W^+_1)$}
 \put(-338,55)	{$\scriptstyle i_1(W^-_1)$}
\put(-255,150)	{$\scriptstyle i_1(Y^c_1)$}
 \put(-300,-5)	{$\underbrace{\quad\qquad\qquad\qquad\qquad\qquad\qquad\qquad\qquad\qquad\quad}_{\Sigma^c}$}
 \caption{ A Riemannian bordism (object of ${\RB 2}_1$)}
\label{fig:bordism}
\end{center}
\end{figure}
Roughly speaking, a bordism between objects $Y_0$ and $Y_1$ of ${\RB d}_0$ is just an ordinary bordism $\Sigma^c$ from $Y_0^c$ to $Y_1^c$ equipped with a Riemannian metric, thickened up a little bit at its boundary to make gluing possible..

A morphism from a bordism $\Sigma$ to a bordism $\Sigma'$ is a germ of a triple of isometries
\begin{align*}
F\colon X\to X'
\qquad
f_0\colon V_0\to V'_0
\qquad
f_1\colon V_1\to V'_1.
\end{align*}
Here $X$ (respectively $V_0, V_1$) is an open neighborhood of $\Sigma^c\subset \Sigma$ (respectively $Y^c_0\subset W_0\cap i_0^{-1}(X), Y^c_1\subset W_1\cap i_1^{-1}(X)$) and similarly for $X'$, $V_0'$, $V_1'$. We require the conditions for $f_j$ to be a morphism from $Y_j$ to $Y_j'$ in ${\RB d}_0$, namely $f_j(Y^c_j)=(Y_j')^c$ and $f_j(V_j^\pm)=(V_j')^\pm$. In addition, we require that these isometries are compatible in the sense that the diagram
\[
\xymatrix{
V_1\ar[r]^{i_1}\ar[d]_{f_1}
&X\ar[d]_F&V_0\ar[l]_{i_0}\ar[d]_{f_0}\\
V'_1\ar[r]_{i'_1}
&X'&V'_0\ar[l]^{i'_0}
}
\]
is commutative. Two such triples $(F,f_0,f_1)$ and $(G,g_0,g_1)$ represent the same germ if there there are smaller open neighborhoods $X''$ of $\Sigma^c\subset X$ and  $V_j''$ of $Y_j\subset V_j\cap i_j^{-1}(X'')$ such that  $F$ and $G$ agree on $X''$, and $f_j$ and $g_j$ agree on $V_j''$ for $j=0,1$.

\medskip

\noindent{\bf Source, target, unit and composition functors.} The functors 
\[
s,t\colon {\RB d}_1\ra {\RB d}_0
\]
send a bordism $\Sigma$ from $Y_0$ to $Y_1$ to $Y_0$ respectively $Y_1$. There is also the functor
 \[
u\colon {\RB d}_0\ra {\RB d}_1
\]
that sends $(Y,Y^c)$ to the Riemannian bordism given by the identity isometry $\id\colon Y\to Y$.
 These functors are compatible with taking disjoint unions and hence they are symmetric monoidal functors, \ie morphisms in $\SymGrp$.

There is also a {\em composition functor}
\[
c\colon {\RB d}_1\times_{{\RB d}_0}{\RB d}_1\ra{\RB d}_1
\]
given by gluing bordisms. Let us describe this carefully, since there is a subtlety involved here due to the need to adjust the size of the Riemannian neighborhood along which we glue. Let $Y_0$, $Y_1$, $ Y_2$ be objects of ${\RB d}_0$, and let $\Sigma$, $\Sigma'$ be  bordisms from $Y_0$ to $Y_1$ respectively from $Y_1$ to $Y_2$. These data involve in particular smooth maps
\[
i_1\colon W_1\ra \Sigma
\qquad
i_1'\colon W_1'\ra \Sigma',
\]
where $W_1,W_1'$ are open neighborhoods of $Y^c_1\subset Y_1$. We set $W_1''\=W_1\cap W_1'$ and note that our conditions guarantee that $i_1$ (respectively $i_1'$) restricts to an isometric embedding of  $(W_1'')^+\=W_1''\cap Y_1^+$ to $\Sigma$ (respectively $\Sigma'$). We use these isometries to glue $\Sigma$ and $\Sigma'$ along $(W_1'')^+$ to obtain $\Sigma''$ defined as follows:
\[
\Sigma''\=\left(\Sigma'\smallsetminus i_1'((W'_1)^+\smallsetminus (W_1'')^+)\right)
\cup_{(W_1'')^+}
\left(\Sigma\smallsetminus i_1(W^-_1\cup Y_1^c)\right)
\]
The maps $i_0\colon W_0\to \Sigma$ and $i_2\colon W_2\to \Sigma'$ can be restricted to maps (on smaller open neighborhoods) into $\Sigma''$ satisfying our  conditions. This makes $\Sigma''$ a bordism from $Y_0$ to $Y_2$. 

As explained above (see Equation \eqref{eq:lax_associativity}), the composition functor $c$ is not strictly associative, but there is a natural transformation $\alpha$ as in diagram \eqref{eq:lax_associativity} which satisfies the pentagon identity. 
\end{defn}

\begin{rem}\label{rem:Hausdorff}
We point out that conditions (+) and (c) in the above definition of a Riemannian bordism also make sure that the composed bordism is again a Hausdorff space. In other words, gluing two topological spaces along open subsets preserves conditions like `locally homeomorphic to $\R^n$' and structures like Riemannian metrics. However, it can happen that the glued up space is not Hausdorff, for example if one glues two copies of $\R$ along the interval $(0,1)$. The reader is invited to check that our claim follows from the following easy lemma.
\end{rem}
\begin{lem}\label{lem:Hausdorff}
Let $X,X'$ be manifolds and let $U$ be an open subset of $X$ and $X'$. Then $X\cup_U X'$ is a manifold if and only if the natural map $U\to X \times X'$ sends $U$ to a {\em closed} set.
\end{lem}

\subsection{The internal category of vector spaces}\label{subsec:vect}

\begin{defn}\label{def:TV} The category $\TV$ of (complete locally convex) topological vector spaces internal to $\SymGrp$ (the strict $2$\nb-category category of symmetric monoidal groupoids) is defined by the object respectively morphism groupoids as follows.
\begin{itemize}
\item[$\TV_0$] is the groupoid whose objects are complete locally convex topological vector spaces over $\C$ and whose morphisms are invertible continuous linear maps. The completed projective tensor product gives $\TV_0$ the structure of a symmetric monoidal groupoid.
\item[$\TV_1$] is the symmetric monoidal groupoid whose objects are continuous linear maps $f\colon V_0\to V_1$. The morphisms from $f\colon V_0\to V_1$ to $f'\colon V_0'\to V_1'$ are a pair of isomorphisms $(g_0,g_1)$ making the diagram
\[
\xymatrix{
V_0\ar[r]^{g_0}_\cong\ar[d]_f&V_0'\ar[d]^{f'}\\
V_1\ar[r]^{g_1}_\cong&V_1'
}
\]
commutative. It is a symmetric monoidal groupoid via the projective tensor product.
\end{itemize}
There are obvious source, target, unit and composition functors
\[
s,t\colon \TV_1\ra \TV_0
\qquad
u\colon \TV_0\ra \TV_1
\qquad
c\colon \TV_1\times_{\TV_0}\TV_1\ra \TV_1
\]
which make $\TV$ an internal category in $\SymGrp$. 
This is a {\em strict} internal category in the sense that associativity holds on the nose (and not just up to natural transformations).
\end{defn}

Now we are ready for a preliminary definition of a $d$\nb-dimensional Euclidean field theory, which will be modified by adding a smoothness condition in the next section.

\begin{defn}\label{def:prelim}{\bf (Preliminary!)} A $d$\nb-dimensional {\em Riemannian field theory} over a smooth manifold $X$ is a functor 
\[
E\colon \RB d(X)\ra\TV
\]
of categories internal to $\SymGrp$,  the strict $2$\nb-category of symmetric monoidal group\-oids. A functor $E\colon \RB{d}(\pt)=\RB{d}\to \TV$ is a $d$\nb-dimensional  Riemannian field theory.

Similarly, a $d$\nb-dimensional {\em Euclidean field theory} over a smooth manifold $X$ is a functor 
\[
E\colon \EB d(X)\ra\TV
\]
of categories internal to $\SymGrp$, where the Euclidean bordism category ${\EB d}$ is defined completely analogously to $\RB{d}$ by using Euclidean  structures ($=$ flat Riemannian metrics) instead of Riemannian metrics, and by requiring that for any object $(Y,Y^c,Y^\pm)$ the core $Y^c$ is a totally geodesic submanifold of $Y$. 
\end{defn}

The feature missing from the above definition is the requirement that $E$ should be {\em smooth}. Heuristically, this means that the vector space $E(Y)$ associated to an object $Y$ of the bordism category as well as the operator $E(\Sigma)$ associated to a bordism $\Sigma$ should depend smoothly on $Y$ respectively $\Sigma$. To make this precise, we replace the categories ${\EB d}$, $\TV$ by their {\em family versions}  whose objects and morphisms are {\em smooth families} of objects/morphisms of the original category parametrized by some manifold. 
\subsection{Families of rigid geometries}\label{subsec:rigid}

In this subsection we define families of Euclidean $d$\nb-manifolds, or more generally a families of manifolds equipped with a rigid geometry. This leads to the definition of the bordism category $\GMBord$ of manifolds with rigid geometry $(G,\M)$ (see Definition \ref{def:GMBord}) and to the notion of a field theory based on a rigid geometry (see Definition \ref{def:GMFT}).

Let $G$ be a Lie group acting on a manifold $\M$. We want to think of $\M$  as the local model for {\em rigid geometries} with isometry group $G$. This idea is very well explained in \cite{Th} and goes back to Felix Klein. 

\begin{defn}\label{def:GM_object} 
A {\em $(G,\M)$-structure} on a manifold $Y$ is a maximal atlas consisting  of {\em charts} which are diffeomorphisms 
\[
\xymatrix{
Y\supseteq U_i\ar[r]^{\varphi_i}_\cong& V_i\subseteq \M
}
\]
between open subsets of $Y$ and open subsets of $\M$ such that the $U_i$'s cover $Y$, and a collection of elements $g_{ij}\in G$ which determine the transition functions in the sense that for every $i,j$ the diagram
\[
\xymatrix{
&U_i\cap U_j\ar[ld]_{\varphi_j}\ar[dr]^{\varphi_i}&\\
\M\ar[rr]^{g_{ij}}&&\M
} 
\]
is commutative (here we interpret $g_{ij}\in G$ as an automorphism of $\M$ via the action map $G\times\M\to \M$). The $g_{ij}$'s are required to satisfy the cocycle condition
\begin{equation}\label{eq:cocycle}
g_{ij}\cdot g_{jk}=g_{ik}
\end{equation}
\end{defn}

We note that if $G$ acts effectively on $\M$, then the element $g_{ij}\in G$ is determined by $\varphi_i$ and $\varphi_j$. In particular, the cocycle condition is automatic, and the definition can be phrased in a simpler way as in our previous paper \cite[Definition 6.13]{HoST}. In this paper we also wish to consider cases where $G$ does not act effectively on $\M$ as in the second example below. 

For our definition of the bordism category $\GMBord$ we will also need the notion of a $(G,\M)$-structure on a pair $(Y,Y^c)$ consisting of a manifold $Y$ and a codimension one submanifold $Y^c$.

\begin{defn}\label{def:GMpair} Let $(G,\M)$ be a geometry and let $\M^c\subset \M$ be a codimension one submanifold of $\M$ (there is no condition relating $\M^c$ and the $G$\nb-action). From now on, a `geometry' will refer to such a triple $(G,\M,\M^c)$, but we will suppress $\M^c$ in our notation. A {\em $(G,\M)$-structure on a pair $(Y,Y^c)$}   is a maximal atlas consisting of charts $\{(U_i,\varphi_i)\}$ for $Y$ as in Definition \ref{def:GM_object} above with the additional requirement that $\varphi_i$ maps $U_i\cap Y^c$ to $\M^c\subset\M$.
\end{defn}

\begin{ex}\label{ex:euclidean}
Let  $\M:=\E^d:=\R^d$ be the $d$\nb-dimensional Euclidean space, given by the manifold $\R^d$ equipped with its standard Riemannian metric. Let $G:=\Iso(\E^d)$ be the isometry group of $\E^d$. More explicitly, $\Iso(\E^d)$  is the semi-direct product $\E^d\rtimes O(d)$ of $\E^d$ (acting on itself by translations) and the orthogonal group $O(d)$. Let $\M^c:=\R^{d-1}\times \{0\}\subset \E^d$.

 A {\em Euclidean structure} on a smooth $d$-manifold $Y$ is an $(\Iso(\E^d), \E^d)$-structure in the above sense. It is clear that such an atlas determines a flat Riemannian metric on $Y$ by transporting the metric on $\E^d$ to $U_i$ via the diffeomorphism $\varphi_i$. Conversely, a flat Riemannian metric on a manifold $Y$ can be used to construct such an atlas.  
An $(\Iso(\E^d), \E^d)$\nb-pair amounts to:
\begin{itemize}
\item a flat Riemannian manifold $Y$;
\item a totally geodesic codimension one submanifold $Y^c\subset Y$;
\end{itemize}
\end{ex}

\begin{ex}\label{ex:euclidean_spin}
A {\em Euclidean spin structure} on a $d$\nb-manifold $Y$ is exactly an $(\Iso(\E^{d|0}),\E^{d|0})$-structure on $Y$. Here $\E^{d|0}=\E^d$, and $\Iso(\E^{d|0})=\E^d\rtimes Spin(d)$ is a double covering of $\E^d\rtimes SO(d)\subset \E^d\rtimes O(d)$ and acts on $\E^d$ via the double covering map. The peculiar notation $(\Iso(\E^{d|0}),\E^{d|0})$ is motivated by later generalizations: in section \ref{subsec:superEuclidean} we will define the super Euclidean space $\E^{d|\delta}$ (a supermanifold of dimension $d|\delta$) and the super Euclidean group $\Iso(\E^{d|\delta})$ (a super Lie group which acts on $\E^{d|\delta}$).

We note that the  action of $G=\Iso(\E^{d|0})$ on $\E^{d}$ is {\em not effective}, since $-1\in Spin(d)$ acts trivially. 
However, it lifts to an effective action on the principal $Spin(d)$\nb-bundle $\E^d\times Spin(d)\to \E^d$ where the translation subgroup  $\E^d$ acts trivially on $Spin(d)$, and $Spin(d)\subset G$ acts by left-multiplication. This implies that the transition elements $g_{ij}\in G$ determine a {\em spin-structure} on $Y$, \ie a principal $Spin(d)$\nb-bundle $Spin(Y)\to Y$ which is a double covering of the oriented frame bundle $SO(Y)$. Hence a $(\Iso(\E^{d|0}),\E^{d|0})$\nb-structure on a manifold $Y$ determines a flat Riemannian metric and a spin structure. Conversely, a spin structure on a flat Riemannian manifold $Y$ determines a $(\Iso(\E^{d|0}),\E^{d|0})$\nb-structure. 
\end{ex}

The last example shows that the definition of morphisms between manifolds equipped with $(G,\M)$-structures requires some care: if $Y$, $Y'$ are manifolds with $(\Iso(\E^{d|0}),\E^{d|0})$\nb-structures, a morphisms $Y\to Y'$ should not just be an isometry $f\colon Y\to Y'$ compatible with the spin structures, but it should include an {\em additional datum}, namely a map $Spin(Y)\to Spin(Y')$ of the principal $Spin(d)$\nb-bundles. In particular, each $(\Iso(\E^{d|0}),\E^{d|0})$\nb-manifold $Y$ should have an involution which acts trivially on $Y$, but is multiplication by $-1\in Spin(d)$ on the principal bundle $Spin(Y)\to Y$. 

It may be informative to compare this rigid geometry to other version of geometric structures on manifolds. If a Lie group $H$ acts on a finite dimensional vector space $V$, an {\em $H$-structure} on a smooth manifold X is an $H$-principal bundle $P$ together with an isomorphism $P \times_H V \cong  TX$. 
An example of such an $H$-structure is the {\it flat $H$-structure\/} on the vector space~$V$ itself.
It is given by $P:=V\times H\to V$ and the isomorphism is given by~$P\times_H V
=(V\times H)\times_H V\cong V\times V\cong TV$ via translation in~$V$.
\begin{defn}  An $H$-structure on~$X$ is {\it integrable\/} if it is locally flat.
\end{defn}
\noindent Examples of this notion include
\begin{enumerate}
\item An integrable $GL_n(\C)$-structure (on $V=\C^n$) is a complex structure: There are complex charts as discussed below.
\item For $H=Sp(2n)$ and $V=\R^{2n}$ integrable structures
are symplectic structures by Darboux's theorem.
\item  For $H=O(n)$ and $V=\R^n$ integrable structures are flat Riemannian metrics.
\item For $H=U(n)$ and $V=\C^n$ integrable structures are flat K\"ahler structures.
\end{enumerate}
The total space of the cotangent bundle $T^*X$ carries a canonical (exact) symplectic structure. Moreover,
a Riemannian metric on~$X$ induces one on~$T^*X$.
This metric is integrable if and only if the original metric is flat.
Nevertheless, it turns out that $T^*X$ is always K\"ahler.

There is a chart version of integrable $H$-structures.
Choose a covering collection of charts on the manifold~$X$ with codomain being an open
subset of a vector space~$V$. We can now say that an integrable $H$-structure
is a lift of the derivatives of transition functions $\varphi_{ij}$ along the map $\rho\colon H\to GL(V)$
that satisfies the usual cocycle conditions. 

An $H$-structure on~$X$ is {\it rigid\/} (and in particular integrable) if each transition function $\varphi_{ij}$ is the (locally constant)
restriction of the action on~$V$ by the semi-direct product $G= (V,+) \rtimes H$ of `translations and rotations', see Definition~\ref{def:GM_object}. For $H=O(n)$ and $V=\R^n$ we get the notion of a rigid Euclidean manifold used in this paper (which is equivalent to an integrable $O(n)$-structure).
For a general rigid geometry one would generalize the model space from a vector space to an arbitrary (homogenous) $H$-space.

We now turn to the morphisms between manifolds with rigid geometries.
\begin{defn}\label{def:GM_morphism} We denote by $\GMMan$ the category whose objects are $(G,\M)$-manifolds. If  $Y$, $Y'$ are $(G,\M)$-manifolds, a morphism  from $Y$ to $Y'$ consists of a smooth map $f\colon Y\to Y'$ and  elements $f_{i'i}\in G$ for each pair of charts $(U_i,\varphi_i)$, $(U'_{i'},\varphi'_{i'})$ with $f(U_i)\subset U'_{i'}$ such that the diagram
\[
\xymatrix{
U_i\ar[r]^f\ar[d]_{\varphi_i}&U_{i'}\ar[d]^{\varphi'_{i'}}\\
\M\ar[r]^{f_{i'i}}&\M
}
\]
commutes. These elements of $G$ are required to satisfy the coherence condition
\[
f_{j'j}\cdot g_{ji}=g'_{j'i'}\cdot f_{i'i}
\]
If the $G$\nb-action on $\M$ is effective, then the elements $f_{i'i}\in G$ are determined by $f$ and the charts $\varphi_i$, $\varphi_{i'}$, and so an isometry is simply a map $f\colon Y\to Y'$ satisfying a condition (namely, the existence of the $f_{i'i}$'s satisfying the requirements above). We note that these conditions do not imply that $f$ is surjective or injective; \eg a $(G,\M)$-structure on $Y$ induces a $(G,\M)$-structure on any open subset $Y'\subset Y$, or covering space $Y'\to Y$ and these maps $Y'\to Y$ are morphisms of $(G,\M)$-manifolds in a natural way. 
 \end{defn}

Next, we define a `family version' or `parametrized version' of the category $\GMMan$.

\begin{defn}\label{def:GM_fam} A {\em family of $(G,\M)$-manifolds} is a smooth map  $p\colon Y\to S$ together with a maximal atlas consisting of charts which are diffeomorphisms $\varphi_i$ between open subsets of $Y$ and open subsets of $S\times\M$ making the following diagram commutative:
\[
\xymatrix{
Y\supseteq U_i\ar[rr]^<>(.5){\varphi_i}_<>(.5)\cong\ar[rd]_p&&V_i\subseteq S\times \M\ar[ld]^{p_1}\\
&S&
}
\]
In addition, there are {\em transition data} which are smooth maps $g_{ij}\colon p(U_i\cap U_j)\to G$ which make the diagrams
\begin{equation}\label{eq:transition}
\xymatrix{
&&U_i\cap U_j\ar[lld]_{\varphi_j}\ar[drr]^{\varphi_i}&&\\
p(U_i\cap U_j)\times\M\ar[rr]^<>(.5){\id\times g_{ij}\times\id}&&p(U_i\cap U_j)\times G\times \M\ar[rr]^{\id\times\mu}&&p(U_i\cap U_j)\times\M
} 
\end{equation}
and 
\[
\xymatrix{
p(U_i\cap U_j\cap U_k)\ar[d]_{g_{ij}\times g_{jk}}\ar[dr]^{g_{ik}}&\\
G\times G\ar[r]^\mu&G
}
\]
commutative. Here $\mu$ is the multiplication map of the Lie group $G$.   We note that the conditions imply in particular that $p$ is a {\em submersion} and $p(U_i\cap U_j)\subseteq S$ is open. 

A {\em family of $(G,\M)$-pairs} is a smooth map $p\colon Y\to S$, a codimension one submanifold $Y^c\subset Y$, and a maximal atlas consisting of charts $\{(U_i,\varphi_i)\}$ for $Y$ as above, with the additional requirement that $\varphi_i$ maps $U_i\cap Y^c$ to $S\times \M^c\subset S\times\M$.

If $Y\to S$ and $Y'\to S'$ are two families of $(G,\M)$-manifolds, a {\em morphism} from $Y$ to $Y'$ consists of the following data:
\begin{itemize}
\item  a pair of maps $(f,\hat f)$ making the following diagram commutative:
\[
\xymatrix{
Y\ar[d]\ar[r]^{\hat f}&Y'\ar[d]\\
S\ar[r]^f&S'
}
\]
\item a smooth map $f_{i'i}\colon p(U_i)\to G$ for each pair of charts $(U_i,\varphi_i)$ of $Y$ respectively $(U'_{i'},\varphi_{i'})$ of $Y'$ with $\wh f(U_i)\subset U'_{i'}$ making the diagrams
\[
\xymatrix{
U_i\ar[rrr]^{\wh f}\ar[d]_{\varphi_i}&&&U'_{i'}\ar[d]^{p_2\circ \varphi_{i'}}\\
p(U_i)\times\M\ar[rr]^<>(.5){f_{i'i}\times\id}&&G\times\M\ar[r]^<>(.5)\mu&\M
}
\]
and
\[
\xymatrix{
p(U_i\cap U_j)\ar[rrr]^{\id\times f_{j'j}}\ar[dd]_{\id\times g_{ij}}&&&
p(U_i\cap U_j)\times G\ar[d]^{g'_{i'j'}\times\id}\\
&&&G\times G\ar[d]^\mu\\
p(U_i\cap U_j)\times G\ar[rr]^<>(.5){f_{i'i}\times\id}&&G\times G\ar[r]^<>(.5)\mu&G
}
\]
commutative. 
\end{itemize}

Abusing notation, we will usually write $\GMMan$ for this family category, but we use the notation $\GMMan/\Man$ if we want to emphasize that we talk about the family version.
\end{defn}
\subsection{Categories with flips}\label{subsec:cat_flip}

Let $(G,\M)$ be a geometry in the sense discussed in the previous subsections. We note that an element $g$ in the center of $G$ determines an automorphism $\theta_Y\colon Y\to Y$ for any $(G,\M)$-manifold $Y$. This automorphism is induced by multiplication by $g$ on our model space $\M$ (more precisely, in terms of Definition \ref{def:GM_morphism} it is given by setting $f_{i'i}=g$ for every $i,i'$). For example, for $G=\Iso(\E^{d|0})=\E^d\rtimes Spin(d)$, the center of $G$ is $\{\pm 1\}\subset Spin(d)$. 
If $Y$ is a $(G,\E^d)$\nb-manifold, \ie a Euclidean manifold with spin structure (see Example \ref{ex:euclidean_spin}), then $\theta_Y$ is multiplication by $-1\in Spin(d)$ on the principal $Spin(d)$\nb-bundle $Spin(Y)\to Y$ (in particular, it is the identity on the underlying manifold). As in our previous paper \cite{HoST} we will refer to $\theta_Y$ as `spin-flip'. 

Let $E$ be a Euclidean spin field theory in the sense of the preliminary Definition \ref{def:prelim}, , \ie a functor 
\[
E\colon \EB{d|0}\to \TV.
\]
of categories internal to  $\SymGrp$, the strict $2$\nb-category of symmetric monoidal group\-oids. Here $\EB{d|0}$ is the variant of the Euclidean bordism category $\EB{d}$  where `Euclidean structures' are replaced by `Euclidean spin structures' (see Example \ref{ex:euclidean_spin}). Then $E$ determines in particular a symmetric monoidal functor
\begin{equation}\label{eq:flip_example}
E_0\colon \EB{d|0}_0\ra \TV_0
\end{equation}
which we can apply to an object $Y$ of that category. Thought of as a morphism $\theta_Y\in \EB{d|0}_0(Y,Y)$, the spin-flip induces an involution on the vector space $E(Y)$; \ie $E(Y)$ becomes a {\em super vector space} (for notational simplicity we drop the subscript $E_0$). If $Y'$ is another object of the bordism category, then the symmetric monoidal functor $E$ gives a commutative diagram
\[
\xymatrix{
EY\otimes EY'\ar[r]\ar[d]_{\sigma_{EY,EY'}}^\cong&E(Y\amalg Y')\ar[d]^{E(\sigma_{Y,Y'})}_\cong\\
EY'\otimes EY\ar[r]&E(Y'\amalg Y)
}
\]
Here $\sigma$ is the braiding isomorphism (in both categories), and the horizontal isomorphisms are part of the data of the symmetric monoidal functor $E$ (see \cite[Ch.\ XI, \S 2]{McL}). Unfortunately, this is {\em not} what we want, and this reveals another shortcoming of the preliminary definition \ref{def:prelim}. 
We should emphasize that $\sigma_{EY,EY'}$ is the braiding isomorphism in the category of {\em ungraded} vector spaces given by 
\[
E(Y)\otimes E(Y')\ra E(Y')\otimes E(Y)
\qquad
v\otimes w\mapsto w\otimes v
\]
What we want is the commutativity of the above diagram, but with $\sigma_{EY,EY'}$ being the braiding isomorphism in the category of super vector spaces given by 
\[
v\otimes w\mapsto (-1)^{|v||w|}w\otimes v
\]
 for homogeneous elements $v,w$ of degree $|v|,|w|\in \Z/2$. That suggests to replace the symmetric monoidal category $\TV_0$ by its `super version' consisting of $\Z/2$\nb-graded topological vector spaces, the projective tensor product, and the desired braiding isomorphism of super vector spaces. However, this doesn't solve the problem: $E(Y)$ now has two involutions: its grading involution $\theta_{EY}$ as super vector space and the involution $E(\theta_Y)$ induced by the spin-flip $\theta_Y$. So we {\em require} that these two involutions agree, as in \cite[Definition 6.44]{HoST}. We will refer to $\theta_Y$ as a {\em flip} since the terminology `twist' used in the context of balanced monoidal categories  \cite[Def.\ 6.1]{JS} unfortunately conflicts with our use of `twisted field theories'.

\begin{rem}\label{rem:spin-statistic}
The requirement that the grading involution on the quantum space space is induced by the spin flip of the world-sheet is motivated by the $1|1$-dimensional $\Sigma$-model with target a Riemannian spin manifold $X$. It turns out that the flip on the world-sheet $\R^{1|1}$ quantizes into the grading involution on the quantum state space $\Gamma(X;S)$, the spinors on $X$. 
\end{rem}
\begin{defn}\label{def:cat_flip} A {\em flip} for a category $\sC$ is a natural family of isomorphisms 
\[
\theta_Y\colon Y\overset\cong\ra Y
\]
for $Y\in \sC$. If $\sC$, $\sD$ are categories with flips, a functor $F\colon \sC\to \sD$ is {\em flip-preserving} if $F(\theta_Y)=\theta_{FY}$. 

If $F,G\colon \sC\to \sD$ are two flip-preserving functors, and $N\colon F\to G$ is a natural transformation, we note that the commutativity of the diagram
\[
\xymatrix{
FY\ar[r]^{NY}\ar[d]_{\theta_{FY}}&GY\ar[d]^{\theta_{GY}}\\
FY\ar[r]^{NY}&GY
}
\]
is automatic due to $\theta_{FY}=F(\theta_Y)$, $\theta_{GY}=G(\theta_Y)$. In other words, we don't need to impose any restrictions on natural transformations (other than being natural transformation between flip-preserving functors) to obtain a strict $2$\nb-category $\Cat^\fl$ whose objects are categories with flips, whose morphisms are flip preserving functors, and whose $2$\nb-morphisms are natural transformations. 
\end{defn}

Here are two basic examples of categories with flips.

\begin{ex}\label{ex:SVect} Let $\SVect$ be the category of super vector spaces.  An object of $\SVect$ is a vector space $V$ equipped with a `grading involution' $\theta_V\colon V\to V$ (which allows us to write $V=V^{ev}\oplus V^{odd}$, where $V^{ev}$ (respectively $V^{odd}$ is the $+ 1$\nb-eigenspace (respectively $-1$\nb-eigenspace of $\theta_V$). Morphisms from $V$ to $W$ are linear maps $f\colon W\to V$ compatible with the grading involutions in the sense that $f\circ \theta_V=\theta_W\circ f$. In particular, $\theta_V$ is a morphism in $\SVect$, and hence $\SVect$ is a {\em category  with flip}. 

In fact, this is a {\em symmetric monoid} in the strict $2$\nb-category $\Cat^\fl$ of categories with flip. For $V,W\in \SVect$, the tensor product $V\otimes W$ is the usual tensor product of vector spaces equipped with grading involution $\theta_{V\otimes W}=\theta_V\otimes\theta_W$, and the braiding isomorphism is described above. To check compatibility of the symmetric monoidal structure with the flip, we only need to check that the {\em functors} defining the symmetric monoid in $\Cat$ are compatible with the flip (as discussed above, there are no compatibility conditions for the natural transformations). These two functors are 
\begin{itemize}
\item the tensor product
\[
c\colon \SVect\times\SVect\ra \SVect
\qquad (V,W)\mapsto V\otimes W
\]
which is compatible with flips by construction of the grading involution,
\item the unit functor $u$ from the terminal object of $\Cat$, the discrete category with one object, to $\SVect$. The functor $u$ maps the unique object to the monoidal unit $\one\in\SVect$, which is the ground field with the trivial involution. In particular, $\theta_{\one}=\id_{\one}$, which is one way of saying that the functor $u$ preserves the flip.
\end{itemize}
\end{ex}

\begin{ex}\label{ex:g-flip} Let $(G,\M)$ be a rigid geometry, and let $g$ be an element of the center of $G$ which acts trivially on the model space $\M$. Then as discussed at the beginning of this section, $g$ determines an automorphism $\theta_Y$ for families of  $(G,\M)$-manifolds $Y$. In other words, $g$ determines a flip $\theta$ for the category $\GMMan$. The functor $p\colon \GMMan\to \Man$ is flip preserving if we equip $\Man$ with the trivial flip given by $\theta_M=\id_M$. This gives in particular a flip for the categories
$\EB{d|0}_i, i=0,1$, in our motivating example \eqref{eq:flip_example}; using our new terminology we want to require that $E_i$ are symmetric monoidal functors preserving the flip.
\end{ex} 
\subsection{Fibered categories}\label{subsec:fibered}

Before defining the family versions of the bordism categories $\GMBord$ associated to a geometry $(G,\M)$ in the next section, we recall in this section the notion of a 
{\em Grothendieck fibration}. An excellent reference is \cite{Vi}, but we recall the definition for the convenience of the reader who is not familiar with this language. Before giving the formal definition, it might be useful to look at an example. Let $\Bun$ be the category whose objects are smooth fiber bundles $Y\to S$, and whose morphisms from $Y\to S$ to $Z\to T$ are smooth maps $f$, $\phi$ making the diagram 
\begin{equation}\label{eq:bundlemap}
\xymatrix{
Y\ar[r]^{\phi}\ar[d]&Z\ar[d]\\
S\ar[r]^f&T
}
\end{equation}
commutative. 
Let us consider the forgetful functor 
\begin{equation}\label{eq:forgetfunctor}
p\colon \Bun\ra \Man
\end{equation}
which sends a bundle to its base space. We note that if $Z\to T$ is a bundle, and $f\colon S\to T$ is a smooth map, then there is a pull-back bundle $f^*Z\to S$, and a tautological  morphism of bundles $\phi\colon Y=f^*Z\to Z$ which maps to $f$ via the functor $p$. The bundle morphism $\phi$ enjoys a universal property called {\em cartesian}, which  more generally can be defined for any morphism $\phi\colon Y\to Z$ of a category $\sB$ equipped with a functor $p\colon \sB\to \sS$ to another category $\sS$.  In the following diagrams, an arrow going from an object $Y\in \sB$ to an object  $S\in\sS$, written as $Y\mapsto S$, will mean  $p (Y) = S$. Furthermore, the commutativity of the diagram
   \begin{equation}\label{eq:square}
   \xymatrix{{}Y\ar@{|->}[d]\ar[r]^\phi &{}Z\ar@{|->}[d] \\
   S\ar[r]^f & T}
   \end{equation}
will mean $p (\phi )= f$.

\begin{defn} \label{def:cartesian} Let $p\colon\sB\to\sS$ be a functor. An arrow  $\phi\colon Y \to Z$ of $\sB$ is \emph{cartesian} if for any arrow $\psi\colon X \to Z$ in $\sB$ and any arrow $g \colon p(X) \to p (Y)$ in $\sS$ with $p (\phi) \circ  g = p( \psi)$, there exists a unique arrow $\theta \colon X \to Y$ with $p (\theta)  = g$ and $\phi\circ \theta = \psi$, as in the commutative diagram
   \[
   \xymatrix@R=6pt{
   {X}\ar@{|->}[dd] \ar@{-->}[rd]_{\theta}
   \ar@/^/[rrd]^<>(.6){\psi} \\
   & {}Y\ar@{|->}[dd]\ar[r]_{\phi}
   & {}Z\ar@{|->}[dd] \\
   {}R\ar[rd]_g\ar@/^/[rrd]|(.487)\hole^<>(.6)h\\
   & {}S \ar[r]_f
   & {}T  
   }
   \]
If $\phi\colon Y\to Z$ is cartesian, we say that the diagram \eqref{eq:square} is a {\em cartesian square}.
 \end{defn}

In our example of the  forgetful functor $p\colon \Bun\to \Man$, a bundle morphism $(f,\phi)$ as in diagram \eqref{eq:bundlemap} is cartesian \Iff $\phi$ is a fiberwise diffeomorphism. In particular, the usual pullback of bundles provides us with many cartesian squares. This implies that the functor $p$ is a {\em Grothendieck fibration} which is defined as follows.

\begin{defn}\label{def:fibered} A functor $p\colon \sB\to \sS$ is a {\em (Grothendieck) fibration}  if pull-backs exist: for every object $Z\in \sB$ and every arrow $f\colon S\to T=p(Z)$ in $\sS$, there is a cartesian square
  \[
   \xymatrix{{}Y\ar@{|->}[d]\ar[r]^\phi &{}Z\ar@{|->}[d] \\
   S\ar[r]^f & T}
   \]

A \emph{fibered category over $\sS$} is a category $\sB$ together with a functor $p\colon \sB\to \sS$ which is a fibration. If $p_\sB\colon \sB\to \sS$ and $p_\sC\colon \sC\to \sS$ are fibered categories over $\sS$, then a \emph{morphism of fibered categories} $F \colon \sB \to \sC$ is a base preserving functor ($p_\sB \circ F = p_{\sC}$) that sends cartesian arrows to cartesian arrows.

There is also a notion of base-preserving natural transformation between two morphisms from $\sB$ to $\sC$. These form the $2$\nb-morphisms of a strict $2$\nb-category $\Cat/\sS$ whose objects are categories fibered over $\sS$ and whose morphisms are morphisms of fibered categories. 
\end{defn}

\begin{rem} \label{rem:stack} There is a close relationship between categories fibered over a category $\sS$ and pseudo-functors $\sS^{op}\to \Cat$ (see \cite[Def.\ 3.10]{Vi} for a definition of `Pseudo-functor'; from an abstract point of view, this is a $2$\nb-functor, where we interpret $\sS^{op}$ as a $2$\nb-category whose only $2$\nb-morphisms are identities). Any pseudo-functor determines a fibered category with an extra datum called `cleavage', and conversely, a category $\sC$ fibered over $\sS$ with cleavage determines a pseudo-functor $\sS^{op}\to \Cat$ (see \cite[\S 3.1.2 and \S 3.1.3]{Vi}). Since the `space of cleavages' of a fibered category is in some sense contractible, it is mostly a matter of taste which language to use (see discussion in last paragraph of \S 3.1.3 in \cite{Vi}). 
We prefer the language of fibered categories. An advantage of the Pseudo-functor approach is the definition of a `stack' is a little easier, but also this can be done in terms of fibered categories (see \cite[Ch.\ 4]{Vi}).
\end{rem}

In our example $\Bun\to \Man$ of a fibered category, there is a symmetric monoidal structure we haven't discussed yet: if $Y\to S$ and $Z\to S$ are fiber bundles over the same base manifold $S$, we can form the disjoint union $Y\amalg Z\to S$ to obtain a new bundle over $S$. This gives a morphism of fibered categories 
\[
c\colon\Bun\times_\Man\Bun\ra\Bun
\]
which makes $\Bun$ a symmetric monoid in the strict $2$\nb-category $\Cat/\Man$ of categories fibered over $\Man$ (with the monoidal structure given by the categorical product). Of course, the other data $u$, $\alpha$, $\lambda$, $\rho$ of a symmetric monoid (see Def.\ \ref{def:symm_monoid}) need to be specified as well -- this is left to the reader. 

\subsection{Field theories associated to rigid geometries}\label{subsec:GMBord}

The goal of this section is Definition \ref{def:GMFT} of a field theory associated to a `rigid geometry' $(G,\M)$ (see section \ref{subsec:rigid}). We begin by defining the (family) bordism category $\GMBord$. The definition will be modeled on Definition~\ref{def:RB} of the Riemannian bordism category, but replacing Riemannian structures by $(G,\M)$-structures. 
The bordism category $\GMBord$ is a category internal to the strict $2$\nb-category $\Sym(\Cat^\fl/\Man)$ of symmetric monoids in the $2$\nb-category of categories with flip fibered over $\Man$. In fact, $\GMBord_0$ and $\GMBord_1$, the two categories over $\Man$, have a useful additional property: they are {\em stacks} (see Remark \ref{rem:stack}). However, we won't discuss this property here, since it is not needed for the proofs of our results in this paper.

We recall that for any geometry $(G,\M)$, there is a category $\GMMan$ of families of $(G,\M)$-manifolds  $Y\to S$. The obvious forgetful functor
\[
\GMMan\to \Man
\]
 which sends a family to its parameter space is a Grothendieck fibration. Moreover, as discussed in Example \ref{ex:g-flip}, they are flip preserving. Hence it is an object in the strict $2$\nb-category $\Cat^\fl/\Man$ of categories with flip fibered over $\Man$. In fact, like $\Bun\to \Man$, these categories  over $\Man$ are symmetric monoids in $\Cat^\fl/\Man$. The monoidal structure is given by the disjoint union. 
  
 Now we can define the categories $\GMBord_i$ for $i=0,1$, simply by repeating Definition \ref{def:RB} word for word, replacing `Riemannian $d$\nb-manifolds' by `families of $(G,\M)$-manifolds'.

\begin{defn}\label{def:GMBord} An object in ${\GMBord}_0$ is a quadruple $(Y,Y^c,Y^\pm)\to S$ consisting of a family of $(G,\M)$-pairs $(Y,Y^c)$ (see Definition \ref{def:GM_fam}) and a decomposition of $Y\setminus Y^c$ as the disjoint union of subspaces $Y^\pm$. It is required that $p\colon Y^c\to S$ is {\em proper}. This assumption is a family version of  our previous assumption in Definition~\ref{def:RB} that $Y^c$ is compact, since it reduces to that assumption for $S=\pt$. 

A morphism from $(Y_0,Y_0^c,Y_0^\pm)\to S_0$ to $(Y_1,Y_1^c,Y_1^\pm)\to S_1$ is a morphism from $W_0\to S$ to $W_1\to S$ in $\GMMan$. Here $W_j\subset Y_j$ are open neighborhoods of $Y_j^c$, and it is required that this map sends $W_0\cap Y_0^c$ to $Y_1^c$ and $W_0\cap Y_0^\pm$ to $Y_1^\pm$. More precisely, a morphism  is a  {\em germ} of such maps, \ie two such maps represent the {\em same} morphism in $\GMBord_0$ if they agree on some smaller open neighborhood  of $Y_0^c$ in $Y_0$.

An object in $\GMBord_1$ consists of the following data
\begin{enumerate}
\item a manifold $S$ (the parameter space);
\item a pair of objects $(Y_0,Y_0^c,Y_0^\pm)\to S$, $(Y_1,Y_1^c,Y_1^\pm)\to S$ of $\GMBord_0$ over the same parameter space $S$ (the source respectively target);
\item an object $\Sigma\in \GMMan_S$ (\ie an $S$\nb-family of $(G,\M)$-manifolds); 
\item smooth maps  $i_j\colon W_j\to \Sigma$ compatible with the projection to $S$. Here $W_j\subset Y_j$ are open neighborhoods of the cores $Y_j^c$. 
\end{enumerate}
 Letting $i_j^\pm: W_j^\pm \to \Sigma$ be the restrictions of $i_j$ to $W_j^\pm:=W_j\cap Y_j^\pm$, we require that 
\begin{itemize}
\item[(+)] $i_j^+$ are morphisms in $\GMMan_S$ from $W_j^+$ to $\Sigma \smallsetminus i_1(W_1^-\cup Y_1^c)$ and
\item[$(c)$] the restriction of $p_\Sigma$ to the {\em core} $\Sigma^c:=
\Sigma\smallsetminus \left(i_0(W_0^+)\cup i_1(W_1^-)\right)$ is proper.
\end{itemize}
Morphisms in the category $\GMBord_1$ are defined like the morphisms in $\RB{d}_1$ in Definition \ref{def:RB}, except that now the maps $F$, $f_0$, $f_1$ are morphisms in the category $\GMMan$.

The categories $\GMBord_0$, $\GMBord_1$ are fibered over the category $\Man$ via the functor $p\colon \GMBord_i\to \Man$ associating to an object in $\GMMan$ its parameter space $S\in \Man$. This functor preserves the flips (given on $\GMBord_i$ by the choice of an element $g$ in the center of $G$ which acts trivially on $\M$; see Example \ref{ex:g-flip}) and the trivial flip on $\Man$ (\ie $\theta_S=\id_S$ for all $S\in \Man$). This functor is also a Grothendieck fibration, and hence 
it is a  symmetric monoid in the strict $2$\nb-category $\Cat^\fl/\Man$. The monoidal structure is given by the disjoint union. As explained in Definition \ref{def:RB} for the Riemannian bordism category $\RB{d}$, the two objects $\GMBord_i$ of the strict $2$\nb-category $\Sym(\Cat^\fl/\Man)$ of symmetric monoids in $\Cat^\fl/\Man$ fit together to give an internal category, which we denote $\GMBord$ (or $\GMBord/\Man$ if we want to emphasize that we are thinking of families here). We call $\GMBord$ the {\em bordism category of $(G,\M)$-manifolds}. In the special case $(G,\M)=(\Iso(\E^d),\E^d)$, we call this internal category the {\em $d$\nb-dimensional Euclidean bordism category}, and write $\EB{d}$ (or $\EB{d}/\Man$).
\end{defn}

Next is the definition of the family version of the internal category $\TV$. Abusing notation we will write $\TV$ for this family version as well, or $\TV/\Man$ if we wish to distinguish it from its non-family version. Like the family bordism category $\GMBord$, the category $\TV$ is internal to the $2$\nb-category $\Sym(\Cat^\fl/\Man)$. At first it seemed natural to us to think of a `family of vector spaces' as a vector bundle over the parameter space. Later we noticed that we should let go of the local triviality assumption for reasons outlined in Remark \ref{rem:sheaves}.

\begin{defn}\label{def:famTV}  The internal category $\TV$ consists of categories $\TV_0$, $\TV_1$ with flip fibered over $\Man$. The categories $\TV_0$, $\TV_1$ are defined as follows. 
\begin{itemize}
\item[$\TV_0$] An object of $\TV_0$ is a manifold $S$ and a sheaf $V$ over $S$ of (complete, locally convex) $\Z/2$\nb-graded topological modules over the structure sheaf $\scr O_S$ (of smooth functions on $S$).
A morphism from a sheaf $V$ over $S$ to a sheaf $W$ over $T$ is a smooth map $f\colon S\to T$ together with a continuous $\scr O_T(U)$\nb-linear map $V(f^{-1}(U))\to W(U)$ for every open subset $U\subset T$ (here  $\scr O_T(U)$ acts on $V(f^{-1})$ via the algebra homomorphism $f^*\colon \scr O_T(U) \to \scr O_S(f^{-1}(U))$. As in the category of super vector space $\SVect$ (see Example \ref{ex:SVect}), the flip is given by the grading involution of the sheaf $V$. 
\item[$\TV_1$] An object of $\TV_1$ consists of a manifold $S$, a pair of sheaves of topological $\scr O_S$-modules $V_0$, $V_1$ and an $\scr O_S$\nb-linear map of sheaves $V_0\to V_1$. We leave the definition of morphisms as an exercise to the reader. 
\end{itemize}
Both of these fibered categories with flip are symmetric monoids in $\Cat^\fl/\Man$; the monoidal product is given by the projective tensor product over the structure sheaf. 
 \end{defn}

\begin{defn}\label{def:GMFT}
A $(G,\M)$-field theory is a functor
\[
\GMBord\ra \TV
\]
of categories internal to  $\Sym(\Cat^\fl/\Man)$, the strict $2$\nb-category of symmetric mon\-oids in the $2$\nb-category of categories with flip fibered over $\Man$. If $X$ is a smooth manifold, a  {\em $(G,\M)$-field theory over $X$} is a functor 
\[
E\colon \GMBord(X)\ra \TV,
\]
where $\GMBord(X)$ is the generalization of $\GMBord$ obtained by furnishing every $(G,\M)$-manifold with the additional structure of a smooth map to $X$. 
Specializing the geometry $(G,\M)$ to be the  Euclidean geometry $(\Iso(\E^d),\E^d)$, where $\Iso(\E^d)$ is the group of  isometries, we obtain the notion of a Euclidean field theory of dimension $d$. 
\end{defn}

\section{Euclidean field theories and their partition function}
\label{sec:2EFT}
\subsection{Partition functions of $\EFT{2}$'s}
\label{subsec:partition_2EFT}
In this subsection we will discuss Euclidean field theories of dimension $2$ and $2|1$. The most basic invariant of a $2$\nb-dimensional EFT $E$ is its {\em partition function}, which we defined in Definition \ref{def:partition} to be the function $Z_E\colon \fh\ra \C$, $\tau\mapsto E(T_\tau)$,
where $T_\tau=\C/\Z\tau+\Z$. While this definition is good from the point of view that it makes contact with the definition of modular forms as functions on $\fh$, it is not good in the sense that a partition function should look at the value of $E$ on every closed Euclidean $2$\nb-manifold, but not every closed Euclidean $2$\nb-manifold is isometric to one of the form $T_\tau$. 
We observe that every closed oriented Euclidean $2$\nb-manifold is isometric to a torus 
\[
T_{\ell,\tau}:=\ell(\Z\tau+\Z)\backslash \E^2
\]
obtained as the quotient of the subgroup $\ell(\Z\tau+\Z)\subset\E^2\subset\Iso(\E^2)$ acting on $\E^2$ for some $\ell\in \R_+$, $\tau\in \fh$. Here we switch notation since we want to write objects in the bordism categories $\GMBord$ systematically as quotients of the model space $\M$ by a left action of a subgroup of $G$ acting freely on $\M$ (we insist on a {\em left} action since in the case of Euclidean structures on supermanifolds the group $G$ is not commutative). 
Then we extend the domain of $Z_E$ from $\fh$ to $\R_+\times \fh$ by defining
\[
Z_E\colon \R_+\times \fh\ra \C
\qquad
\tau\mapsto E(T_{\ell,\tau}).
\]
 Abusing notation, we will again use the notation $Z_E$ for this extension, and refer to it as `partition function'. 
 
  A torus $T_{\ell,\tau}$ is isometric to $T_{\ell',\tau'}$ \Iff\ $(\ell,\tau)$ and $(\ell',\tau')$ are in the same orbit of the $SL_2(\Z)$\nb-action on $\R_+\times \fh$ given by 
\begin{equation}\label{eq:SL_action}
\left(\begin{matrix}
a&b\\c&d\end{matrix}\right)
(\ell,\tau)=\left(\ell |c\tau+d|,\frac{a\tau+b}{c\tau+d}\right)
\end{equation}
If we forget about the first factor, the quotient stack $SL_2(\Z)\doublebackslash \fh$ has the well-known interpretation as the moduli stack of {\em conformal structures} on pointed tori. Similarly, the moduli stack of {\em Euclidean structures} on pointed tori can be identified with the quotient stack $SL_2(\Z)\doublebackslash \left(\R_+\times \fh\right)$. As in the conformal situation, the product $\R_+\times \fh$ itself can be interpreted as the moduli space of Euclidean tori furnished with a basis for their integral first homology. Then the $SL_2(\Z)$\nb-action above corresponds to changing the basis.

What can we say about the partition function $Z_E$? First of all, it is a {\em smooth} function. To see this, we note that the tori $T_{\ell,\tau}$ fit together to a smooth bundle 
\begin{equation}\label{eq:univ_bundle}
p\colon \Sigma\to \R_+\times \fh
\end{equation}
 with a fiberwise Euclidean structure, such that the fiber over $(\ell,\tau)\in \R_+\times \fh$ is the Euclidean torus $T_{\ell,\tau}$.  Applying $E$ to this smooth family results in a smooth function on the parameter space $\R_+\times \fh$. Compatibility of $E$ with pullbacks guarantees that this is the function $Z_E$. 
Secondly, the partition function $Z_E$ is {\em invariant under the $SL_2(\Z)$\nb-action}, since if $(\ell,\tau)$, $(\ell',\tau')$ are in the same orbit, then, as mentioned above,  the tori $T_{\ell,\tau}$ and $T_{\ell',\tau'}$ are isometric, and
hence $E(T_{\ell,\tau})=E(T_{\ell',\tau'})$.

\begin{rem}\label{rem:ell_independence}
If the function $Z_E(\ell,\tau)$ is  independent of $\ell\in \R_+$, then the invariance of $Z_E$ under the $SL_2(\Z)$\nb-action on $\R_+\times\fh$ implies that $Z_E(1,\tau)$ has the transformation properties of a modular form of weight zero. This is the case \eg if $E$ is a {\em conformal} field theory, since the {\em conformal} class of the torus $T_{\ell,\tau}$ is independent of the scaling factor $\ell$. If $E$ is not conformal, there is no reason to expect $Z_E(\ell,\tau)$ to be independent of $\ell$, and hence no reason for $Z_E(1,\tau)$ to be invariant under the $SL_2(\Z)$\nb-action. 
Similarly, even if $E$ is a conformal theory, one shouldn't expect $Z_E(1,\tau)$ to be a {\em holomorphic} function, unless $E$ is {\em holomorphic} in the sense that the operators associated to any bordism $\Sigma$ depend holomorphically on the parameters determining the conformal structure on $\Sigma$. A precise definition of a holomorphic theory can be given in the terminology of this paper by working with families of conformal bordisms parametrized by complex instead of smooth manifolds. 

However, as we will see in the proof of Theorem \ref{thm:modular}, if $E$ has an extension to a supersymmetric Euclidean field theory of dimension $2|1$, then the function $Z_E(\ell,\tau)$ is independent of $\ell$ and holomorphic in $\tau$. 
\end{rem}

As a step towards our proof of Theorem \ref{thm:modular} we will prove the following result in the next section.

\begin{prop}\label{prop:partition_2EFT} Let $f\colon \fh\to \C$ be a $SL_2(\Z)$\nb-invariant holomorphic function, meromorphic at infinity with $q$\nb-expansion $f(\tau)=\sum_{k=-N}^\infty a_kq^k$ with non-negative integral Fourier coefficients $a_k$. Then $f$ is the partition function of a $\CEFT{2}$, a conformal Euclidean field theory (see right below for a definition). 
\end{prop}
 
\subsection{Generators and relations of $\EB{2}$}\label{subsec:generators}

A representation of a group $G$ can be thought of as a functor $G\to \Vect$ from $G$, viewed as a groupoid with one object whose automorphism group is $G$, into the category of vector spaces. Hence a field theory, a functor from a bordism category to the category of (topological) vector spaces,  can be thought of as a `representation' of the bordism category. In the same way a presentation of $G$ in terms of generators and relations is helpful when trying to construct a representation of $G$, a `presentation' of the bordism category  is helpful for the construction of  field theories.

In this section we will construct a presentation of $\CEB{2}$, the {\em conformal (oriented) Euclidean bordism category}, a variant of $\EB{2}$ based on the geometry 
\[
(\E^2\rtimes(SO(2)\times\R_+),\E^2) \cong (\C \rtimes \C^\times, \C)
\]
where $\ell\in \R_+$ acts on $\E^2$ by multiplication by $\ell$.  The reason for our interest in $\CEB{2}$ is that it is simpler to write down a presentation for $\CEB{2}$ than for $\EB{2}$ (see Proposition \ref{prop:realizing_EFT}). Also, every conformal (oriented) Euclidean field theory gives an (oriented) Euclidean field theory by precomposing with the obvious functor between bordism categories. We will use $\EFT{2}$'s of this type to prove Proposition \ref{prop:partition_2EFT}.
We begin by describing particular objects of the categories $\EB{2}_0$ and $\EB{2}_1$.

\medskip
\noindent{\bf The circle $K_\ell\in \EB{2}_0$}. 
 The core of this object is the circle of length $\ell>0$, which we prefer to think of as $\ell\Z\backslash \E^1$. The collar neighborhood is $Y=\ell \Z\backslash\E^2\supset \ell\Z\backslash\E^1=Y^c$; the complement $Y\setminus Y^c$ decomposes as the disjoint union of $Y^+=\ell \Z\backslash\E^2_+$ and $Y^-=\ell \Z\backslash\E^2_-$, where $\E^2_\pm$ is the upper (respectively lower) open half plane. The group $\ell Z$ acts on $\E^2$ via the embeddings $\ell\Z\subset\R\subset\E^2\subset \Iso(\E^2)$. We note that there are simpler ways to describe this object,  but it is this description that generalizes nicely to the case of Euclidean supermanifolds of dimension $2|1$. 
Below is a picture of the object $S^1_\ell$.
\begin{figure}[h]
\begin{center}
\scalebox{.80}{\includegraphics{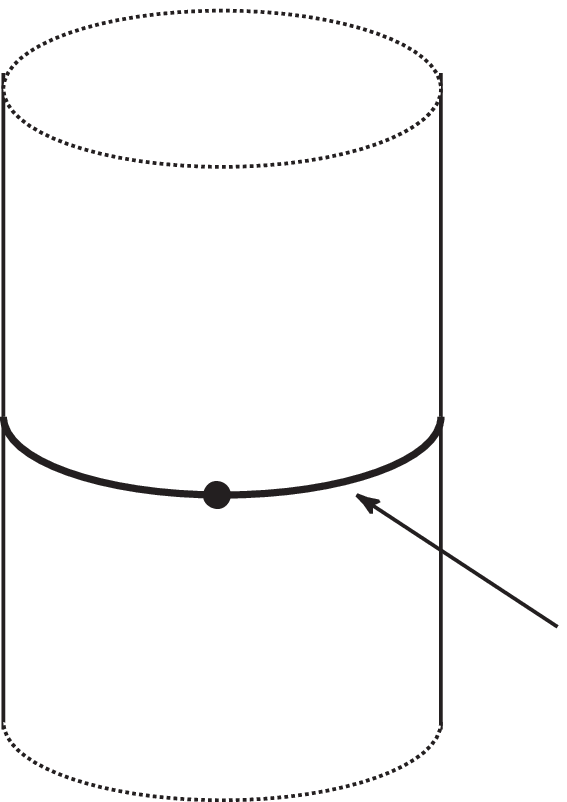}}
 \put(-90,20)	{$\ell \Z\backslash\E^2_-$}
\put(-82,55)	{$0$}
\put(-90,110)	{$\ell \Z\backslash\E^2_+$}
 \put(0,25)	{$\ell \Z\backslash\E^1$}
  \caption{{\bf The object $K_\ell$ of $\EB{2}_0$}}
\label{fig:K}
\end{center}
\end{figure}

\medskip

\noindent{\bf The cylinder  $C_{\ell,\tau}\in \EB{2}_1(K_\ell,K_\ell)$}.
We recall that an object of $\EB{2}_1$ is a pair $Y_0$, $Y_1$ of objects of $\EB{2}_0$ and bordism from $Y_0$ to $Y_1$, \ie a triple
\[
(\xymatrix@1{W_1\ar[r]^{i_1}
&\Sigma&W_0\ar[l]_{i_0}})
\]
where $\Sigma$ is a Euclidean $d$\nb-manifold, $W_j$ is a neighborhood of $Y_j^c\subset Y_j$ for $j=0,1$, and $i_0$, $i_1$ are local isometries such that certain conditions are satisfied (see Definition \ref{def:RB} and figure \ref{fig:bordism}). We make $C_{\ell,\tau}$ precise as an object of $\EB{2}_1$ by declaring it to be the following bordism from $K_\ell$ to itself:
\[
C_{\ell,\tau}:=\left(\xymatrix@1{\ell \Z\backslash\E^2\ar[r]^{\id}
&\ell \Z\backslash\E^2&\ell \Z\backslash\E^2\ar[l]_{\ell\tau}}\right),
\]
where $\ell\tau\in\fh\subset\E^2\subset\Iso(\E^2)$ induces an isometry on the quotient $\ell \Z\backslash\E^2$ since it commutes with $\ell\in \Iso(\E^2)$, see figure \ref{fig:bordism}.

\begin{figure}[h]
\begin{center}
\scalebox{.60}{\includegraphics{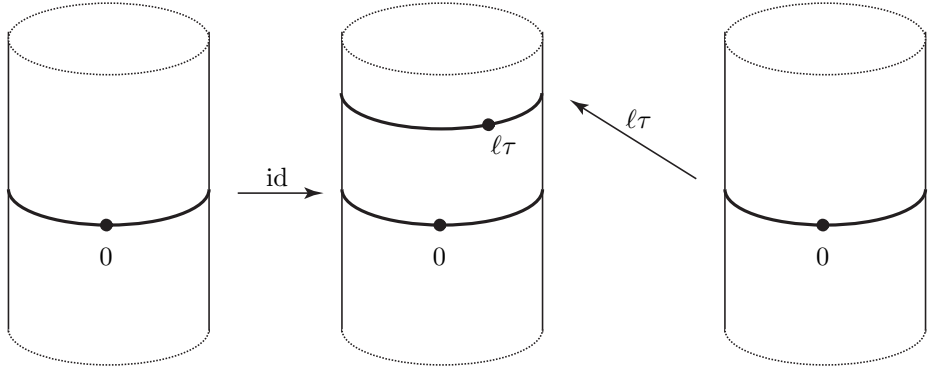}}
 \put(-165,80)	{$\ell\tau$}
\put(-42,38)	{$0$}
\put(-187,38)	{$0$}
\put(-313,38)	{$0$}
\put(-114,90)	{$\ell\tau$}
 \put(-250,68)	{$\id$}
  \caption{{\bf The object $C_{\ell,\tau}\in \EB{2}_1(K_\ell,K_\ell)$}}
\label{fig:C}
\end{center}
\end{figure}

\medskip

\noindent{\bf The left cylinder $L_{\ell,\tau}\in \EB{2}_1(K_\ell\amalg K_\ell,\emptyset)$}:
\[
L_{\ell,\tau}=\left(\xymatrix@1{\emptyset\ar[rr]
&&\ell \Z\backslash\E^2&&\ell \Z\backslash\E^2\,\amalg \,\ell \Z\backslash\E^2 \ar[ll]_<>(.5){\ell\tau\amalg-I}}\right),
\]
where $-I=\left(\begin{smallmatrix} -1&0\\0&-1\end{smallmatrix}\right)\in SO(2)\subset \Iso(\E^2)$. The terminology `left' is motivated by reading bordisms from right to left, \ie drawing the domain of the bordism on the right side, and its range on the left, as in figure \ref{fig:bordism}.

\medskip
\noindent{\bf The right cylinder $R_{\ell,\tau}\in \EB{2}_1(\emptyset,K_\ell\amalg K_\ell)$}.  
\[
R_{\ell,\tau}=\left(\xymatrix@1{\ell \Z\backslash\E^2\, \amalg\, \ell \Z\backslash\E^2\supset W_1\ar[rrr]^<>(.5){\ell\tau\circ(-I)\amalg I}
&&&\R^2/\ell \Z&&\emptyset\ar[ll]}\right),
\]
\medskip
Here, using the terminology of Definition \ref{def:RB},  $W_1$ is an open neighborhood of 
\[
Y_1^c=\ell \Z\backslash\E^1\,\amalg\,\ell \Z\backslash\E^1 \quad\subset\quad \ell \Z\backslash\E^2\,\amalg \,\ell \Z\backslash\E^2=Y_1
\]
 which needs to be chosen carefully in order to satisfy conditions ($+$) and (c). The following choice works: 
\[
W_1:= \ell \Z\backslash(\R\times(-\infty,\epsilon))\,\amalg \,\ell \Z\backslash(\R\times(-\infty,\epsilon))
\quad\subset\quad
\ell \Z\backslash\E^2\,\amalg\, \ell \Z\backslash\E^2=Y_1,
\]
where $\epsilon$ is any real number with $0<\epsilon<\im(\tau)/2$. Note that in particular we need $\im(\tau)>0$ for $R_{\ell,\tau}$ to exists, whereas $L_{\ell,\tau}, C_{\ell,\tau}$ make sense also for $\im(\tau)=0$. 
 
\begin{prop}\label{prop:relations} The following isomorphisms hold in the category $\EB{2}_1$:
\begin{equation}\label{eq:symmetric}
L_{\ell,\tau}\circ \sigma\cong L_{\ell,\tau}
\qquad
\sigma\circ R_{\ell,\tau}\cong R_{\ell,\tau}
\end{equation}
Here $\sigma\colon K_\ell\amalg K_\ell\to K_\ell\amalg K_\ell$ is the isometry switching the two circles, \ie the braiding isomorphism in the symmetric monoidal groupoid $\EB{2}_0$.
\begin{equation}\label{eq:modular}
T_{\ell,\tau}\cong T_{\ell',\tau'}
\qquad\text{if $\ell'=\ell|c\tau+d|$, $\tau'=\frac{a\tau+b}{c\tau+d}\quad$ for some
$\left(\begin{smallmatrix}
a&b\\
c&d
\end{smallmatrix}
\right)\in SL_2(\Z)$}
\end{equation}
\begin{align}
\label{eq:additive}\left(\xymatrix{
\emptyset\ar[rr]^<>(.5){R_{\ell,\tau_1}\amalg R_{\ell,\tau_2}}&&
K_\ell\amalg K_\ell\amalg K_\ell\amalg K_\ell
\ar[rr]^<>(.5){\id\amalg L_{\ell,\tau_3}\amalg\id}&&
K_\ell\amalg K_\ell}
\right)
&\cong
R_{\ell,\tau_1+\tau_2+\tau_3}\\
\label{eq:relationC}\left(
\xymatrix{
\emptyset\amalg K_\ell\ar[rr]^<>(.5){R_{\ell,\tau_1}\amalg\id}&&
K_\ell\amalg K_\ell\amalg K_\ell
\ar[rr]^<>(.5){\id\amalg L_{\ell,\tau_2}}&&
K_\ell\amalg \emptyset=K_\ell}
\right)
&\cong
C_{\ell,\tau_1+\tau_2}\\
\label{eq:relationL}\left(\xymatrix{
K_\ell\amalg K_\ell\ar[rr]^<>(.5){C_{\ell,\tau_1}\amalg \id}&&
K_\ell\amalg K_\ell
\ar[rr]^<>(.5){L_{\ell,\tau_2}}&&
\emptyset}
\right)
&\cong
L_{\ell,\tau_1+\tau_2}\\
\label{eq:relationT}\left(\xymatrix{
\emptyset\ar[rr]^<>(.5){R_{\ell,\tau_1}}&&
K_\ell\amalg K_\ell
\ar[rr]^<>(.5){ L_{\ell,\tau_2}}&&
\emptyset}
\right)
&\cong
T_{\ell,\tau_1+\tau_2}
\end{align}
In the last 4 lines the parameter $\tau$ can be replaced by $\tau+1$, \ie it really lies in $\Z\backslash\fh$.
\end{prop}

These relations show in particular that $T_{\ell,\tau}$, $C_{\ell,\tau}$ and $L_{\ell,\tau}$ can be expressed in terms of $R_{\ell,\tau}$ and $L_{\ell,0}$. In fact {\em every} bordism (\ie  object of $\EB{2}_1$) can be expressed in terms of these. Still, a precise formulation of a `presentation' of $\EB{2}$ is tricky due to the dependence on the scale parameter. It becomes easier when we pass to the conformal Euclidean bordism category $\CEB{2}$, where objects become independent of the scale parameter $\ell$ (and hence we drop the subscript $\ell$).

A $2$\nb-dimensional conformal Euclidean field theory $E\colon \CEB{2}\to \TV$ determines the following data:
\begin{enumerate}[(1)]
\item a topological vector space $V=E(K)$ with a smooth action of $S^1$; 
\item a smooth map $\rho\colon \Z\backslash\fh\to V\otimes V$, $\tau\mapsto E(R_\tau)$;\item a continuous linear map $\lambda=E(L_0)\colon V\otimes V\to \C$
\end{enumerate}
These are subject to the following conditions
\begin{enumerate}[(a)]
\item $\lambda$ and $\rho(\tau)$ are symmetric;
\item the map 
\[
\xymatrix@1{V\otimes V\otimes V\otimes V\ar[rr]^{\id\otimes\lambda\otimes\id}&& V\otimes V\otimes V\otimes V}
\] 
sends $\rho(\tau_1)\otimes \rho(\tau_2)$ to $\rho(\tau_1+\tau_2)$;
\item The action map $\Z\backslash\R\times V\to V$ and the map $\Z\backslash\fh\times V\to V$, $(\tau,v)\mapsto (\gamma(\tau))(v)$ fit together to give a smooth map $\Z\backslash\bar\fh\times V\to V$.  Here $\gamma(\tau)$ is the composition 
\begin{equation}\label{eq:Atau}
\xymatrix@1{
V\cong \C\otimes V\ar[rr]^<>(.5){\rho(\tau)\otimes\id}&&V\otimes V\otimes V\ar[rr]^<>(.5){\id\otimes \lambda}&&V\otimes \C\cong  V
}
\end{equation}
\item the function $\fh\to\C$, $\tau\mapsto \lambda(\rho(\tau))$ is $SL_2(\Z)$\nb-equivariant.
\end{enumerate}
Concerning (1), we note that the automorphism group of $K\in \CEB{2}_1$ is $S^1$ (which acts by rotation); by functoriality, $E(K)$ then inherits an $S^1$\nb-action. The relations (a) are immediate consequences of the relations \ref{eq:symmetric}. Condition  (b) is a consequence of relation \ref{eq:additive}, and condition (d) is a consequence of relation \ref{eq:modular}. Concerning (c) we note that relation \ref{eq:relationC} implies that $\gamma(\tau)=E(C_\tau)$ for $\tau\in \fh$. Unlike the right cylinders $R_\tau$, the cylinders $C_\tau$ are defined for $\tau$ in the {\em closed} upper half plane $\bar\fh$. For $\tau\in\R$, the cylinder $C_\tau\in \CEB{2}_1$ is the image of `rotation by $\tau$' under the canonical map from the endomorphisms of $K$ in $\CEB{2}_0$ to the objects of $\CEB{2}_1(K,K)$. It follows that $\gamma(\tau)$ for $\tau\in \R$ is given by the rotation action on $V$ mentioned in (1). The $C_\tau$'s for $\tau\in \bar\fh$ fit together to form a  smooth family of cylinders parametrized by $\bar\fh$. Applying the functor $E$, we see that 
\begin{equation}\label{eq:Cfamily}
\gamma\colon \bar\fh\times V\ra V
\end{equation}
is a smooth map.

The reader might wonder what we mean here by saying that $\gamma$ is `smooth' since $\bar\fh$ is a manifold with boundary. For this we require that for every smooth map $f\colon S\to \bar\fh\subset\R^2$ (where $S\in \Man$ is a manifold without boundary) the composition with $f$ is smooth. In other words, we are thinking of manifolds with boundary as presheaves on $\Man$. It is a formality to extend a category fibered over $\sS$ to a category fibered over the category of presheaves on $\sS$. Similarly, a morphism of categories fibered over $\sS$ extends canonically to a morphism of these larger categories fibered over the presheaves on $\sS$.

\begin{prop}[\cite{ST5}]\label{prop:realizing_EFT} Given a triple $(V, \lambda, \rho)$ as in (1)-(3) above, satisfying conditions (a)-(d), there is a unique conformal Euclidean field theory $E$ of dimension $2$ which realizes the triple in the sense that $E(K,L_0,R_\tau)=(V, \lambda, \rho(\tau))$.
\end{prop}

The proof -- which we won't give in this paper -- is based on two facts:
\begin{enumerate}
\item every object $(Y,Y^c,Y^\pm)\in\CEB{2}_0$ with connected core $Y^c$ is equivalent to $K$.
\item every object $(\Sigma,i_0,i_1)\in\CEB{2}_1$ with connected core is isomorphic to $C_\tau, L_\tau, R_\tau$ or $T_\tau$ for some $\tau\in \Z\backslash\fh$.
\end{enumerate}
The relations of Proposition \ref{prop:relations} then imply that $E(C_\tau)$ and $E(L_\tau)$ can be reconstructed from the data $\lambda$ and $\rho$. Using the monoidal structure, the functor $E$ is then determined. This is basically the argument in the non-family version.

 \begin{proof}[Proof of Proposition \ref{prop:partition_2EFT}] 
 Let $f\colon \fh\to \C$ be a $SL_2(\Z)$\nb-invariant holomorphic function, meromorphic at infinity with $q$\nb-expansion $f(\tau)=\sum_{k=-N}^\infty a_kq^k$.  We will use Proposition \ref{prop:realizing_EFT} to construct $E\in \EFT{2}$ whose partition function is $f$. We construct the data $V$, $\lambda$, $\rho$ as follows:
 \begin{itemize}
\item $V$ is a completion, described below in Remark \ref{rem:completion}, of the algebraic direct sum 
\[
\bigoplus_{k=-N}^\infty V_k=\{(v_k)_{k=-N}^\infty\mid v_k\in V_k\}
\]
 Here $V_k=\C^{a_k}$, and the group $S^1$ acts on $V_k$ by letting $q\in S^1$ act by scalar multiplication by $q^k$. 
\item $\lambda\colon V\otimes V\to \C$, $(v_k)\otimes(w_k)\mapsto \sum_{k=-N}^\infty\bra{\bar v_k,w_k}$, where $\bra{\bar v_k,w_k}$ is the Hermitian inner product of $\bar v_k$, the complex conjugate of $v_k$, and $w_k$ (this depends $\C$\nb-linearly on $v_k$ and $w_k$).
\item $\rho\colon \fh\to V\otimes V$, $\tau\mapsto \sum_{k=-N}^\infty \sum_{i=1}^{a_k}q^ke_i\otimes e_i$, where $\{e_i\}$ is the standard basis of $\C^{a_k}$.
\end{itemize}
It is straightforward to check that conditions (a)-(d) are satisfied, and hence there is a $2$\nb-dimensional conformal Euclidean field theory $E$ that realizes these data. Let us calculate the partition function of $E$:
\[
Z_E(\tau)=E(T_\tau)=E(L_0\circ R_\tau)=E(L_0)\circ E(R_\tau)=\lambda(\rho)=\sum_{k=-N}^\infty a_kq^k=f(\tau)
\]
\begin{rem}\label{rem:completion}
It is natural to complete the algebraic direct sum $\bigoplus_{k=-N}^\infty V_k$ to the Hilbert space direct sum 
\[
H=\{v=(v_k)\in \prod_k V_k\mid ||v||:=\sum_k |v_k|^2<\infty\}
\]
Unfortunately if we do this, the action map $S^1\times H\to H$ is not smooth: its derivative at $\tau=0$ is the operator $N\colon H\to H$ which is multiplication by $2\pi ik$ on $V_k$. In particular, the eigenvalues of $N$ are unbounded (except if $f$ is a Laurent polynomial which only happens if $f$ is a constant modular function), and hence $N$ is not continuous.  At first glance it seems that this is a problem no matter which topology on $V$ we pick: if the eigenvalues of $N$ are unbounded, we expect that $N$ is not continuous. However, this is not the case if we leave the world of Banach spaces behind: \eg the operator $\frac{d}{d\theta}$ acting on $C^\infty(S^1)$ equipped with the Fr\'echet topology is continuous despite its eigenvalues being unbounded.  

This example suggests to define 
\[
V:=\left\{v=(v_k)\in\prod_k V_k\mid ||v||_n:=\sum_k |v_k|^2k^{2n}<\infty \,\forall n\in \N\right\}
\]
equipped with the Fr\'echet topology determined by the semi-norms $||\quad||_n$.  It is clear that $N$ is continuous in this topology and it can be shown that the action map $\bar\fh\times V\to V$ is smooth using this topology on $V$. 
\end{rem}
\end{proof}

\begin{rem}\label{rem:sheaves}
Unlike infinite dimensional separable Hilbert spaces, there are many non-isomorphic topological vector spaces constructed this way (for a complete isomorphism classification of these spaces see \cite[Prop.\ 29.1]{MV}). For a $\EFT{2}$ $E$ the  isomorphism class of the topological vector space $E(S^1_\ell)$ is an invariant of $E$, but we don't want it to be an invariant for the {\em concordance class of $E$}. This forces us to consider {\em sheaves} rather than locally trivial bundles of topological vector spaces as the appropriate notion of `family of topological vector spaces' (see Definition \ref{def:famTV}).
\end{rem}
\subsection{Partition functions of spin $\EFT{2}$'s}
The goal of this subsection is to show that every integral holomorphic modular function is the partition function of a {\em spin} $\EFT{2}$. Our interest in spin EFT's comes from 
the fact that they are closely related to $\EFT{2|1}$'s: any $\EFT{2|1}$ $E$ determines a reduced  spin $\EFT{2}$ $\bar E$, and the partition function of $E$ by definition is the partition function of $\bar E$. 

\begin{prop}\label{prop:partition_spin_2EFT} Every integral holomorphic modular function is the partition function of a $\CEFT{2|0}$.
\end{prop}

\begin{proof}
We've described the objects $K_\ell$, $T_{\ell,\tau}$, $C_{\ell,\tau}$, etc as quotients of (subsets of) our model space $\sM=\E^2$. These quotients are orbit spaces for a subgroup $H$ of our symmetry group $G=\R^2\rtimes SO(2)$ which acts freely on $\sM$. This subgroup $H$ is contained in the translation group $\R^2\subset G$ and is isomorphic to $\Z^2$ for $T_{\ell,\tau}$, and isomorphic to $\Z$ in the other cases. 

We recall that in the case of Euclidean spin-structures, the model space is still $\E^2$, but the symmetry group now is $\E^2\rtimes Spin(2)$, which acts on $\E^2$ via the double covering map $\E^2\rtimes Spin(2)\to \R^2\rtimes SO(2)$. In particular, the action of the symmetry group on the model space is no longer effective which makes understanding this structure more challenging. For the case at hand, it is helpful to think of the group $\E^2\rtimes Spin(2)$ as acting not only on $\E^2$, but also compatibly on the spinor bundle of $\E^2$. This action is effective: the element $-1\in Spin(2)$ in the kernel of $Spin(2)\to SO(2)$ acts trivially on $\R^2$, but by multiplication by $-1$ on its spinor bundle. 

To specify the spin structure on $\Z\backslash\E^2$, we need to pick a lift of the generator of $\Z\subset\E^2\subset \E^2\rtimes SO(2)$ to $\E^2\rtimes Spin(2)$; in other words, we need to pick an element of $\{\pm 1\}\subset Spin(2)$. The choice $+1$ is usually referred to as `periodic spin structure', since sections of the spinor bundle of the quotient can be interpreted as $\C$\nb-valued functions on $\E^2$ which are periodic (\ie invariant) w.r.t.\ the $\Z$\nb-action. The choice $\{-1\}$ is called `anti-periodic spin structure' since spinors can be identified with functions which are anti-periodic, \ie the generator acts by multiplication by $-1$. For a quotient of the form $\Z^2\backslash\E^2$, we need to specify an element in $\{\pm 1\}$ for {\em both} generators. We will use superscripts $s\in \{\pm\}$ to specify the spin structure, \eg $K_\ell^s$ or $C_{\ell,\tau}^s$ (and in Section~\ref{sec:twisted} we will use the notation $+,-$ for $p,a$). For the torus $T_{\ell,\tau}=\ell(\Z\tau+\Z)\backslash\E^2$ we specify the spin structure by two superscripts: $T_{\ell,\tau}^{s_1s_2}$ (with $s_1$ corresponding to the first generator $\ell\tau$; $s_2$ corresponds to the second generator $\ell$).

Now the proof proceeds as in the previous case: a conformal Euclidean spin field theory determines the following algebraic data
\begin{itemize}
\item a $\Z/2$\nb-graded topological vector space $V^s:=E(K^{s})$ for both spin structures $s=\pm$ with an action of $\R/\Z$ (for $s=+$) respectively $\R/2\Z$ (for $s=-$);
\item a continuous linear map $\lambda^s=E(L_0^s)\colon V\otimes V\to \C$ for $s=\pm$; 
\item a smooth map $\rho^s\colon \fh\to V^s\otimes V^s$ given by $\tau\mapsto E(R_{\tau}^s)$ for $s=\pm$;
\end{itemize}
which are subject to compatibility conditions. Conversely, if these compatibility conditions are satisfied, then these data are realized by a conformal Euclidean spin field theory. 

We recall that from the discussion of the non-spin case that the most interesting compatibility condition came from the transformation properties of the partition function $Z\colon \fh\to \C$, $\tau\mapsto E(T_{\tau})$. This is the only condition that we will formulate and check in the spin case. The analog of the partition function in the spin case is the function 
\[
Z\colon \coprod_{s_1,s_2}\fh^{s_1s_2}\ra \C
\qquad\text{given by}\qquad
\tau\mapsto E(T^{s_1s_2}_{\tau})\quad\text{for}\quad
\tau\in \fh^{s_1s_2}
\]
Here the superscripts ${s_1,s_2}$ is just a way to distinguish the four copies of $\fh$ corresponding to the four spin structures.

In terms of the data above, $Z$ is determined by 

\[
Z(\tau)=
\begin{cases}
\lambda^s((\alpha\otimes\id)c(\rho^s_\tau))&\tau\in \fh^{+s}\\
\lambda^s(c(\rho^s_\tau))&\tau\in \fh^{-s}\\
\end{cases}
\]
where $c\colon V\otimes V\cong V\otimes V$ is the braiding isomorphism, and $\alpha\colon V\to V$ is the grading involution. 
It can be shown that the isometry $T^{s_1s_2}_{\ell,\tau}\to T^{s'_1,s'_2}_{A(\ell,\tau)}$ for $A\in SL_2(\Z)$ is spin structure preserving \Iff 
\[
A\left(\begin{matrix}
s_1\\
s_2
\end{matrix}
\right)
=\left(\begin{matrix}
s_1'\\
s_2'
\end{matrix}
\right)
\]
Here we think of $\left(\begin{smallmatrix}
s_1\\
s_2
\end{smallmatrix}
\right)
$ as a vector in $(\Z/2)^2$ (i.e., identifying $\Z/2$ and $\{\pm\}$). This implies that the function $Z$ is $SL_2(\Z)$\nb-equivariant, where $SL_2(\Z)$ acts on $\fh$ as in \ref{eq:SL_action} and permutes the four copies of $\fh$ in the obvious way. 

Now let $f\colon \fh\to \C$ be a holomorphic modular function with $q$\nb-expansion $f(\tau)=\sum_{k=-N}^\infty a_kq^k$ with non-negative integral Fourier coefficients $a_k$. To find a spin $\EFT{2}$ with partition function $f$, we need to construct data $V^s$, $\lambda^s$ and  $\rho_\tau^s$ as above  such that $Z(\tau)=f(\tau)$ for $\tau\in\fh^{++}$. Starting with $f$, we construct $V$, $\lambda$, $\rho_\tau$ as in the proof of Proposition \ref{prop:partition_2EFT} such that $\lambda(c(\rho_\tau))=f(\tau)$. We note that 
\[
\lambda((\alpha\otimes\id)c(\rho_\tau))=\lambda(c(\rho_\tau)),
\]
 since $V$ is an even vector space thanks to our assumption that the Fourier coefficients of $f$ are non-negative. Hence defining 
 $V^s:=V$, $\lambda^s:=\lambda$, $\rho^s_\tau:=\rho_\tau$ for $s=\pm$, we deduce that $Z(\tau)=f(\tau)$ for $\tau\in \fh^{s_1s_2}$ for {\em all} spin structures $s_1,s_2$. In particular, the function $Z$ is $SL_2(\Z)$\nb-equivariant. Changing the $\Z/2$\nb-grading on $V^+$, while not changing $V^-$, $\lambda^s$ and $\rho^s$ has the effect that  $Z(\tau)$ changes sign for $\tau\in \fh^{++}$, while it doesn't change for $\tau\in \fh^{s_1s_2}$, $s_1s_2\ne ++$. It follows that $-f$ can also be realized as partition function  of a spin $\EFT{2}$. 
 
 To finish the proof, we recall that the ring of integral weakly holomorphic forms is given by 
 \[
 MF^*=\Z[c_4,c_6,\Delta^{\pm 1}]/(c_4^3-c_6^2-1728\Delta)
 \]
 where 
 \[
 c_4=1+240\sum_{k>0}\sigma_3(k)q^k
 \qquad
 c_6=1-504\sum_{k>0}\sigma_5(k)q^k
 \]
 are modular forms of weight $4$ respectively $6$. In particular, every element in $MF^0$ is an integral linear combination of positive powers of $j:=c_4^{3n}/\Delta^n$. We note that $\Delta^{-1}=q^{-1}\prod_{m=1}^\infty\sum_i q^{mi}$ has non-negative Fourier coefficients and hence so does $j$. By our arguments above it follows that $j$ and $-j$ are the partition functions of spin $\EFT{2}$'s. Forming sums and tensor products of field theories, we see that every linear combination of powers of $j$ is realized as the partition function of a spin $\EFT{2}$.
\end{proof}

\subsection{Modularity of the partition function of $\EFT{2|1}$'s, part I}
\label{subsec:modular1}

The proof of Theorem \ref{thm:modular} (according to which the partition function of a $\EFT{2|1}$ is an weakly holomorphic integral modular form of weight zero) consists of two steps. This section provides the first of these steps, namely to show that if $E$ is a spin $\EFT{2}$ (of degree zero) satisfying an additional condition, then the partition function of $E$ is a weakly holomorphic integral modular form of weight zero (see Proposition \ref{prop:modular}). After defining $\EFT{2|1}$'s in section \ref{sec:susy}, we will define the partition function of a $\EFT{2|1}$ as the partition function of the spin $\EFT{2}$ $E$ it determines. We will call $E$ the {\em associated reduced Euclidean field theory}. If $E\in \EFT{2|0}$ is obtained this way, we say that $E$ has a `supersymmetric extension'. We will show in Proposition \ref{prop:square} that if $E$ has a supersymmetric extension, then it satisfies the condition of Proposition \ref{prop:modular}. This is the second and final step in the proof of Theorem \ref{thm:modular}.

Let $E$ be a  spin $\EFT{2}$, and for $s\in \{\pm\}$ let $V^s:=E(K^s_\ell)$ be the topological vector space associated to $K^s_\ell$ (we recall from section \ref{subsec:generators} that $K^s_\ell$ is the object of $\EB{2|0}_0$ whose core $(K^s_\ell)^c$ is the circle of diameter $\ell$ with spin structure $s$; here $s=+$ corresponds to the periodic, and $s=-$ to the anti-periodic spin structure). Applying the functor $E$ to the cylinder $C^s_{\ell,\tau}$, $\tau\in \bar\fh$, (an object in $\EB{2|0}_1(K^s_\ell,K^s_\ell)$; see section  \ref{subsec:generators}), we obtain an operator (\ie a continuous linear map)
\[
A^s_\tau=A^s(\tau):=E(C^s_{\ell,\tau})\colon V^s\ra V^s
\]
We note that for fixed $\ell$ the map $\bar\fh\times V^s\ra V^s$, $(\tau,v)\mapsto A^s_\tau(v)$ is smooth as discussed above (see \eqref{eq:Cfamily}).

\begin{prop}\label{prop:modular} Let $E$ be a spin EFT of dimension $2$. We assume that for every $\tau\in \fh$ there is an odd operator $B_\tau\colon V\to V$ which commutes with $A^+_\tau$ and satisfies
\begin{equation}\label{eq:condition_susy}
\frac{\p A^+_\tau(v)} {\p\bar\tau}=-B_\tau^2(v)
\qquad\forall v\in V.
\end{equation}
Then the partition function $Z_E^{++}\colon \fh^{++}\to \C$ is a holomorphic modular function with integral Fourier coefficients. 
\end{prop}

The proof of this proposition is based on a `supersymmetry cancellation argument' (see proof of Lemma \ref{lem:cancellation} below), an argument that   seems to be fairly standard in the physics literature, at least on the Lie algebra level, \ie for infinitesimal generators of the above families $A_\tau$, $B_\tau$. 

The proof of this result consists of a number of steps that we formulate as lemmas.
\begin{lem} For any spin structure $s\in \{\pm\}$ and $\tau\in \fh$ the operator $A^s_\tau$ is {\em compact}, \ie in the closure of the finite rank operators with the respect to the compact-open topology on the space of  continuous linear maps $V^s\to V^s$.
\end{lem}

\begin{proof}
We use a line of argument developed in our paper \cite{ST2}. A central definition in that paper was the following. 

\begin{defn}\label{def:thick} A morphism $f\colon X\to Y$ in a monoidal category $\sC$ is {\em thick} if it can be factored in the form
\begin{equation}\label{eq:factorization}
\xymatrix{
X\cong \one\otimes X\ar[rr]^{t\otimes \id_X}&&Y\otimes Z\otimes X\ar[rr]^{\id_Y\otimes b}&&Y\otimes\one\cong Y
}
\end{equation}
for morphisms $t\colon \one\to Y\otimes Z$, $b\colon Z\otimes X\to \one$.
\end{defn}
We will consider thick morphisms in  the fibered category $\bEB{2|0}$ over $\Man$ whose objects are the objects of $\EB{2|0}_0$. For $Y_0,Y_1\in \EB{2}_0$ we define the morphism set $\bEB{2|0}(Y_0,Y_1)$ to be the  isomorphism classes of objects in 
$\EB{2|0}_1(Y_0,Y_1)$. We note that the functors $s$, $t$, $u$ and $c$ give $\bEB{2|0}$ the structure of a category. The monoidal structure on the groupoids $\EB{2|0}_i$, $i=0,1$, induces a monoidal structure on $\bEB{2|0}$. Moreover, the functor $E\colon\EB{2|0}\to \TV$ (of categories internal to $\Sym(\Cat^\fl/\Man)$) induces a monoidal functor, fibered over $\Man$,
\[
\overline{E}\colon \bEB{2|0}\ra \bTV
\]

The isomorphism \eqref{eq:additive} of Proposition \ref{prop:relations} holds as well in the bordism category $\EB{2|0}$ with a fixed spin structure $s$ on the cylinders involved. This implies that the morphism $C^s_{\ell,\tau}\in \bEB{2|0}(K^s_\ell,K^s_\ell)$ factors in the form 
\[
 \xymatrix{
K^s_\ell=\emptyset\amalg K^s_\ell\ar[rr]^<>(.5){R^s_{\ell,\tau}\amalg\id}&&
K^s_\ell\amalg K^s_\ell\amalg K^s_\ell
\ar[rr]^<>(.5){\id\amalg L^s_{\ell,0}}&&
K^s_\ell\amalg \emptyset=K^s_\ell}.
\]
In particular, $C^s_{\ell,\tau}$ is a thick morphism in the bordism category $\bEB{2|0}$ for $\tau\in \fh$ (note that unlike $C^s_{\ell,\tau}$ which is defined for  $\tau$ in the {\em closed } upper half-plane, $R^s_{\ell,\tau}$ is only defined for $\tau\in \fh$; is is not hard to show that $C^s_{\ell,\tau}$ is not thick for $\tau\in \R\subset\bar\fh$). 

It follows that $A_\tau^s=\overline{E}(C_{\ell,\tau}^s)\colon V^s\to V^s$ is a thick morphism in the category $\bTV$. One of the main results of our paper \cite[Theorem 4.27]{ST2} gives a  characterization of thick morphisms in the category $\TV$ as  {\em nuclear operators} (known as {\em trace class operators} if domain and range are Hilbert spaces). We won't repeat here the definition of `nuclear operator' (see \eg \cite[Def.\ 4.25, Def.\  4.28 and Lemma 4.37]{ST2}); it suffices here to know that any nuclear operator is compact. 
\end{proof}

We observe that the relations \eqref{eq:additive} and \eqref{eq:relationC} imply the isomorphism $C_{\ell,\tau_1}\circ C_{\ell,\tau_2}\cong C_{\ell,\tau_1+\tau_2}$  in $\EB{2}_1$ for $\tau_1,\tau_2\in\fh$. The same relation holds in $\EB{2|0}$ if we consider cylinders $C_{\ell,\tau}^s$ for a fixed spin structure $s\in \{\pm\}$. This implies that for fixed $\ell$, $A_\tau^s=E(C_{\ell,\tau}^s)$ is a commutative semi-group of compact operators parametrized by the upper half-plane $\fh$. The following considerations are independent of the spin structure $s$ and so we will suppress the superscript $s\in \{\pm\}$. By the spectral theorem we can decompose $V$ into a sum of simultaneous (generalized) eigenspaces for these operators. The non-zero eigenvalues give smooth homomorphisms $\fh\to \C^*$ and hence can be written in the form 
\begin{equation}\label{eq:eigenvalue}
\mu(\tau)=e^{2\pi i(a\tau-b\bar\tau)}=q^a\bar q^b
\end{equation}
for some $a,b\in \C$. Let us denote by $V_{a,b}\subset V$ the generalized eigenspace corresponding to the eigenvalue function $\mu(\tau)$ given by equation \eqref{eq:eigenvalue}. We note that the spaces $V_{a,b}$ are {\em finite dimensional}, since the operators $A_\tau$ are compact; in particular, any generalized eigenspace with non-zero eigenvalue is finite dimensional. 

The element $\rho_\tau\in V\otimes V$ can now be studied on the $V_{a,b}$'s. The following result shows that the `off-diagonal' entries vanish.

\begin{lem} Let $p_{a,b}\colon V\to V_{a,b}$ be the spectral projection onto $V_{a,b}$. Then for $(a,b)\ne (a',b')$ 
\[
p_{a,b}\otimes p_{a',b'}(\rho_\tau)=0
\]
\end{lem}

\begin{proof}
The isomorphisms \eqref{eq:additive} and \eqref{eq:relationC} in Proposition \ref{prop:relations} (or more precisely, their analog for $\EB{2|0}$) imply that the composition 
\[
\xymatrix@1{
\emptyset\ar[rr]^<>(.5){R^s_{\ell,\tau}}&&K^s_\ell\amalg K^s_\ell\ar[rr]^<>(.5){C^s_{\ell,\tau_1}\amalg C^s_{\ell,\tau_2}}&&K^s_\ell\amalg K^s_\ell
}
\]
in $\bEB{2|0}$ is equal to $R^s_{\ell,\tau+\tau_1+\tau_2}$. Applying the functor $\overline{E}$, we obtain 
\[
(A_{\tau_1}\otimes A_{\tau_2})(\rho_\tau)=\rho_{\tau+\tau_1+\tau_2}
\]
In particular, $(A_{\tau_1}\otimes \id)(\rho_\tau)=(\id\otimes A_{\tau_1})(\rho_\tau)$. Restricted to $V_{a,b}\otimes V_{a',b'}$, the operator  $A_{\tau_1}\otimes \id$ has eigenvalue $\mu_{a,b}(\tau_1)$, while 
$\id\otimes A_{\tau_1}$ has eigenvalue $\mu_{a',b'}(\tau_1)$, this shows that the projection of $\rho_\tau$ to $V_{a,b}\otimes V_{a',b'}$ must be zero. In fact, this argument only applies to the actual Eigenspaces, not the generalized Eigenspaces. However, using the Jordan normal form, we see that by an iterated application of our argument to the relevant filtration, we can conclude the same result.
\end{proof}

We recall that $Z^{++}(\tau)=\mu^+(\sigma(\rho^+_\tau))$ and we shall suppress the superscripts `$+$' from now on.

\begin{lem} 
$\mu(\sigma(p_{a,b}\otimes p_{a,b})(\rho_\tau))=\str((A_\tau)_{|V_{a,b}})$
\end{lem}

\begin{proof} The restriction of the form $\lambda\colon V\otimes V\to \C$ to $V_{a,b}$ is non-degenerate (otherwise the restriction of $A_\tau$, which can be expressed in terms of $\lambda$ and $\rho$, would have a non-trivial kernel. Let $\{e_i\}$ be a basis of $V_{a,b}$ such that $\lambda(e_i\otimes e_j)=\delta_{ij}$. Writing $(p_{a,b}\otimes p_{a,b})(\rho_\tau)$ as a linear combination of the basis elements $e_i\otimes e_j$ for $V_{a,b}\otimes V_{a,b}$ and calculating both sides proves the lemma.
\end{proof}
We conclude the last equality in the following equation. For the second equality more care is needed since $V$ is not just the direct sum of the $V_{a,b}$'s. However, a careful estimate shows that for a fixed $a\in \Z$ the `small' eigenvalues of $A_\tau$ don't contribute to the coefficient of $q^a$ in the $q$\nb-expansion of $Z(\tau)$, see \cite{ST3}
\begin{equation}\label{eq:Z^pp}
Z(\tau)=\lambda(\sigma(\rho_\tau))=\sum_{a,b}\str((A_\tau)_{|V_{a,b}})=\sum_{a,b}\mu_{a,b}(\tau)\sdim V_{a,b}
\end{equation}

\begin{lem} \label{lem:cancellation} $\sdim V_{a,b}=0$ for $b\ne 0$.
\end{lem}

\begin{proof}
By assumption, the operator $B\colon V\to V$ commutes with the operators $A_\tau$ and hence it restricts to an operator $B_{|V_{a,b}}\colon V_{a,b}\to V_{a,b}$. Restricted to $V_{a,b}$,   the operator $A_\tau$ has eigenvalue $\mu_{a,b}(\tau)$ and hence $\p A_\tau/\p \bar \tau$ has eigenvalue
\[
\frac{\p \mu_{a,b}(\tau)}{\p \bar\tau}=\frac\p{\p\bar\tau}e^{2\pi i(a\tau-b\bar\tau)}=-2\pi ib\mu_{a,b}(\tau)
\]
In particular, for $b\ne 0$, the operator $\p A_\tau/\p \bar \tau=-B^2$ is invertible, and hence $B\colon V_{a,b}\to V_{a,b}$ is an isomorphism. Since $B$ is odd, it maps the even part $V^{ev}_{a,b}$ isomorphically to the odd part $V^{odd}_{a,b}$ and hence $\sdim V_{a,b}=0$. 
\end{proof}

This then implies $Z_E(\tau)=\sum_{a\in \Z}\sdim V_{a,0}\, q^a$. Here $a$ is an integer by the following argument. 
By construction,  $C_{\ell,\tau}\cong C_{\ell,\tau'}$ if $\tau=\tau'\mod\Z$ and hence the operator $A_\tau=E(C_{\ell,\tau})$ depends only on $q=e^{2\pi i\tau}$. That forces $\mu_{a,b}(\tau)=q^a$ to be a function of $q$ which forces $a$ to be an integer. 

This shows that $Z^{++}(\tau)$ is a {\em holomorphic} function with an integral $q$\nb-expansion. Moreover, the fact that $A_\tau$ is a compact operator forces $V_{a,0}=0$ for sufficiently negative integers $a$. 

For the arguments so far, we fixed $\ell>0$, and suppressed the $\ell$\nb-dependence in the notation. If we vary $\ell$, the function $Z(\ell,\tau)$ and hence the coefficients of its $q$\nb-expansion depend continuously on that parameter. Hence the {\em integrality} of the coefficients show that $Z(\ell,\tau)$ is in fact {\em independent of $\ell$}.  This implies by remark \ref{rem:ell_independence} that $Z(1,\tau)$ is a holomorphic modular function.

\section{Supersymmetric Euclidean field theories}\label{sec:susy}
In this section we will define {\em supersymmetric Euclidean field theories} by replacing manifolds by supermanifolds and Euclidean structures by super Euclidean structures in the definitions of the previous two sections. In more detail, we will consider rigid geometries \ie  $(G,\M)$-structures in the category of supermanifolds (Def.\ \ref{def:susyBord}), and will introduce in particular the Euclidean geometry $(\Iso(\E^{d|\delta}),\E^{d|\delta})$ (see section \ref{subsec:superEuclidean}). We will define super versions of our categories $\TV$, $\GMBord$ and ${\EB d}$. From a categorical point of view, these categories will be internal to the strict $2$\nb-category $\Sym(\Cat^\fl/\csM)$ of symmetric monoids in the $2$\nb-category of categories with flip fibered over the category $\csM$ of supermanifolds. A $(G,\M)$-field theory is then defined to be a  functor $\GMBord\to \TV$ of these internal categories (Definition \ref{def:susyGMFT}); a (supersymmetric) Euclidean field theory of dimension $d|\delta$ is the special case where $(G,\M)$ is (super) Euclidean geometry of dimension $d|\delta$. We end this section with the definition of the partition function of a Euclidean field theory of dimension $2|1$ (Def.\ \ref{def:partition2}), and the proof that this partition function is an integral holomorphic modular function.

\subsection{Supermanifolds}
\label{subsec:SMan}
There are two notions of `supermanifolds', based on commutative superalgebras  over $\R$ and $\C$, respectively. Our interest in field theories of dimension $2|1$ in this paper forces us to work with the latter (see Remark \ref{rem:complex}). We recall that a {\em commutative superalgebra} is a commutative monoid in the symmetric monoidal category $\SVect$ of super vector spaces (see Example \ref{ex:SVect}). More down to earth, it is a  $\Z/2$\nb-graded algebra such that for any homogeneous elements $a,b$ we have $a\cdot b= (-1)^{|a||b|}b\cdot a$. 

\begin{defn} A {\em supermanifold $M$ of dimension $p|q$} is a pair $(M_{\red},\scr O_M)$ consisting of a (Hausdorff, second countable) topological space $M_{\red}$ (called the {\em reduced manifold}) and a sheaf $\scr O_M$ (called the {\em structure sheaf}) of commutative superalgebras over $\C$ locally isomorphic to $(\R^p,\mathscr C^\infty(\R^p)\otimes\Lambda[\theta_1,\dots,\theta_q])$. Here $\mathscr C^\infty(\R^p)$ is the sheaf of smooth complex valued functions on $\R^p$, and $\Lambda[\theta_1,\dots,\theta_q])$ is the exterior algebra generated by elements $\theta_1,\dots,\theta_q$ (which is equipped with a $\Z/2$\nb-grading by declaring the elements $\theta_i$ to be odd). It is more customary to require that $\scr O_M$ is a sheaf of real algebras; in this paper we will always be dealing with a {\em structure sheaf of complex algebras} (these are called {\em cs-manifolds} in \cite{DM}). 
Abusing language, the global sections of $\scr O_M$ are called {\em functions} on $M$; we will write $C^\infty(M)$ for the algebra of functions on $M$. 
\end{defn}

As explained by Deligne-Morgan in \cite[\S 2.1]{DM},  the quotient sheaf $\scr O_M/J$, where $J$ is the ideal generated by odd elements, can be interpreted as a sheaf of smooth functions on $M_{\red}$, giving it a smooth structure. Morphisms between supermanifolds are defined to be morphisms of ringed spaces. 

\begin{ex}\label{ex:supermfd}
Let $N$ be a smooth $p$\nb-manifold and $E\to N$ be a smooth complex vector bundle of dimension $q$. Then $\Pi E\=(N,\scr C^\infty(\Lambda E^{\spcheck}) )$ is an example of a supermanifold of dimension $p|q$. Here $\scr C^\infty(\Lambda E^{\spcheck}) )$ is the sheaf of sections of the exterior algebra bundle $\Lambda E^{\spcheck}=\bigoplus_{i=0}^q\Lambda^i(E^{\spcheck})$ generated by $E^{\spcheck}$, the bundle dual to $E$; the $\Pi$ in $\Pi E$ stands for {\em parity reversal}.  We note that {\em every} supermanifold is isomorphic to a supermanifold constructed in this way (but not every morphism $\Pi E\to\Pi E'$ is induced by a vector bundle homomorphism $E\to E'$). In particular if $E=TN_\C$, the complexified tangent bundle of $N$, then the algebra of functions on the supermanifold $\Pi TN_\C$ is given by 
\begin{equation}\label{eq:forms}
C^\infty(\Pi TN_\C)=C^\infty(N,\Lambda TN_\C^{\spcheck})=\Omega^*(N;\C),
\end{equation}
where $\Omega^*(N;\C)$ is the algebra of complex valued differential forms on $N$. 
\end{ex}

Now we discuss how the rigid geometries from section \ref{subsec:rigid}, their family versions, and the $(G,\M)$-bordism groups (see Definition \ref{def:GMBord}) can be generalized to supermanifolds. This is straightforward since we've been careful to express these definition in categorical terms, \ie in terms of objects and morphisms of the category $\Man$ of smooth manifolds, which we now replace by the category $\csM$ of supermanifolds. When topological terms, like `open subset of $Y$' appear in a definition, they have to be interpreted in term of the reduced manifold; \eg an open subset of a supermanifold $Y$ has to be interpreted as the supermanifold given by the structure sheaf $\scr O_Y$ restricted to some open subset $U$ of the reduced manifolds $Y_{red}$. 

\begin{defn}\label{def:susyBord}
Let $G$ be a super Lie group (\ie a group object in the category $\csM$), and suppose that $G$ acts on some supermanifold $\M$ of dimension $d|\delta$. Then we can interpret the pair $(G,\M)$ as a `rigid geometry', and define as in Definition \ref{def:GM_object} the notion of a $(G,\M)$-manifold. If in addition we specify a $1$\nb-codimensional submanifold $\M^c\subset \M$ (this is a manifold of dimension $d-1|\delta$) then we can define a $(G,\M)$-structure on pairs as well. 

Generalizing Definition \ref{def:GM_fam} from manifolds to supermanifolds, we obtain a notion of families $Y\to S$ of $(G,\M)$-manifolds respectively pairs (we note that the parameter space $S$ of these families is in general a supermanifold).  The construction of the bordism categories $\GMBord_0$, $\GMBord_1$ (Def.\ \ref{def:GMBord}) generalizes as well; again, point topology assumptions, \eg the properness assumption on the map $Y^c\to S$ in the definition of an object has to be interpreted as a statement about the reduced manifolds. The forgetful functor which sends a family of $(G,\M)$-manifolds to its parameter space is now a functor from $\GMBord_i$ to the category of supermanifolds $\csM$, making them categories over $\csM$. 

We observe that the category $\csM$ comes equipped with a canonical flip $\theta$: For every supermanifold $M$, we can define $\theta_M(f)=f$ respectively $\theta_M(f)=-f$ depending on whether $f\in C^\infty(M)$ is even respectively odd. 
If there is a central point $g:\pt\to G$ such that multiplication by $g$ induces $\theta_\M$, we can define a flip $\theta$ for the categories $\GMBord_i$, making them objects in the strict $2$\nb-category $\Sym(\Cat^\fl/\csM)$. Here $\Cat^\fl/\csM$ is the strict $2$\nb-category of categories with flip over $\csM$, and $\Sym(\Cat^\fl/\csM)$ is the strict $2$\nb-category  of symmetric monoids in that $2$\nb-category. So what we obtain in the end is a category $\GMBord$ which is internal to $\Sym(\Cat^\fl/\csM)$. More generally, if $X$ is a manifold, we have the bordism category $\GMBord(X)$ of $(G,\M)$-manifolds with maps to $X$. 
\end{defn}
\subsection{Super Euclidean geometry}
\label{subsec:superEuclidean}

Next we define the super analogues of Euclidean manifolds, namely $(G,\M)$-manifolds, where $\M$ is {\em super Euclidean space}, and $G$ is the {\em super Euclidean group}.   Our definitions are modeled on the definitions of {\em super Minkowski space} and {\em super Poincar\'e group} in \cite[\S 1.1]{DF}, \cite[Lecture 3]{Fr}.

To define super Euclidean space, we need the following data:

\medskip

\begin{tabular}{p{1in}p{5.2in}}
$V$&a real vector space with an inner product\\
$\Delta$&a complex spinor representation of $Spin(V)$\\
$\Gamma\colon \Delta\otimes \Delta\to V_\C$
&a $Spin(V)$\nb-equivariant, non-degenerate symmetric pairing
\end{tabular}

\medskip

Here $V_\C$ is the complexification of $V$. A complex representation of $Spin(V)\subset \C\ell(V)^{ev}$ is a {\em spinor representation} if it extends to a module over $\C\ell(V)^{ev}$, the even part of the complex Clifford algebra generated by $V$. 

The supermanifold $V\times \Pi\Delta$ is the {\em super Euclidean space}. We note that this is the supermanifold associated to the trivial complex vector bundle $V\times \Delta\to V$, and hence the algebra of functions on this supermanifold is the exterior algebra (over $C^\infty(V)$) generated by the $\Delta^{\spcheck}$\nb-valued functions on $V$, which we can interpret as {\em spinors on $V$}.

The pairing $\Gamma$ allows us to define a multiplication on the supermanifold $V\times\Pi \Delta$ by
\begin{align}\label{eq:supertranslations}
(V\times\Pi \Delta)\times (V\times\Pi \Delta)&\ra V\times\Pi \Delta\\
(v_1,w_1), (v_2,w_2)&\mapsto (v_1+v_2+\Gamma(w_1\otimes w_2),w_1+w_2),
\end{align}
which gives $V\times\Pi \Delta$ the structure of a super Lie group, \ie a group object in the category of supermanifolds  (see \cite[\S 2.10]{DM}). Here we describe the multiplication map in terms of the {\em functor of points approach} explained \eg in \cite[\SS 2.8-2.9]{DM}: for any supermanifold $S$ the set $X_S$ of  {\em $S$\nb-points} in another supermanifold $X$ consists of all morphisms  $S\to X$. For example, an $S$\nb-point of the supermanifold $V\times\Pi \Delta$ amounts to a pair $(v,w)$ with $v\in C^\infty(S)^{ev}\otimes_\C V_\C$, $w\in C^\infty(S)^{odd}\otimes \Delta$ and $\bar v_{red}=v_{red}$ (where $v_{red}\in C^\infty(S_{red})\otimes V_\C$ is the restriction of $v$ to the reduced manifold, and $\bar v_{red}$ is its complex conjugate). A morphism of supermanifolds $X\to Y$ induces maps  $X_S\to Y_S$ between the $S$\nb-points of $X$ and $Y$, which are functorial in $S$. Conversely, any collection of maps $X_S\to Y_S$ which is functorial in $S$ comes from a morphism $X\to Y$ (by Yoneda's lemma). 

We note that the spinor group $Spin(V)$ acts on the supermanifold $V\times\Pi \Delta$ by means of the double covering $Spin(V)\to SO(V)$ on $V$ and the spinor representation on $\Delta$. The assumption that the pairing $\Gamma$ is $Spin(V)$\nb-equivariant guarantees that this action is compatible with the (super) group structure we just defined. We define the {\em super Euclidean group} to be the semi-direct product
$
(V\times\Pi \Delta)\rtimes Spin(V)$.
By construction, this supergroup acts on the supermanifold $V\times\Pi \Delta$ (the {\em translation} subgroup $V\times\Pi \Delta$ acts by group multiplication on itself, and $Spin(V)$ acts as explained above). 

We will use the following notation and terminology
\begin{alignat}{2}
\E^{d|\delta}&:=V\times\Pi\Delta
&\qquad&\text{super Euclidean space}\\
\Iso(\E^{d|\delta})&:=(V\times\Pi\Delta)\rtimes Spin(V)
&\qquad&\text{super Euclidean group},
\end{alignat}
where $d=\dim_\R V$, $\delta=\dim_\C\Delta$. 

\begin{rem}\label{rem:superEuclidean} 
Up to isomorphism, there are two (respectively one) irreducible module(s) of dimension $2^{(d-2)/2}$ (respectively $2^{(d-1)/2}$) over $\C\ell(V)^{ev}$ if $d$ is even (respectively odd). This implies that $\delta$ is a multiple of $2^{(d-2)/2}$ (respectively $2^{(d-1)/2}$). In particular if $\delta=1$, the case we are interested in, this forces $d=0,1,2$. 

Similarly, up to isomorphism, the inner product space $V$ and hence the associated Euclidean group is determined by the dimension of $V$. By contrast, the isomorphism class of the data $(V,\Delta,\Gamma)$ is in general not determined by the superdimension $d|\delta=\dim_\R V|\dim_\C\Delta$. In particular,   the (isomorphism class of the) pair $(\Iso(\E^{d|\delta}),\E^{d|\delta})$ might depend on (the isomorphism class of) $(V,\Delta,\Gamma)$, not just $d|\delta$,  contrary to what the notation might suggest.
\end{rem}

In this paper, we are only interested in the cases $d|\delta=0|1,1|1$ and $2|1$. We note that $\C\ell_0^{ev}=\C\ell_1^{ev}=\C$ and hence there is only one module $\Delta$ (up to isomorphism) of any given dimension $\delta$. For $d=\dim V=0$, the homomorphism  $\Gamma$ is necessarily trivial; for $d=1$, $\delta=1$, the homomorphism  $\Gamma$ is determined (up to isomorphism of the pair $(\Delta,\Gamma)$) by the requirement that $\Gamma$ is non-degenerate. 

For $d=2$, $\delta=1$, there are two non-isomorphic modules $\Delta$ over $\C\ell_2^{ev}=\C\oplus\C$. To describe them explicitly as representations of $Spin(2)$, we identify the double covering map $Spin(2)\to SO(2)=S^1$ with $\R/2\Z\to \R/\Z$ by mapping $\tau\in \R/\Z$ to $q=e^{2\pi i\tau}\in S^1$.
 The irreducible complex representations of $Spin(2)$ are parametrized by half integers $k\in \frac 12\Z$. For $k\in \frac 12\Z$ let us write $\C_k$ for  the complex numbers equipped with a $Spin(2)$\nb-action such that $\tau\in \R/2\Z$ acts by multiplication by $q^k=e^{2\pi i \tau k}$ (note that this is well-defined for $k\in \frac 12\Z$). Then up to isomorphism $\Delta=\C_k$ for $k=\pm 1/2$ and the $S^1$\nb-equivariant homomorphism 
\[
\Delta\otimes\Delta=\C_{2k}\overset\Gamma\ra
V_\C=\C_1\oplus\C_{-1}
\]
is given by the inclusion map into the first summand (for $k=1/2$) respectively second summand (for $k=-1/2$). For reasons that will become clear later, we fix our choice of $(\Delta,\Gamma)$ to be given by $k=-1/2$. Specializing the multiplication on the super Lie group $\E^{d|\delta}=V\times \Pi \Delta$ (see \eqref{eq:supertranslations}) to the case at hand, we obtain
\begin{equation}\label{eq:E2|1}
\E^{2|1}\times \E^{2|1}\ra \E^{2|1}
\qquad
(\tau_1,\bar\tau_1,\theta_1),(\tau_2,\bar\tau_2,\theta_2)\mapsto
(\tau_1+\tau_2,\bar\tau_1+\bar\tau_2+\theta_1\theta_2,\theta_1+\theta_2),
\end{equation}
where $\theta_i$ are odd functions on some parametrizing supermanifold $S$, and  $\tau_i,\bar\tau_i$ are even functions whose restriction to $S_{red}$ are complex conjugates of each other.

\begin{rem}\label{rem:complex} If the module $\Delta$ over $\C\ell(V)^{ev}$ (the even part of the complex Clifford algebra generated by the real vector space $V$),  is the complexification of a real module $\Delta_\R$ over $C\ell(V)^{ev}$ (the even part of the real Clifford algebra generated by $V$), then we can consider the {\em real supermanifold} $V\times \Pi \Delta_\R$, and work throughout in the category of real supermanifolds. This happens for $d=0,1$, $\delta=1$ and prompted us to work with real supermanifolds in our papers \cite{HoST}, \cite{HKST} which deal with these cases. 

By contrast, for $d=\dim V=2$, the algebra $C\ell(V)^{ev}$ is isomorphic to  $\C$, and hence the smallest real module over this algebra has dimension two. This forces us to work with {\em supermanifolds} if we want to consider Euclidean structures or Euclidean field theories of dimension $2|1$. 
\end{rem}


\subsection{Supersymmetric field theories}\label{subsec:susyFT2}

In Definition \ref{def:famTV} we defined the categories $\TV_i$, $i=0,1$, over $\Man$ resulting in categories with flip over $\Man$, where the flip was given by the grading involution. These categories can be extended to categories with flip over $\csM\supset\Man$. An object of $\TV_0$ is a supermanifold $S$ and a sheaf $V$ over $S_{\red}$ of $\Z/2$\nb-graded topological $\scr O_S$-modules. We note that the flip $\theta_{S,V}$ of such an object is again the automorphism of $(S,V)$ given by the grading involution. In contrast to the earlier situation where $S$ was a manifold, this in now in general a non-trivial automorphism of $S$. In categorical terms, the forgetful functor $\TV_0\to \csM$ preserves the flip defined on both categories by the grading involution. This construction then results in objects $\TV_0,\TV_1\in \Cat^\fl/\csM$ (the strict $2$\nb-category of categories with flip fibered over $\csM$). In fact, these are symmetric monoids in $\Cat^\fl/\csM$ and  fit together to give a category $\TV$ internal to the strict $2$\nb-category $\Sym(\Cat^\fl/\csM)$ of symmetric monoids of $\Cat^\fl/\csM$.

\begin{defn}\label{def:susyGMFT}
Given a geometry $(G,\M)$ as in Definition \ref{def:susyBord}, a {\em $(G,\M)$-field theory} is a functor
\[
E\colon \GMBord\ra \TV
\]
of categories internal to $\Sym(\Cat^\fl/\csM)$. More generally, a {\em $(G,\M)$-field theory over $X$} is a functor $E\colon \GMBord(X)\ra \TV$. If $(G,\M)$ is the Euclidean geometry $(\Iso(\E^{d|\delta}),\E^{d|\delta})$, we refer to $E$ as a {\em (supersymmetric) Euclidean field theory of dimension $d|\delta$}. 
\end{defn}

Let $E$ be a $(G,\M)$-field theory and $\Sigma\in \GMBord_1(\emptyset,\emptyset)$, \ie a family of closed supermanifolds with $(G,\M)$-structure parametrized by some supermanifold $S$. We can evaluate $E$ on $\Sigma$ to obtain an element $E(\Sigma)\in C^\infty(S)^{ev}$. Since $E$ is a functor between categories internal to $\Sym(\Cat^{\fl}/\csM)$, no information is lost by restricting $\Sigma$ to be a family of {\em connected} closed $(G,\M)$-manifolds.  From an abstract point of view, the assignment $\Sigma\mapsto E(\Sigma)$ is a function on the moduli stack of closed connected $(G,\M)$-manifolds.  

\begin{defn}\label{def:partition2}
The function $Z_E$ described above is the {\em partition function} of the field theory $E$.
\end{defn}
If $(G,\M)$ is the Euclidean geometry $(\Iso(\E^{2|0}),\E^{2|0})$ then the corresponding moduli stack has two components labeled by the (isomorphism classes of) spin structures on a torus. The function  on the component corresponding to the non-bounding spin structure $++$ can be interpreted as an $SL_2(\Z)$-invariant function on $\fh \times \R_+$ by mapping this quotient into the moduli space. Moreover, we can further restrict it to $\fh$, giving up the $SL_2(\Z)$-invariance. This restriction $Z^{++}_E:\fh \to \C$ then agrees with our previous Definition~\ref{def:partition}. 

We observe that a Euclidean field theory $E$ of dimension $d|\delta$ determines an associated {\em reduced Euclidean field theory} $\bar E$ of dimension $d|0$ which is given as the composition 
\begin{equation}\label{eq:superfication}
\xymatrix@1{
\EB{d|0}\ar[r]^<>(.5){\cS}&\EB{d|\delta}\ar[r]^<>(.5){E}&\TV
}
\end{equation}
$\cS$ is the {\em superfication functor}, given by associating to a Euclidean manifold $Y$ of dimension $d|0$ functorially a supermanifold $\cS (Y)$ of dimension $d|\delta$ with Euclidean structure. We recall that a $d|0$\nb-dimensional Euclidean structure on $Y$ determines in particular a spin-structure on $Y$ and hence a principal $Spin(d)$\nb-bundle $Spin(Y)\to Y$. Then we define
\[
\cS(Y):=\Pi\left(Spin(Y)\times_{Spin(d)}\Delta\right)
\]
where $\Delta$ is our choice of the complex spinor representation used in the definition of Euclidean structure of dimension $d|\delta$, $\delta=\dim_\C\Delta$. Moreover $\Pi$ is the construction that turns complex vector bundles over ordinary manifolds into supermanifolds (see Example \ref{ex:supermfd}). It is not hard to see that each $\E^d$\nb-chart for $Y$ determines a $\E^{d|\delta}$\nb-chart for $\cS(Y)$. The transition function for two charts for $\cS(Y)$ is the image of the transition function of the corresponding charts for $Y$ under the inclusion homomorphism 
\[
\Iso(\E^d)=\E^d\rtimes Spin(d)\into \E^{d|\delta}\rtimes Spin(d)=\Iso(\E^{d|\delta})
\]

\begin{rem}\label{rem:partition2} Let $E$ be a Euclidean field theory of dimension $2|1$. Then its partition function $Z_{ E}$ is determined by the partition function $Z_{\bar E}$ of the associated reduced Euclidean spin field theory $\bar E=E\circ\cS$ because the moduli stack of supertori only has one odd direction. As mentioned before, we usually only study the restriction $Z^{++}_{\bar E}$ to the component corresponding to the non-bounding spin structure $++$. By Propositions~\ref{prop:modular} and \ref{prop:square}, it is in fact independent of the scale $\ell$ and hence determined by its restriction to $\fh$. This restriction was called the `partition function' in Definition~\ref{def:partition}.
\end{rem}

\subsection{Modularity of the partition function of $\EFT{2|1}$'s, part II}
\label{subsec:modular2}

The goal of this subsection is to prove the following result.

\begin{prop}\label{prop:square} Let $E$ be a $\EFT{2}$ which has a supersymmetric extension $\wh E\in \EFT{2|1}$, i.e., $E=\wh E\circ\cS$. Then $E$ satisfies the assumption of Proposition \ref{prop:modular}.
\end{prop}

In conjunction with Proposition \ref{prop:modular}, this implies the first part of Theorem \ref{thm:modular} for EFT's. The second part will be proved after the proof of Proposition \ref{prop:square}.

\begin{proof}
The proof of this result is based on looking at the semi-group of operators obtained by applying the field theory $\wh E$ to a family of `supercylinders' $C^s\to \fh^{2|1}$ parametrized by $\fh^{2|1}\subset \R^{2|1}$ (here $\fh^{2|1}$ is the supermanifold obtained from $\R^{2|1}$ by restricting the structure sheaf to the upper half plane). This family requires the choice of a scaling parameter $\ell\in \R_+$, and it is an extension of the family  of cylinders $C^s_{\ell,\tau}$ parametrized by $\tau\in \fh$, since the pullback of $C^s$ via $\pt\overset\tau\ra \fh\subset \fh^{2|1}$ is $C^s_{\ell,\tau}$. Composition of supercylinders corresponds to the multiplication map $\mu\colon \E^{2|1}\times \E^{2|1}\to \E^{2|1}$ of the super Lie group $\E^{2|1}$ (see \eqref{eq:E2|1}) in the following sense. Composing supercylinders of our family parametrized by $\fh^{2|1}$ leads to a family of supercylinders parametrized by $\fh^{2|1}\times\fh^{2|1}$. This family is isomorphic to the family  given by pulling back $C^s$ via the map $\mu\colon \fh^{2|1}\times \fh^{2|1}\to \fh^{2|1}$, the restriction of $\mu$ to the semigroup $\fh^{2|1}\subset \E^{2|1}$.

Applying the functor $\wh E$ to the family $C^s$ we obtain a smooth map 
\[
f^s\colon \fh^{2|1}\ra \cN(V^s)
\]
to the space of nuclear operators on $V^s=E(K_\ell^s)$. Since composition of cylinders corresponds to multiplication in $\fh^{2|1}$, this is a semigroup homomorphism. We can write the function $f^s$ in the form $f^s=A^s+\theta B^s$, where $A^s,B^s\colon \fh\to \cN(V)$ are smooth maps and $A^s(\tau)$ (respectively $B^s(\tau)$) is an even (respectively odd) operator on the $\Z/2$\nb-graded vector space $V^s$. The operators $A^s(\tau)$ are determined by the field theory $E$ since $A^s(\tau)=E(C_{\ell,\tau}^s)$. 
The homomorphism property of $f^s$ is equivalent to the following relations (for simplicity we suppress the superscripts $s$ throughout)
\begin{equation}\label{eq:2|1_relation}
\begin{aligned}
A(\tau_1)A(\tau_2)&=A(\tau_1+\tau_2)\\
A(\tau_1)B(\tau_2)&=B(\tau_1)A(\tau_2)=B(\tau_1+\tau_2)\\
B(\tau_1)B(\tau_2)
&=-\frac{\partial A}{\partial \bar z}(\tau_1+\tau_2)
\end{aligned}
\end{equation}
In particular,  $C(\tau):=B^+(\tau/2)$ is an odd operator whose square is $-\frac{\p A^+}{\p\bar\tau}(\tau)$, which proves Proposition \ref{prop:square}.
\end{proof}

\begin{proof}[Proof of Theorem \ref{thm:modular}, second part]
We recall that according to Proposition \ref{prop:partition_spin_2EFT} every integral modular function can be realized as the partition function of a $\CEFT{2|0}$ $E$. It remains to show that we can construct $E$ in such a way that it has an extension to a $\CEFT{2|1}$ $\wh E$. 

We note that there is a necessary condition for the existence of  $\wh E$: The proof of Proposition \ref{prop:square} shows that if $E$ has a supersymmetric extension $\wh E$, then we can find smooth functions $B^s\to \cN(V^s)^{odd}$ for $s=\pm$ such that the relations \eqref{eq:2|1_relation} are satisfied for $A^s(\tau)=E(C^s_\tau)$. It turns out that this condition is sufficient as well. The proof involves a study of the moduli stack of supertori. A priori, the requirement that the partition function $Z_E$ on the stack of Euclidean tori should extend to a function on the stack of supertori could impose new restrictions on $Z_E$ besides $SL_2(\Z)$\nb-equivariance. In fact, it does: $Z_E$ extends to the moduli stack of supertori \Iff $Z_E$ is {\em holomorphic}. However, the existence of $B^s$ satisfying  relations \eqref{eq:2|1_relation} for $A^s(\tau)=E(C_\tau^s)$ ensures that for $\tau\in \fh^{+s}$ the function
\[
Z_E(\tau)=E(T^{+s}_\tau)=\str E(C^s_\tau)=\str A^s(\tau)
\]
is holomorphic. Hence the existence of $B^s$ satisfying the relations implies that $Z_E$ is holomorphic restricted to $\fh^{++}$ and $\fh^{+-}$. This implies that $Z_E$ is holomorphic on the other two components $\fh^{-+}$, $\fh^{--}$ as well, since $Z_E$ is $SL_2(\Z)$\nb-equivariant and the $SL_2(\Z)$\nb-action permutes $\fh^{+-}$, $\fh^{-+}$ and $\fh^{--}$.

For the EFT $E$ we've constructed in the proof of Proposition \ref{prop:partition_2EFT} respectively \ref{prop:partition_spin_2EFT} it is easy to show that there are smooth families $B^s\colon \fh\to \cN(V^s)$ satisfying relations \eqref{eq:2|1_relation}: in terms of the algebraic data $(V^s,\lambda^s,\rho^s(\tau))$  that determine the $\CEFT{2|0}$ $E$, the operator $A^s(\tau)=E(C^s_\tau)$ is given  by the composition \eqref{eq:Atau} of $\rho^s(\tau)$ and $\lambda^s$. Since $\rho^s(\tau)$ depends {\em holomorphically} on $\tau$, the function $A^s(\tau)$ is holomorphic. Hence  setting $B^s\equiv 0$  the relations \eqref{eq:2|1_relation} are satisfied and we conclude that $E$ has a supersymmetric extension $\wh E$. 
\end{proof}

\section{Twisted field theories}\label{sec:twisted}

In this section we will define field theories of non-zero degree, or -- in physics lingo -- non-zero central charge. More generally, we will define {\em twisted} field theories over a manifold $X$. As explained in the introduction we would like to think of field theories over $X$ as representing cohomology classes for certain generalized cohomology theories. Sometimes it is {\em twisted} cohomology classes that play an important role, \eg the Thom class of a vector bundle that is not orientable for the cohomology theory in question. We believe that the twisted field theories defined below (see Definition \ref{def:twistedEFT}) represent twisted cohomology classes, which motivates our terminology.  We will outline a proof of this for $d|\delta=0|1$ and $1|1$.

We will describe  {\em twisted field theories} as natural transformations between functors (see Definition \ref{def:twistedEFT}). More precisely, these are functors between internal categories; their domain is our internal bordism category $\EB {d|\delta}$. So our first task is to describe what is meant by a natural transformation between such functors. Then we will construct the range category and  outline the construction of the relevant functors which will allow us to define Euclidean field theories of degree $n$. We end the section by relating Euclidean field theories of degree zero to field theories as defined in section  \ref{subsec:susyFT2} and by comparing our definition with Segal's definition of conformal field theories with non-trivial central charge.

\subsection{Twisted Euclidean field theories}
\label{subsec:twisted}
We first introduce the internal category that will serve as range category for a field theory.

\begin{defn}\label{def:TA}
The internal category $\TA$ of {\em topological algebras} has the following object and morphism categories:
\begin{itemize}
\item[$\TA_0$] Is the groupoid whose objects are {\em topological algebras}. A topological algebra  is a monoid in the symmetric monoidal groupoid $\TV$ of topological vector spaces (equipped with the projective tensor product); \ie an object $A\in \TV$ together with an associative multiplication $A\otimes A\to A$. Morphisms are continuous algebra isomorphisms.
\item[$\TA_1$] is the category of bimodules over topological algebras. A bimodule is a triple $(A_1,B,A_0)$, where $A_0$, $A_1$ are topological algebras, and $B$ is an $A_1$-$A_0$\nb-bimodule (\ie an object $B\in \TV$ with a morphism $A_1\otimes B\otimes A_0\to B$ satisfying the usual conditions required for a module over an algebra). A morphism from $(A_1,B,A_0)$ to $(A'_1,B',A'_0)$ is a triple $(f_1,g,f_0)$ consisting of isomorphisms $f_0\colon A_0\to A_0'$, $f_1\colon A_1\to A_1'$ of topological algebras and a morphism $f\colon  B\to B'$ of topological vector spaces which is compatible with the left action of $A_1$ and the right action of $A_0$ ($A_1$ acts on $B'$ via the algebra homomorphism $A_0\to A_0'$, and similarly for $A_0$).
\end{itemize}
There are obvious source and target functors 
\[
s,t\colon \TA_1\to\TA_0 \quad\text{given by} \quad s(A_1,B,A_0)= A_0  \quad \text{and} \quad t(A_1,B,A_0)= A_1, 
\]
and a composition functor
\[
c\colon \TA_1\times_{\TA_0}\TA_1\ra \TA_1
\qquad
(A_2,B,A_1),(A_1,B',A_0)\mapsto (A_2,B\otimes_{A_1}B',A_0)
\]
The two categories $\TA_0$, $\TA_1$, the functors $s,t,c$ and the usual associator for tensor products define a category $\TA$ of {\em topological algebras} internal to the strict $2$\nb-category $\Cat$. We can do better by noting that the tensor product (in the category $\TV$) makes $\TA_0$, $\TA_1$ symmetric monoidal categories and that the functors $s,t,c$ are symmetric monoidal functors. This gives $\TA$ the structure of a category internal to the strict $2$\nb-category $\SymCat$ of symmetric monoidal categories.

Analogously to $\TV$, there is a family version of the internal category $\TA$, obtained by replacing algebras (respectively bimodules) by families of algebras (respectively bimodules) parametrized by some supermanifold. To come up with the correct definition of `family' here, it is useful to recall that an algebra is a monoid in the category of vector spaces, and that a $A_1-A_2$\nb-bimodule is an object in the category of vector spaces which comes equipped with a left action of the monoid $A_1$ and a commuting right action of the monoid $A_2$. Since we've decided that a `family of topological vector spaces parametrized by a supermanifold $S$' is a sheaf over $S_{red}$ of topological $\scr O_S$-modules, we decree that an $S$\nb-family of topological algebras (respectively bimodules) is a monoidal object in the category of $S$\nb-families of topological vector spaces (respectively an object with commuting left- and right actions of monoids). 
The obvious forgetful functors $\TA_i\to \csM$ for $i=0,1$ make these categories Grothendieck fibered over the category $\csM$ of supermanifolds; the (projective) tensor product makes them symmetric monoids in $\Cat^\fl/\csM$. 
\end{defn}

\begin{defn}\label{def:twistedEFT} Let $X$ be a smooth manifold and let 
\[
T\colon \GMBord(X)\to \TA
\]
 be a functor between categories internal to  $\Sym(\Cat^\fl/\csM)$; we will refer to such a functor as {\em twist}. 
A {\em $T$\nb-twisted $(G,\M)$-field theory} is a natural transformation, again internal to  $\Sym(\Cat^\fl/\csM)$,
\[
\xymatrix{
\GMBord(X)\ar@/^2pc/[rr]^{T^0}_{}="1"\ar@/_2pc/[rr]_T^{}="2"&& \TA \ar@{=>}^E"1";"2"
}
\]
where $T^0$ is the {\em constant twist} that maps every object $Y$ of $\GMBord(X)_0$ to the algebra $\C\in \TA_0$ (which is the unit in the symmetric monoidal groupoid  $\TA_0$). It maps every object $\Sigma$ of $\GMBord(X)_1$ to $(\C,B,\C)\in \TA_1$,  where $B$ is $\C$ regarded as a $\C$-$\C$\nb-bimodule (this is the monoidal unit in $\TA/\csM_1$). The functors $T^0_i\colon \GMBord_i(X)\to \TA_i$ send every morphism to the identity morphism of the monoidal unit of $\TA_i$.

We note that the symmetric monoidal structure on $\TA_i$, $i=0,1$, allows us to form the tensor product $T_1\otimes T_2$ of two twists. The tensor product $T\otimes T^0$ of any twist $T$ and the constant twist $T^0$ is naturally isomorphic to $T$. If $E_1$ is a $T_1$\nb-twisted field theory and $E_2$ is a $T_2$\nb-twisted field theory, we can form the tensor product $E_1\otimes E_2$, which is a $T_1\otimes T_2$\nb-twisted field theory. 
\end{defn}

Let us unravel this definition. For simplicity, we take the target $X$ to be a point (hence no additional information) and ignore the family aspect by restricting to the fiber category $\GMBord_{\pt}$ over the point $\pt$. A twist
\[
T=(T_0,T_1)\colon \GMBord_{\pt}\ra \TA_{\pt}
\]
 associates to an object $(Y,Y^c,Y^\pm)\in \GMBord_{\pt}$ (consisting of a $(G,\M)$-manifold $Y$ and a compact codimension one submanifold $Y^c$ dividing $Y$ into $Y^+$ and $Y^-$) a topological algebra $T_0(Y)$. To a $(G,\M)$-bordism $\Sigma$ from $Y_0$ to $Y_1$ the functor $T_1$ associates the $T_0(Y_1)$-$T_0(Y_0)$\nb-bimodule $T_1(\Sigma)$. 
  
 To understand the mathematical content of the natural transformation $E=(E_0,E_1)$,  we use Definition \ref{def:lax_nattrans} and Diagram \eqref{eq:lax_nattrans} in the case
\[
\sC=\GMBord_{\pt} \qquad \sD=\TA_{\pt} \qquad f=T^0 \qquad g=T \qquad n=E
\]
We see that a natural transformation $E$ is a pair $(E_0,E_1)$, consisting of the following data (where we suppress the subscript $\pt$):
\begin{itemize}
\item $E_0\colon \GMBord_0\to \TA_1$ is a functor; in particular, $E_0$ associates to each $(Y,Y^c,Y^\pm)$ an object $E_0(Y)$ of $\TA_1$, \ie $E_0(Y)$ is a bimodule. The commutative triangle \eqref{eq:nattrans} implies that $E_0(Y)$ is a left module over $t E_0(Y)=g_0(Y)=T_0(Y)$ and a right module over $sE_0(Y)=f_0(Y)=T^0_0(Y)=\C$; in other words, $E_0(Y)$ is just a left $T_0(Y)$\nb-module.
\item According to diagram \eqref{eq:lax_nattrans} $E_1$ is a natural transformation, \ie for every bordism $\Sigma$ from $Y_0$ to $Y_1$ (this is an object of $\sC_1$) we have a morphism $E_1(\Sigma)$ in $\sD_1$ whose domain (respectively range) is the image of $\Sigma\in \sC_1$ under the functors
\[
\xymatrix@1{\sC_1\ar[r]^<>(.5){g_1\times n_1s}&
\sD_1\times_{\sD_0}\sD_1\ar[r]^<>(.5){c_\sD}&\sD_1}
\quad\text{respectively}\quad
\xymatrix@1{\sC_1\ar[r]^<>(.5){n_1t\times f_1}&
\sD_1\times_{\sD_0}\sD_1\ar[r]^<>(.5){c_\sD}&\sD_1}.
\]
More explicitly, $E_1(\Sigma)$ is a map of left $T_0(Y_1)$\nb-modules 
\[
E_1(\Sigma)\colon T_1(\Sigma)\otimes_{T_0(Y_0)}E_0(Y_0)\ra E_0(Y_1)\otimes_{T^0_0(Y_1)}T^0_1(\Sigma)\cong E_0(Y_1),
\]
or, equivalently, a $T_0(Y_1)-T_0(Y_0)$\nb-bimodule map
\begin{equation}
E_1(\Sigma)\colon T_1(\Sigma)\ra \Hom(E_0(Y_0),E_0(Y_1))
\end{equation}
If $\Phi:\Sigma \to \Sigma'$ is a $(G,\M)$-isometry, \ie a morphism in $\sC_1$ then we obtain a commutative diagram
\begin{equation}\label{eq:isometry}
\xymatrix{
T_1(\Sigma) 
\ar[d]_{T_1(\Phi)}^\cong \ar[r]^<>(.5){E_1(\Sigma)}& \Hom(E_0(Y_0),E_0(Y_1)) \ar[d]^\cong_{E_0(\partial\Phi)}\\
T_1(\Sigma') \ar[r]^<>(.5){E_1(\Sigma')}&  \Hom(E_0(Y'_0),E_0(Y'_1)) 
}
\end{equation}

\end{itemize}
We note that the subscripts for the functors $T_i$, $E_i$ are redundant (whether we mean $T_0$ or $T_1$ is clear from the object we apply these functors to), and hence we will suppress the subscripts from now on.
Summarizing for future reference,  a $T$\nb-twisted field theory $E$ amounts to the following data:
\begin{equation}\label{eq:nat_trans_data}
\begin{aligned}
&\text{left $T(Y)$\nb-modules $E(Y)$ and `isometry invariant'}\\
&\text{$T(Y_1)-T(Y_0)$-bimodule maps $T(\Sigma)\to \Hom(E(Y_0),E(Y_1))$}
\end{aligned}
\end{equation}
for  objects $(Y,Y^c,Y^\pm)$ and bordisms $\Sigma$ from $Y_0$ to $Y_1$.  

The commutativity of the octagon required in the definition of a natural transformation (Definition \ref{def:lax_nattrans}) amounts to the commutativity of the following diagram for any bordism $\Sigma$ from $Y_0$ to $Y_1$ and bordism $\Sigma'$ from $Y_1$ to $Y_2$.
\begin{equation}\label{eq:gluing}
\xymatrix{
T(\Sigma')\otimes_{T(Y_1)}T(\Sigma)
\ar[d]^\cong\ar[rr]^<>(.5){E(\Sigma')\otimes E(\Sigma)}&&
\Hom(E(Y_1),E(Y_2))
\otimes_{T(Y_1)}
\Hom(E(Y_0),E(Y_1))\ar[d]^{\circ}\\
T(\Sigma'\circ\Sigma)
\ar[rr]^<>(.5){E(\Sigma'\circ\Sigma)}&&
\Hom(E(Y_0),E(Y_2))
}
\end{equation}

We note that if the twist $T$ is the constant twist $T^0$, then $E(Y)$ is just a topological vector space, and $E(\Sigma)$ is a continuous linear map $E(Y_0)\to E(Y_1)$. The commutativity of the diagram above is the requirement that composition of bordisms corresponds to composition of the corresponding linear maps. In other words:

\begin{lem} The groupoid of  $T^0$\nb-twisted $(G,\M)$-field theories is isomorphic to the groupoid of $(G,\M)$-field theories as in Definition \ref{def:susyGMFT}. 
\end{lem}

\begin{rem}\label{rem:noninvertible_2cell} 
If $S,T\colon \GMBord\to \TA$ are two twists, and $E\colon S\Rightarrow T$ is an invertible twist, then for each $Y\in \GMBord$ the $T(Y)-S(Y)$-bimodule $E(Y)$ provides a Morita equivalence between the algebras $T(Y)$ and $S(Y)$; in particular, $E(Y)$ is an irreducible bimodule. If $S=T=T^0$, this implies that $E(Y)$ is a complex vector space of dimension one. This shows that only very special field theories correspond to {\em invertible natural transformations}.
\end{rem}

Let $E$ be a $T$\nb-twisted $(G,\M)$-field theory. Given $\Sigma\in \GMBord_1(\emptyset,\emptyset)$, \ie a family of closed supermanifolds with $(G,\M)$-structure parametrized by some supermanifold $S$, we can evaluate $E$ on $\Sigma$ to obtain an element $E(\Sigma)\in \Hom(T(\Sigma),\C)=T(\Sigma)^{\spcheck}$ (we note that since $\Sigma$ is a family of {\em closed} supermanifolds, $T(\Sigma)$ is just a sheaf of topological $\scr O_S$\nb-modules (rather than bimodules over algebras associated to the incoming/outgoing boundary of $\Sigma$). Since $E$ is a functor between categories internal to $\Sym(\Cat^{\fl}/\csM)$, no information is lost by restricting $\Sigma$ to be a family of {\em connected} closed $(G,\M)$-manifolds.  From an abstract point of view, the assignment $\Sigma\mapsto T(\Sigma)^{\spcheck}$ defines a sheaf of topological vector spaces over the moduli stack of closed connected $(G,\M)$-manifolds, and $\Sigma\mapsto E(\Sigma)$ is a section of this sheaf. 

\begin{defn}\label{def:part_funct_twisted}
The section $Z_E$ described above is the {\em partition function} of the twisted field theory $E$.
\end{defn}

\subsection{Segal's weakly conformal field theories as twisted field theories}
\label{subsec:weakly_conformal}
The goal of this subsection is to show that Segal's modular functors and weakly conformal field theory can be interpreted as twists respectively twisted field theories. More precisely, special types of twists $T\colon\CB{2}\to \TA$, whose domain is the $2$\nb-dimensional conformal bordism category, can be identified with {\em modular functors} in the sense of Segal \cite[Definition 5.1]{Se2}. Moreover, $T$\nb-twisted conformal field theories then become  {\em weakly conformal field theories} as in \cite[Definition 5.2]{Se2}. The translation works as follows.

Let $T\colon\CB{2}\to \TA$ be functor  such that the algebra $T(S^1)$ associated to the circle is a finite direct sum of copies of $\C$:
\[
T(S^1)=\bigoplus_{\phi\in \Phi}\C_\phi
\]
Here the subscript $\phi$ is just a book keeping tool to distinguish the various copies. Let us further assume that for every $2$\nb-dimensional conformal bordism $\Sigma$ the bimodule $T(\Sigma)$ is finite dimensional. If $\Sigma$ is a Riemannian surface with parametrized boundary circles, we can think of $\Sigma$ as a bordism from $\emptyset$ to the disjoint union of $k$ copies of $S^1$, and hence $T(\Sigma)$ is a left-module over $T(\p \Sigma)\cong T(S^1)\otimes\dots\otimes T(S^1)$. We note that a $T(S^1)$\nb-module can be thought of as $\Phi$\nb-graded vector space. Hence $T(\Sigma)$ is a $\Phi\times\dots\times\Phi$\nb-graded vector space. In particular, $T(\Sigma)$ can be decomposed in a direct sum of finite dimensional vector spaces; the summands are parametrized by the various ways of assigning a label in $\Phi$ to each boundary circle. These are the data of a modular functor in the sense of Segal \cite[Definition 5.1]{Se2}. Properties (i) and (ii) in Segal's definition follow from the fact that $\Sigma\mapsto T(\Sigma)$ is a monoidal functor. This was worked out in Hessel Posthuma's thesis. Segal requires further that $\dim T(S^2)=1$ and that $T(\Sigma)$ depends holomorphically on $\Sigma$. The second condition can be implemented within our framework by working over the site of complex manifolds (instead of smooth manifolds).  

Let $E$ be a $T$\nb-twisted conformal field theory. From our discussion above we see that $E$ assigns to $Y=S^1\amalg\dots\amalg S^1$ a left $T(Y)$\nb-module $E(Y)$. This vector space decomposes as direct sum of vector spaces, parametrized by ways to label each boundary circle by an element of $\Phi$. To a Riemann surface $\Sigma$ with boundary $\p\Sigma$, viewed as a bordism from $\emptyset$ to $\p\Sigma$, it assigns according to  \eqref{eq:nat_trans_data} a homomorphism
\[
T(\Sigma)\ra \Hom(E(\emptyset),E(\p \Sigma))=\Hom(\C,E(\p \Sigma))=E(\p \Sigma)
\]
which is $T(\p\Sigma)$\nb-equivariant, \ie compatible with the decomposition of these vector spaces according to ways of labeling the boundary circles. These are the data in Segal's definition of a weakly conformal field theory \cite[Definition 5.2]{Se2}, and again it isn't hard to show that Segal's conditions on these data follow from the properties of a twisted field theory.

\subsection{Field theories of degree $n$}
\label{subsec:degree}
\begin{defn} \label{def:prelim_deg}{\bf (Preliminary!)}
A Euclidean field theory of dimension $d|0$ and degree $n\in \Z$ is a natural transformation, internal to  $\Sym(\Cat^\fl)$,
\[
\xymatrix{
\EB {d|0}_{\pt}\ar@/^2pc/[rr]^{T^0}_{}="1"\ar@/_2pc/[rr]_{T^n}^{}="2"&& \TA_{\pt} \ar@{=>}^E"1";"2"
}
\]
Here $T^0$ is the constant twist (see \ref{def:twistedEFT}), and $T^n$ is a twist we will construct below. 
\end{defn}

We recall from \eqref{eq:nat_trans_data} that  the natural transformation $E$ in particular provides us with the following data:
\begin{itemize}
\item a left $T^n(Y)$\nb-module  $E(Y)$ for every object $Y\in (\EB{d|0}_{\pt})_0$;
\item a $T^n(Y_1)-T^n(Y_0)$\nb-bimodule map 
\[
E(\Sigma)\colon T(\Sigma)\ra \Hom(E(Y_0),E(Y_1))
\]
 for every object $\Sigma\in (\EB{d|0}_{\pt})_1(Y_0,Y_1)$.
\end{itemize}

\begin{rem}\label{rem:prelim} The definition above is preliminary in the same sense that Definition \ref{def:prelim} was preliminary: we should be {\em working in families}, \ie replace the internal categories $\EB {d|0}_{\pt}$ and $\TA_{\pt}$ by their (super) family versions $\EB{d|0}$, $\TA$ which are categories internal to $\Sym(\Cat^{\fl}/\csM)$ (see Def.\  \ref{def:TA} and Def.\ \ref{def:susyGMFT}). We expect to construct functors $\wh T^n\colon\EB{d|\delta}\to \TA$ of categories internal to $\Sym(\Cat^{\fl}/\csM)$ for $\delta=0,1$ such that the composition with the inclusion functors $\EB{d|\delta}_{\pt}\to \EB{d|\delta}$
is the functor $T^n$ constructed in this section.  Then $\EFT{d|\delta}$'s of degree $n$ for $\delta=0,1$ should be defined as a $\wh T^n$\nb-twisted Euclidean field theories of dimension $d|\delta$. For $d=0$, $\delta=1$, we gave an ad hoc construction of the twist \cite[Def.\ 3]{HKST}. The details of the construction of $\wh T^n$ for $d>0$, $\delta=0,1$ still have to be worked out.
\end{rem}

We construct the twist $T^n$ in terms of the `basic' twists  $T^{-1}$ and $T^{1}$ by defining
\[
T^n:=
\begin{cases}
(T^1)^{\otimes |n|}&n\ge 0\\
(T^{-1})^{\otimes |n|}&n\le 0
\end{cases}
\]
Here the tensor product $S\otimes T$ of two functors $S,T\colon \EB{d|0}_{\pt}\to \TA_{\pt}$ is defined using the symmetric monoidal structure of the categories 
$(\TA_{\pt})_i$ for $i=0,1$. This in turn is given by 
the (projective) tensor product (over $\C$) of algebras respectively bimodules. Below we will describe the construction of $T^{-1}$. The functor $T^1$ is  dual to $T^{-1}$ in the sense that $T^1\otimes T^{-1}$ is (lax) equivalent to the `trivial' twist $T^0$. In particular, if $\Sigma$ is a closed Euclidean spin $d$\nb-manifold, the complex line $T^1(\Sigma)$ is dual to $T^{-1}(\Sigma)$. 
 
For $Y$, $\Sigma$ as above, the $\E^{d|0}$\nb-structure determines in particular a spin-structure and a Riemann metric on these manifolds (see Example \ref{ex:euclidean_spin}), and hence a {\em Dirac operator}. We recall that the construction of the $\Z/2$\nb-graded spinor bundle $S=S^+\oplus S^-$ on a spin $d$\nb-manifold $\Sigma$ involves the choice of a $\Z/2$\nb-graded module $N=N^{ev}\oplus N^{odd}$ over the Clifford algebra $C\ell(\R^d)$. Given $N$, the spinor bundle $S$ is defined as the associated bundle $S=Spin(\Sigma)\times_{Spin(d)}N$, where $Spin(\Sigma)\to \Sigma$ is the principal $Spin(d)$\nb-bundle determined by the spin structure on $\Sigma$. In \cite{ST1} we chose $N=C\ell(\R^d)$ as a left-module over itself (leading to the `Clifford linear Dirac operator') while here we choose $N=\C\ell(\R^d)\otimes_{\C\ell(\R^d)^{ev}}\Delta^{\spcheck}$, where $\Delta$ is the module needed as a datum in the construction of the super Euclidean group in section \ref{subsec:superEuclidean}. 

\begin{enumerate}
\item $T^{-1}(Y)$ is the Clifford algebra generated by $C^\infty(Y^c,S^+_{|Y^c})$, the space of sections of the spinor bundle restricted to $Y^c$, equipped with the symmetric complex bilinear form that
gives the boundary term of Green's formula;
\item If $\Sigma$ is a closed manifold, then $T^{-1}(\Sigma)$ is the complex line $\Lambda^{top}(\cH^+(\Sigma))^{\spcheck}$. Here $\cH^+(\Sigma)$ is the finite dimensional space of section of $S^+$ which are {\em harmonic} (\ie in the kernel of the Dirac operator), and $\Lambda^{top}(V):=\Lambda^{\dim V}(V)$. 
\item If $\Sigma$ is a bordism from $Y_0$ to $Y_1$ without closed components, then $T^{-1}(Y)$ is the Fockspace, a $T^{-1}(Y_1)-T^{-1}(Y_0)$\nb-bimodule over these Clifford algebras determined by the Lagrangian subspace of $C^\infty(\p\Sigma,S^+_{|\p\Sigma})$ given by the boundary values of harmonic spinors.
\end{enumerate}
\begin{rem} 
For $d=1,2$, the algebra $T^{-1}(Y)$ associated to an object $Y\in \EB{d|0}_{\pt}$ is the Clifford algebra denoted $C(Y)$ in our earlier paper \cite[Def.\ 2.20]{ST1}. Similarly, the bimodule $T^{-1}(\Sigma)$ associated to a Euclidean spin bordism $\Sigma\in (\EB{d|0}_{\pt})_1$ is the Fock space module $F(\Sigma)$ of \cite[Def.\ 2.23]{ST1}. In \cite{ST1} we only considered these for $d=1,2$ as a side-effect of considering the Clifford-linear Dirac operator which makes the spaces $C^\infty(Y^c,S^+_{|Y^c})$ and $\cH^+(\Sigma)$ right modules over $C\ell(\R^{d-1})$. The Clifford algebra $C\ell(\R^{d-1})$ is $\R$ for $d=1$, $\C$ for $d=2$, but it is a non-commutative algebra for $d>2$. This confusing additional structure kept us from us using these data as input to the Clifford-algebra/Fockspace machine. With our current choice of $N$, the spinor bundle is just a complex vector bundle and we obtain a functor $T^{-1}\colon \EB{d|0}\to \TA$ for any $d$. 

Note also that the field theories of degree in \cite{ST1} are related to field theories of degree $-n$ in our current definition. 

This seems an appropriate place to point out an error in our formula (2.3) derived from Green's formula for the Dirac operator $D$. It should read
\[
\bra{D\psi e_1,\phi}+\overline{\bra{\psi,D\phi e_1}}=\bra{c(\nu)\psi_{|}e_1,\phi_{|}}
\]
\ie the second term should involve a complex conjugation which is missing in the formula as stated in \cite{ST1}. This conjugation is necessary to make all terms $\C$\nb-linear in $\phi$, and $\C$\nb-anti-linear in $\psi$ (our hermitian pairings are anti-linear in the first, and linear in the second slot).
\end{rem}

\begin{rem}\label{rem:dual} The reader might wonder about the various duals occurring in the above construction. 
Concerning the dual $\Delta^{\spcheck}$ in the construction of the spinor bundle, we note that the functions on our `model space' $\R^d\times\Pi \Delta$ for Euclidean structures are $C^\infty(\R^d)\otimes \Lambda^*(\Delta^{\spcheck})$. In particular, using $\Delta^{\spcheck}$ rather than $\Delta$ for the construction of the spinor bundle allows us to interpret spinors on a Euclidean $d|0$\nb-manifold $\Sigma$ as  {\em odd functions} on  the associated Euclidean supermanifold $\cS(\Sigma)$ of dimension $d|\delta$. This is a first step towards constructing the functor $\wh T^{-1}$ out of the Euclidean bordism category $\EB{d|\delta}$. 
\end{rem}

Specializing Definition \ref{def:part_funct_twisted} of the partition function of a $T$\nb-twisted field theory, we see that the partition function of a  Euclidean field theories of dimension $2|0$ and degree $n$ is a section $Z_E$ of a line bundle over the moduli stack of pointed Euclidean spin tori of dimension $2$. Explicitly, this stack is the quotient
\[
\fM=SL_2(\Z)\doublebackslash \left(\R_+\times(\fh^{++}\amalg  \fh^{-+}\amalg  \fh^{+-}\amalg  \fh^{--})\right),
\]
with the Euclidean torus $T_{\ell,\tau}^{s_1s_2}$ corresponding to $(\ell,\tau)\in \R_+\times\fh^{s_1s_2}$; here the superscripts $s_1,s_2\in \{+,-\}$ specify  the spin structure on the torus as explained in the proof of Proposition \ref{prop:partition_spin_2EFT}. The  group $SL_2(\Z)$  acts on $\R_+\times\fh$ by formula \ref{eq:SL_action}; in addition, it permutes the four copies of the upper half plane by 
$(A,\left(\begin{smallmatrix} s_1\\s_2\end{smallmatrix}\right))\mapsto A\left(\begin{smallmatrix} s_1\\s_2\end{smallmatrix}\right)$ 
for $A\in SL_2(\Z)$ and using the group isomorphism $\{+,-\} = \{+1,-1\}=\Z/2$. In particular, $\fh^{++}$ is fixed and the other three copies are permuted cyclically. The line bundle over this stack is the $SL_2(\Z)$\nb-equivariant line bundle over $\R_+\times\coprod_{s_1s_2}\fh^{s_1s_2}$ whose fiber over $(\ell,\tau)\in \R_+\times\fh^{s_1s_2}$ is the line $T^{-1}(T_{\ell,\tau}^{s_1s_2})^{\spcheck}$. The partition function $Z_E$ is the equivariant   section of this line bundle given by 
\[
Z_E(\ell,\tau)=\left(E(T_{\ell,\tau}^{s_1s_2})\colon T^n(T_{\ell,\tau}^{s_1s_2})\ra \Hom(T^n(\emptyset),T^n(\emptyset))=\Hom(\C,\C)=\C\right)
\]
for $\tau\in \fh^{s_1s_2}$. 

We note that the Weitzenb\"ock formula \cite[Ch.\ II, Thm.\ 8.8]{LM} implies that harmonic spinors on any compact Euclidean manifold are {\em parallel} (\ie their covariant derivative vanishes). In particular, the space of harmonic spinors on the torus $T^{s_1s_2}_{\ell,\tau}$ can be identified with the subspace of the parallel spinors on the universal covering $\E^2$ which is invariant under the action of the covering group $\Z^2$. Spinors on $\E^2$ are by construction functions on $\E^2$ with values in $\Delta^{\spcheck}$ (see Remark \ref{rem:dual}). Hence the space of parallel spinors is the space of constant functions which can be identified with $\Delta^{\spcheck}$. For $(s_1,s_2)\ne (+,+)$ the $\Z^2$\nb-action on the parallel spinors is non-trivial, and hence $T^{-1}(T_{\ell,\tau}^{s_1s_2})=\Lambda^{top}(\cH^+(T_{\ell,\tau}^{s_1s_2}))^{\spcheck}$ can be canonically identified with $\C$. For $(s_1,s_2)=(+,+)$, the group $\Z^2$ acts trivially on parallel spinors and hence $T^{-1}(T_{\ell,\tau}^{++})$ can be identified with $\Delta$.

If $E$ is a {\em conformal} Euclidean field theory of dimension $2|0$, then its partition function $Z_E$ becomes independent of $\ell$ in the sense that it is invariant under a natural $\R_+$\nb-action on the line bundle over the moduli stack $\fM$ (compare Remark \ref{rem:ell_independence} for field theories of degree zero). Then the section $Z_E$ amounts to 
\begin{enumerate}
\item a function on $\fh^{-+}$ which is invariant under the index three subgroup $\Gamma_0(2)$ of $SL_2(\Z)$. This subgroup fixes the spin structure $(-,+)$ and consists of matrices $A=\left(\begin{smallmatrix} a&b\\c&d\end{smallmatrix}\right)$ with $c\equiv 0\mod 2$. 
\item a function on $\fh^{++}$ with the transformation properties of a modular form of weight $-n/2$ (using a suitable choice of trivialization of the line bundle over $\fh^{++}$).
\end{enumerate}

\begin{rem}\label{rem:comparison_Segal} A $2$\nb-dimensional conformal field theory of central charge $c$ in the sense of Segal \cite{Se2} (more precisely, a {\em spin} CFT, see \cite{Kr}) gives a conformal Euclidean field theory of dimension $2|0$ and degree $n=c$. 
Its  partition function as defined by Segal  is a function on $\fh$ \cite[\S 6]{Se2} (see also \cite{Kr}). Based on Segal's discussion at the end of \S 6, it is not hard to show that his partition function is obtained from the function in (2) above by multiplying by $\eta(\tau)^{n/2}$, where $\eta(\tau)$ is Dedekind's $\eta$\nb-function. 
\end{rem}

\section{A periodicity theorem}
\label{sec:periodicity}
In this section we prove Theorem~\ref{thm:periodicity}. A field theory $P\in \EFT{d|0}^k(X)$ is a natural transformation
\[
\xymatrix{
\EB {d|0}(X) \ar@/^2pc/[rr]^{T^0}_{}="1"\ar@/_2pc/[rr]_{T^k}^{}="2"&& \TA \ar@{=>}^P"1";"2"
}
\]
Then multiplication with $P$ gives a functor
\[
P\otimes\colon \EFT{d|0}^n(X)\ra\EFT{d|0}^{n+k}(X)
\]
We note that if $P$ is invertible in the sense that there is a natural transformation $P^{-1}\colon T^k\Rightarrow T$ such that $P\circ P^{-1}$ and $P^{-1}\circ P$ are equivalent to the identity natural transformation, then $P\otimes$ is an isomorphism, with inverse given by multiplication by the Euclidean field theory 
\[
\xymatrix@1{
T^0\cong T^{-k}\otimes T^k\ar[rr]^{\id_{T^{-k}}\otimes P^{-1}}&&
T^{-k}\otimes T^0\cong T^{-k}
}
\]
of degree $-k$. 

We recall from Equation \eqref{eq:nat_trans_data} that $P$ gives a left $T^k(Y)$\nb-module $P(Y)$ for every object $Y\in\EB{d|0}_0$ and a bimodule map $P(\Sigma)\colon T^k(\Sigma)\to \Hom(P(Y_0),P(Y_1))$ for every Euclidean bordism $\Sigma$ of dimension $d|0$ from $Y_0$ to $Y_1$. 

To prove the periodicity statement, we will construct an invertible conformal Euclidean field theory $P$ of degree $-48$. The first step is an explicit description of the algebra $T^{-1}(K^s)$ and the bimodule $T^{-1}(C_\tau^s)$. We recall that $K^s\in \CEB{2|0}_0$ is the circle and $C_\tau^s\in \CEB{2|0}_1(K^s,K^s)$ is the cylinder with parameter $\tau\in \fh$. Here the superscript $s$ specifies the spin structure: $s=+$ is the periodic spin structure, and $s=-$ is the anti-periodic spin structure. We recall that the group $\R$ acts on the spin manifolds $K^s$ and $C^s_\tau$. Ignoring the spin-structures, $\tau\in \R$ acts  by rotation by $q:=e^{2\pi i\tau}$ on the circle $S^1$ respectively the cylinder $C_\tau$. We note that $1\in \R$ acts trivially on the spin manifolds $K^+$ and $C_\tau^+$, but by multiplication by $-1$ on the spinor bundle for $K^-$ and $C_\tau^-$, thus leading to an effective action of $\R/\Z$ for $K^+$, $C_\tau^+$ and an effective action of $\R/2\Z$ for $K^-$,  $C_\tau^-$. The $\R$\nb-action on $K^s$ as an object of $\CEB{2|0}_0$ induces by functoriality an $\R$\nb-action on the algebra $T^{-1}(K^s)$. We have the following $\R$\nb-equivariant algebra isomorphisms:
\begin{equation}\label{eq:decom_alg}
T^{-1}(K^+)\cong \Cl(\C_0)\otimes\bigotimes_{m\in \N} \Cl(H(\C_m))
\qquad
T^{-1}(K^-)\cong \bigotimes_{m\in \N_0+\frac 12} \Cl(H(\C_m))
\end{equation}
Here \begin{enumerate}
\item For a complex vector space $W$ equipped with a $\C$\nb-bilinear symmetric form $\omega\colon W\times W\to\C$, $\Cl(W)$ is the complex Clifford algebra generated by $W$;
\item $H(V)$ for a complex vector space $V$ is the hyperbolic form on $H(V):=V^{\spcheck}\oplus V$ defined by $\omega((f,v),(f',v'))=f(v')+f'(v)$;
\item $\C_m$ for $m\in \R$ is a copy of the complex numbers equipped with an $\R$\nb-action given by scalar multiplication by $q^m:=e^{2\pi i m\tau}$ for $\tau\in \R$. The Clifford algebra $\Cl(\C_0)$ is with respect to the standard form $\omega(z,z')=zz'$ on $\C_0$. The $\R$\nb-action on $H(\C_m)$ respectively $\C_0$ is by isometries and hence induces an action on the Clifford algebra it generates.
\item the tensor products are the {\em restricted} tensor product (\ie the closure of finite sums of tensor products $\bigotimes_m a_m$ where $a_m=1$ for all but finitely many $m$'s). 
\end{enumerate}

The structure of the bimodule $T^{-1}(C_\tau^s)$ is determined by the following isomorphism of $T^{-1}(K^s)-T^{-1}(K^s)$\nb-bimodules with $\R$\nb-action:
\[
T^{-1}(C_\tau^s)\cong T^{-1}(K^s)_{\phi(\tau)}
\]
Here $T^{-1}(K^s)$ is considered as a bimodule over itself, but with the right\nb-action twisted by the algebra homomorphism $\phi(\tau)\colon T^{-1}(C_\tau^s)\to T^{-1}(C_\tau^s)$ (\ie the right action of $a\in T^{-1}(K^s)$ on $m\in T^{-1}(K^s)$ is given by the product $m\cdot (\phi(\tau)(a))$ in the algebra $T^{-1}(K^s)$).

Moreover, $\phi(\tau)=\bigotimes_m\phi_m(\tau)$, where $\phi_m(\tau)$ is an algebra automorphism of $\Cl(H(\C_m))$ (respectively $\Cl(\C_0)$ for $m=0$). This automorphism is induced by the action of $\tau\in \fh\subset\C$ on $\C_m$ given by scalar multiplication by $e^{2\pi im\tau}$ (so it extends the $\R$\nb-action defined above).

Our goal now is to construct an invertible field theory of $P$ of degree $-n$, \ie an invertible natural transformation $P\colon T^0\Rightarrow (T^{-1})^{\otimes n}$ for $n$ as small as possible. We note that such a $P$ gives in particular a Morita-equivalence between the algebras $T^{0}(K^s)=\C$ and $(T^{-1})^n(K^s)$. Since the algebras $C\ell(H(\C_m))$ and $C\ell(\C_0)^{\otimes 2}$ are Morita-equivalent to $\C$, but $C\ell(\C_0)$ is not, we need to assume that $n$ is even.

To construct $P$, according to \eqref{eq:nat_trans_data}, we need to construct in particular a left $(T^{-1}(K^s))^{\otimes n}$\nb-module $P(K^s)$ and a $T^{-1}(K^s)^{\otimes n}-T^{-1}(K^s)^{\otimes n}$\nb-bimodule map
\begin{equation*}
P(C_\tau^s)\colon (T^{-1}(C_\tau^s))^{\otimes n}\ra \End(P(K^s))
\end{equation*}
We define $P(K^s)$ to be the following left module over $T^{-n}(K^s)=(T^{-1}(K^s))^{\otimes n}$:
\begin{equation}
P(K^+):=M_0^{\otimes \frac n2}\otimes\bigotimes_{m\in \N}M_m^{\otimes n}
\qquad
P(K^-):=\bigotimes_{m\in \frac 12+\N_0}M_m^{\otimes n}
\end{equation}
Here $M_m$ (respectively $M_0$) is an irreducible graded module over $\Cl(H(\C_m))$ (respectively $\Cl(\C_0)^{\otimes 2}$), and we let each factor from the tensor product decomposition \eqref{eq:decom_alg} act on the corresponding factor in the tensor product above. The modules $M_m$ provide a Morita equivalence between $\C$ and $\Cl(H(\C_m))$ for $m\ne 0$ and between $\C$  and $\Cl(\C_0)^{\otimes 2}$ for $m=0$.  It follows that $P(K^s)$ provides a Morita equivalence between $\C$ and $T^{n}(K^s)$  which is a necessary feature to insure {\em invertibility} of $P$ as a natural transformation from $T^0$ to $T^{n}$. We note that there are {\em two} non-isomorphic irreducible graded modules over $\Cl(H(\C_m))$ and $\Cl(\C_0)$. We will specify our choice for $M_m$ below. 

We note that any bimodule map 
\[
T^{-n}(C_\tau^s)=(T^{-1}(K^s)_{\phi(\tau)})^{\otimes n}\to \End(P(K^s))
\]
 is determined by the image of the unit $1\in (T^{-1}(K^s)_{\phi(\tau)})^{\otimes n}$. Our map $P(C_\tau^{s})$ is determined by specifying 
\begin{alignat*}{3}
P(C_\tau^+)(1)&=q^2\id_{M_0^{\otimes \frac n2}}\otimes\bigotimes _{m\in \N}b_m^{\otimes n}(\tau)\ &\in &&\ \End(P(K^+))&=\End\left(M_0^{\otimes \frac n2}\otimes\bigotimes_{m\in \N}M_m^{\otimes n}\right)\\
P(C_\tau^-)(1)&=q\bigotimes _{m\in \frac 12+\N_0}b_m^{\otimes n}(\tau)
&\in &&\ \End(P(K^-))&=\End\left(\bigotimes_{m\in \frac 12+\N_0}M_m^{\otimes n}\right)
\end{alignat*}
Here $b_m(\tau)$ is an element of the Clifford algebra $\Cl(H(\C_m))$ which acts by left-multiplication on the module $M_m$. It is defined by 
\[
b_m(\tau):=1+(1-q^m)\frac {ef}2\in \Cl(H(\C_m))
\]
where $\{e,f\}$ is a basis of $H(\C_m)$ given by $e:=1\in \C_m$, and letting $f\in \C_{-m}^{\spcheck}$ be the dual element (in particular $\omega(e,e)=\omega(f,f)=0$, $\omega(e,f)=1$).  It is a straightforward calculation to show that our prescription makes $P(C_\tau^s)$ a well-defined map of bimodules (we need to check that the annihilator of $1\in T^{-n}(K^s)$ agrees with the annihilator of its desired image element $P(C_\tau^s)(1)$). It is also not hard to show that $\fh\to \End(P(K^s))$, $\tau\mapsto P(C_\tau^s)(1)$ is a homomorphism; this shows that our definition of $P(C_\tau^s)$ is compatible with composing two cylinders (in more technical terms, diagram \eqref{eq:gluing} commutes if $\Sigma$, $\Sigma'$ are two cylinders). 

So far we have not constructed the bimodule maps $P(\Sigma)$ for the connected bordisms $R^s_\tau$, $L^s_\tau$, and $T^{s_1s_2}_\tau$. It turns out that these are determined by $P(C_\tau^s)$. For example, an element $\gamma\in T^{-n}(C_\tau^{s_2})$ determines an associated element $\delta\in  T^{-n}(T^{s_1s_2}_\tau)$ such that 
\[
P(T^{s_1s_2}_\tau)(\delta)=
\begin{cases}
\str P(C_\tau^{s_2})(\gamma)&s_1=+\\
\tr P(C_\tau^{s_2})(\gamma)&s_1=-
\end{cases}
\]
This allows us to calculate the partition function $Z_P$ of this putative conformal Euclidean field theory. It turns out that to ensure existence of a conformal field theory $P$ with $P(C_\tau^s)$ as above, the {\em only} consistency condition that needs to be checked (in addition to additivity for gluing cylinders) is that $Z_P$ is an $SL_2(\Z)$\nb-equivariant section. We recall that $T^{s_1s_2}(T_\tau)=\C$ for $(s_1,s_2)\ne (+,+)$, and hence $Z_P$ is just a function on $\fh^{-+}\amalg \fh^{+-}\amalg\fh^{--}$. There is a trivialization of the line bundle restricted to  $\fh^{++}$ such that an $SL_2(\Z)$\nb-equivariant section corresponds to a function on the upper half plane with the equivariance properties of  a modular form of weight $n/2$. Using this trivialization, the  function $Z_P$ is as follows:
\begin{equation}
Z_P(\tau)=
\begin{cases}
\left(q^{1/24}\prod_{m\in \N}(1-q^m)\right)^n\\
\left(2^{1/2}q^{1/24}\prod_{m\in \N}(1+q^m)\right)^n\\
\left(q^{-1/48}\prod_{m\in \N_0+\frac 12}(1-q^m)\right)^n\\
\left(q^{-1/48}\prod_{m\in \N_0+\frac 12}(1+q^m)\right)^n
\end{cases}
=\begin{cases}
\eta(\tau)^n&\tau\in \fh^{++}\\
\left(\frac{2^{1/2}\eta(2\tau)}{\eta(\tau)}\right)^n&\tau\in \fh^{-+}\\
\left(\frac{\eta(\tau/2)}{\eta(\tau)}\right)^n& \tau\in \fh^{+-}\\
\text{something}&\tau\in \fh^{--}
\end{cases}
\end{equation}
It is well-known that $\eta(\tau)^n$ is a modular form \Iff $n$ is a multiple of $24$. However, the function $Z_P$ is {\em not} $SL_2(\Z)$\nb-equivariant for $n=24$: the matrix $T=\left(\begin{smallmatrix} 1&1\\0&1\end{smallmatrix}\right)$ maps $\tau\in \fh^{+-}$ to $\tau+1\in \fh^{--}$, but $Z_P(T\tau)\ne Z_P(\tau)$ due to the factor $(q^{-1/48})^{24}=q^{-1/2}$ which changes sign under $T$. This forces us to take the power $n=48$.



\end{document}